\documentclass{amsart}
\usepackage{amssymb}
\usepackage{mathrsfs}
\usepackage{stmaryrd}
\usepackage{bbm}
\usepackage{oldgerm}
\usepackage[english]{babel}
\usepackage[T1]{fontenc}
\usepackage[utf8x]{inputenc}
\usepackage[all]{xy}
\usepackage[colorlinks]{hyperref}
\usepackage{upref}
\usepackage{xcolor}
\usepackage{scrextend}
\usepackage{tikz-cd}

\colorlet{wine-stain}{red!50!black}
\colorlet{light-blue}{cyan!60!black}
\hypersetup{
					colorlinks=true, 
					linkcolor=wine-stain, 
					citecolor=cyan, 
					urlcolor=green, 
					linktoc=page}

\newtheorem{thm}[subsection]{Theorem}
\newtheorem{prop}[subsection]{Proposition}
\newtheorem{cor}[subsection]{Corollary}
\newtheorem{lem}[subsection]{Lemma}

\theoremstyle{definition}

\newtheorem{defi}[subsection]{Definition}
\newtheorem{rem}[subsection]{Remark}

\numberwithin{equation}{subsection}
\def\bA{\mathbb{A}}

\def\bF{\mathbb{F}}

\def\bK{\mathbb{K}}
\def\bL{\mathbb{L}}

\def\bQ{\mathbb{Q}}

\def\bZ{\mathbb{Z}}

\def\fD{\mathfrak{D}}

\def\fX{\mathfrak{X}}
\def\fY{\mathfrak{Y}}

\def\d{\mathfrak{d}}
\def\m{\mathfrak{m}}

\def\p{\mathfrak{p}}
\def\q{\mathfrak{q}}
\def\r{\mathfrak{r}}

\def\cC{\mathcal{C}}

\def\cF{\mathcal{F}}
\def\cG{\mathcal{G}}

\def\cI{\mathcal{I}}

\def\cO{\mathcal{O}}

\def\cU{\mathcal{U}}
\def\cV{\mathcal{V}}

\def\scrC{\mathscr{C}}
\def\scrD{\mathscr{D}}

\def\scrS{\mathscr{S}}

\def\oK{\overline{K}}

\def\ox{\overline{x}}
\def\oy{\overline{y}}

\def\hf{\widehat{f}}
\def\hg{\widehat{g}}

\DeclareMathOperator{\Spec}{Spec}
\DeclareMathOperator{\Spf}{Spf}
\DeclareMathOperator{\Sp}{Sp}
\DeclareMathOperator{\Hom}{Hom}

\DeclareMathOperator{\St}{St}
\DeclareMathOperator{\sw}{sw}
\DeclareMathOperator{\Ind}{Ind}
\DeclareMathOperator{\ord}{ord}
\DeclareMathOperator{\Aut}{Aut}
\DeclareMathOperator{\CC}{CC}
\DeclareMathOperator{\AS}{AS}
\DeclareMathOperator{\rk}{rk}

\title{Variation of the Swan conductor of an $\bF_{\ell}$-sheaf on a rigid annulus}
\author{Amadou Bah}
\address{Columbia University, Department of Mathematics, 2990 Broadway, New York, NY 10027, USA}
\email{bah@math.columbia.edu}

\date{}

\begin{document}

\begin{abstract}
Let $C=A(r, r')$ be a closed annulus of radii $r$ and $r'$ ($r < r' \in \bQ_{\geq 0}$) over a complete discrete valuation field with algebraically closed residue field of characteristic $p>0$. To an étale sheaf of $\bF_{\ell}$-modules $\cF$ on $C$, ramified at most at a finite set of rigid points of $C$, we associate an Abbes-Saito Swan conductor function $\sw_{\AS}(\cF, \cdot): [r, r']\cap \bQ_{\geq 0} \to \bQ$ which, for the variable $t$, measures the ramification of $\cF\lvert C^{[t]}$ - the restriction of $\cF$ to the sub-annulus $C^{[t]}$ of $C$ of radius $t$ with $0$-thickness - along the special fiber of the normalized integral model of $C^{[t]}$. We show that this function is continuous,  convex and piecewise linear outside the radii of the ramification points of $\cF$, with finitely many slopes which are all integers. For two distinct radii $t$ and $t'$ lying between consecutive radii of ramification points of $\cF$, we compute the difference of the slopes of $\sw_{\AS}(\cF, \cdot)$ at $t$ and $t'$ as the difference of the orders of the characteristic cycles of $\cF$ at $t$ and $t'$. 
\end{abstract}

\maketitle

\tableofcontents

\section{Introduction}\label{Introduction}
\subsection{}\label{VCSII-Prologue}
Let $K$ be a complete discrete valuation field, $\cO_K$ its valuation ring and $\pi$ a uniformizer of $\cO_K$. 
Let also $\oK$ be a separable closure of $K$ and $v_K:\oK^{\times}\to \bQ$ the valuation of $\oK$ normalized by $v(\pi)=1$.
Assume that the residue field $k$ of $\cO_K$ is \textit{algebraically closed} of characteristic $p>0$. Denote by $D$ the closed rigid unit disc over $K$. 
Let $0\leq r < r'$ be rational numbers and denote by $C=A(r', r)$ the closed sub-annulus of $D$ of radii $\lvert \pi\lvert^{r'} < \lvert\pi\lvert^r$.
For a rational number $r\leq t\leq r'$, we denote by $C^{[t]}$ the sub-annulus of $C$ of radius $t$ with $0$-thickness.
Let $\Lambda$ be a finite field of characteristic $\ell\neq p$. An étale sheaf of $\Lambda$-modules $\cF$ on $C$ is said to be \textit{meromorphic} if it is lisse, i.e. locally constant and constructible, on the complement of a finite set of rigid points of $C$.
In this paper, we associate to such an $\cF$ a Swan conductor function $\sw_{\AS}(\cF, \cdot): [r, r']\cap \bQ_{\geq 0} \to \bQ$, defined by the (logarithmic) ramification theory of Abbes and Saito, and which, for the variable $t$, measures the ramification of $\cF\lvert C^{[t]}$ along the special fiber of the normalized integral model of $C^{[t]}$. We show that this function is continuous and piecewise linear outside the radii of the ramification points of $\cF$, with finitely many slopes which are all integers. For two distinct radii $t$ and $t'$ lying between consecutive radii of ramification points of $\cF$, we compute the difference of the slopes of $\sw_{\AS}(\cF, \cdot)$ at $t$ and $t'$ as the difference of the orders of the characteristic cycles of $\cF$ at $t$ and $t'$. 

\subsection{}\label{Intro-Faisceau-Meromorphe}
Let $\psi : \bF_p\to \Lambda^{\times}$ be a nontrivial character.
Let $r \leq t \leq r'$ be a rational number distinct from the radii of the ramification points of $\cF$. Then, the restriction $\cF\lvert C^{[t]}$ corresponds to a connected Galois étale cover $f^{[t]}: X^{[t]} \to C^{[t]}$ and a continuous finite dimensional $\Lambda$-representation $\rho_{\cF}[t]$ of $G^{[t]}=\Aut(X^{[t]}/C^{[t]})$ \cite[2.10]{deJong}. Let $\fX_{K'}^{[t]}\to \scrC_{K'}^{[t]}$ be the normalized integral model of $f ^{[t]}$ over $\cO_{K'}$, for some large finite extension $K'$ of $K$ \eqref{Models Formels}, $\overline{\p}^{(t)}$ a geometric generic point of the special fiber $\scrC_{s'}^{[t]}$ of $\scrC_{K'}^{[t]}$ and $\fX_{K'}^{[t]}(\overline{\p}^{(t)})$ the set of geometric generic points of the special fiber $\fX_{s'}^{[t]}$ above $\overline{\p}^{(t)}$.
The canonical right action of $G^{[t]}$ on $X^{[t]}$ induces a transitive action of $G^{[t]}$ on $\fX_{K'}^{[t]}(\overline{\p}^{(t)})$.
The stabilizer of any $\overline{\q}^{(t)}\in \fX_{K'}^{[t]}(\overline{\p}^{(t)})$ is isomorphic to the Galois group $G_{\overline{\q}^{(t)}}$ of a finite Galois extension of henselian discrete valuation fields \eqref{Artin-Intrinseque}.
Therefore, the ramification theory of Abbes and Saito \cite{A.S.1, A.S.3} applies to the restriction $M_{\overline{\q}^{(t)}}=\rho_{\cF}[t]\vert G_{\overline{\q}^{(t)}}$; it yields the Swan conductor
\begin{equation}
	\label{Intro-Faisceau-Meromorphe 1}
\sw_{\AS}(\cF, t)=\sw_{G_{\overline{\q}^{(t)}}}^{\AS}(M_{\overline{\q}^{(t)}}) \quad \in \bQ
\end{equation}
and the characteristic cycle $\CC_{\psi}(M_{\overline{\q}^{(t)}})$ of $M_{\overline{\q}^{(t)}}$, which are independent of the choice of both $K'$ and $\overline{\q}^{(t)} \in \fX_{K'}^{[t]}(\overline{\p}^{(t)})$ \eqref{SW&CC}.
In our setting, the characteristic cycle lies in $(\Omega^1_{\kappa(\overline{\p}^{(t)})})^{\otimes m_t}$, where $m_t= \dim_{\Lambda}(M_{\overline{\q}^{(t)}}/(M_{\overline{\q}^{(t)}})^{(0)})$, with $(M_{\overline{\q}^{(t)}})^{(0)}$ being the part of $M_{\overline{\q}^{(t)}}$ fixed by the tame inertia subgroup of $G_{\overline{\q}^{(t)}}$, and $\kappa(\overline{\p}^{(t)})$ coincides with the field $\cO_{\scrC_{s'}^{[t]}, \overline{\p}^{(t)}}$. The latter has a normalized discrete valuation map $\ord_{\overline{\p}^{(t)}} : \kappa(\overline{\p}^{(t)})^{\times}\to \bZ$ defined by $\ord_{\overline{\p}^{(t)}}(\xi)=1$; it extends uniquely to $(\Omega^1_{\kappa(\overline{\p}^{(t)})})^{\otimes m_t}$. We put
\begin{equation}
	\label{Intro-Faisceau-Meromorphe 2}
\varphi_s(\cF, t)= - \ord_{\overline{\p}^{(t)}}(\CC_{\psi}(M_{\overline{\q}^{(t)}})) - m_t.
\end{equation}
This integer is independent not only of $\overline{\q}^{(t)} \in \fX_{K'}^{[t]}(\overline{\p}^{(t)})$, but also of the choice of both $\psi$ and a uniformizer of $\cO_{K'}$ \cite[11.6]{Hu}.

We note that definitions \eqref{Intro-Faisceau-Meromorphe 1} and \eqref{Intro-Faisceau-Meromorphe 2} extend to the radii of the ramification points of $\cF$ : one just replaces $C^{[t]}$ by the complement of the ramification points of radius $t$ and make the same constructions again, with the same arguments. Our main result takes the following general form.

\begin{thm}
	\label{Intro-Thm-Faisceau-Meromorphe}
Let $\cF$ be a meromorphic étale sheaf of $\Lambda$-modules on the annulus $C$ as above and $\{r_1 > \ldots > r_n\}$ the ordered set of the radii of its ramification points.
Then, the function 
\begin{equation}
	\label{Intro-Thm-Faisceau-Meromorphe 1}
\sw_{\AS}(\cF, \cdot) : [r, r']\cap\bQ_{\geq 0} \to \bQ,
\end{equation}
is continuous,  convex and piecewise linear on $\bQ - \{r_1, \ldots, r_n\}$,  with finitely many slopes which are all integers.  Moreover, for rational numbers $r_i <  t < t' <r_{i-1}$, the difference of the right and left derivatives of $\sw_{\AS}(\cF, \cdot)$ at $t$ and $t'$ respectively is
\begin{equation}
	\label{Intro-Thm-Faisceau-Meromorphe 2}
\frac{d}{dt}\sw_{\AS}(\cF, t+) - \frac{d}{dt}\sw_{\AS}(\cF, t'-)=\varphi_s(\cF, t) - \varphi_s(\cF, t').
\end{equation}
\end{thm}

\noindent We note first that if $\cF$ is a lisse sheaf on the entire rigid unit disc $D$, we proved a more precise statement–\textit{minus the convexity}– in \cite[Theorem 1.9]{VCS},  namely that the function $\sw_{\AS}(\cF, \cdot)$, defined by the formula \eqref{Intro-Faisceau-Meromorphe 1} on the whole $\bQ_{\geq 0}$, is continuous and piecewise linear, with finitely many slopes which are all integers, and its right derivative \textit{is exactly} the function $\varphi_s(\cF, \cdot) : \bQ_{\geq 0}\to \bQ$,  defined by the formula \eqref{Intro-Faisceau-Meromorphe 2}.  This may still be true in the setting of \ref{Intro-Faisceau-Meromorphe}, but our techniques are unlikely to yield it. Instead,  what we get is the variation of the slope between two distinct radii.

In \textit{loc. cit.},  we also shared the expectation that the main theorem there \cite[Theorem 1.9]{VCS} should hold for an étale sheaf on $D$ (or $C$) with a finite number of ramification points and that the function $\sw_{\AS}(\cF, \cdot)$ should be convex as  well,  expressing the hope of tackling these questions soon. The present paper fulfills these expectations,  in the form of Theorem \ref{Intro-Faisceau-Meromorphe},  by both extending the techniques of \textit{loc. cit.} to handle morphisms to annuli and adding new cohomological ingredients. to the mix. \newline
In fact,  by excluding the radii of the ramification points of $\cF$, we see that we are reduced to treating the case of a lisse étale sheaf of $\Lambda$-modules on the annulus $C$, hence to studying étale morphisms to annuli.  

\subsection{}\label{Intro-Strategy-I}
Except for the convexity statement,  which we deal with somewhat separately with cohomological arguments,  the overall strategy of proof is the same as the one employed for in \cite{VCS}.  Therefore,  rather than repeating it here,  we refer the reader to the introduction of \textit{loc. cit.} for an overview and briefly indicate the most salient changes.  We then proceed to explain how the convexity property is established.

Instead of a lisse sheaf,  we more generally consider a smooth $K$-affinoid space $X$ and a finite flat morphism $f: X\to C=A(r', r)$ which is generically étale.  We extend the constructions and results in \cite[\S 4, 7]{VCS} to this setting.  Most notably,  we generalize the nearby cycles formula \cite[4.28]{VCS} linking the derivative of the discriminant function $\partial_f$ studied by W. Lütkebohmert \cite[\S 1, 2]{Lutke} with cohomological invariants associated to the local homomorphisms $\cO_{\scrC_{K'}, o}\to \cO_{\fX_{K'}, \ox}$, for the formal étale topology, induced by the normalized integral model $\hf: \fX_{K'}\to \scrC_{K'}$ of $f$ over $\cO_{K'}$, where the $\ox$'s are the geometric points of the special fiber $\fX_{s'}$ of $\fX_{K'}$ above the origin $o$ of $\scrC_{s'}$ \eqref{Nearby Cycles Lutke}. We do this under the assumption that $f$ is étale over $A(r, r)$ and $A(r', r')$ and that the inverse images by $f$ of $A(r, r)$ and $A(r', r')$ decomposes into finite disjoint unions of annuli
\begin{equation}
	\label{Intro-Strategy-I 1}
f^{-1}(A(r, r))=\coprod_i A(r/d_i, r/d_i) \quad {\rm and} \quad f^{-1}(A(r', r'))=\coprod_j A(r'/d'_j, r'/d'_j)
\end{equation}
respectively, which is later satisfied for all but a finite number of the radii of interest, thanks to the semi-stable reduction theorem \eqref{Decomposition SemiStable}. We have to compactify both sets of annuli, unlike in \cite[4.28]{VCS} where we have only one side.  Let $Y\to S'=\Spec(\cO_{K'})$ be the algebraization of the compactification of $\fX_{K'}$ obtained by replacing the annuli of the disjoint unions in \eqref{Intro-Strategy-I 1} with discs of the same radii (see \eqref{Nearby Cycles Lutke 6}); it is a normal and proper relative curve,  smooth outside the $\ox$'s, which features prominently in our stratefy of proof. 
Taking into account orientation issues,  the final nearby cycles formula exhibits the difference $\frac{d}{dt}\partial_f(r+) - \frac{d}{dt}\partial_f(r'-)$ as the (alternating) sum,  over the $\ox$'s,  of the dimensions of the stalks $R\Psi_{Y/S'}(\Lambda)_{\ox}$ (plus a sum of local numerical invariants for horizontal ramification which vanish when $f$ is étale) .
This difference between the inner and outer radii is a feature that carries over throughout all the text.\newline
From this formula and the non-vanishing of $R^1\Psi_{Y/S'}(\Lambda)_{\ox}$ \eqref{Saito-Vanishing}, we deduce the following additional property for Lütkebohmert's discriminant function.

\begin{prop}[{cf. \ref{Discriminant-Convex}}]
	\label{Intro-Discriminant-Convex}
Assume that $f : X\to C$ is étale. Then, the discriminant function $\partial_f$ is convex.
\end{prop} 

\subsection{}\label{Intro-Strategy-II}
Let $G$ be a finite group with a right action on $X$ such that $f$ is $G$-equivariant. Just like in \cite[\S 7]{VCS}, for a rational number $r\leq t\leq r'$, Kato's ramification theory for $\bZ^2$-valuation rings \cite[\S 6]{VCS} yields an Artin class function $\widetilde{a}_f(t)$ on $G$ \eqref{Artin-Intrinseque}. But, unlike in \textit{loc. cit.}, for the aforementioned reason, the appropriate Swan class function to consider is $\widetilde{\sw}_f^{\beta}([t, t'])$ for an interval $[t, t']$ such that $t\neq t'$, accounting for the two branches $A(t, t)$ and $A(t', t')$. (We don't use this notation in the body of the text, writing instead $\widetilde{\sw}_{f^{[t, t']}}^{\beta} + \widetilde{\sw}_{f^{[t, t']}}^{\prime\beta}$ \eqref{FunctionsArtinSwan}.)\newline
Let $r_G$ be the character of the regular representation of $G$ and $\langle \cdot, \cdot \rangle$ the usual pairing of class functions on $G$. Then, we have \eqref{Var-Reg}
\begin{equation}
	\label{Intro-Strategy-II 1}
\partial_f(t)=\langle\widetilde{a}_f(t), r_G \rangle,
\end{equation}
\begin{equation}
	\label{Intro-Strategy-II 2}
\frac{d}{dt}\partial_f(t+) - \frac{d}{dt}\partial_f(t'-)=\langle \widetilde{\sw}_f^{\beta}([t, t']),  r_G \rangle.
\end{equation}
The proof of \ref{Intro-Strategy-II 2} requires not only the nearby cycles formula \eqref{Nearby Cycles Lutke 2} but also a Hurwitz formula, due to Kato for local homomorphisms like $\cO_{\scrC_{K'}, o}\to \cO_{\fX_{K'}, \ox}$ above. For the analogous statement \cite[7.12]{VCS}, we could get by with a special case of this Kato-Hurwitz formula (by the regularity of the affine line !), as we considered only sub-discs of $D$. For annuli with thickness however, the more general formula \eqref{Kato-Formula 1} is needed.\newline
From the above identities and \ref{Intro-Discriminant-Convex}, we deduce that the function $t\mapsto \langle\widetilde{a}_f(t), r_G \rangle$ is continuous,  convex and piecewise linear with (finitely many) integer slopes whose variation from $t$ to $t'$ is given by $\langle \widetilde{\sw}_f^{\beta}([t, t']),  r_G \rangle$. \newline

\noindent The identities \eqref{Intro-Strategy-II 1} and \eqref{Intro-Strategy-II 2} hold in fact more generally with,  for a given subgroup $H\subset G$,  $f$ replaced by the quotient morphism $X/H\to C$ on the left-hand-side and $r_G$ replaced by (the character of) $\Ind_H^G 1_H$.  Then,  by induction arguments,  the aforementioned variational properties extend to the function $t\mapsto \langle\widetilde{a}_f(t),  \chi \rangle$, for any character $\chi$ of $G$,  \textit{except for the convexity}.\newline
For the latter property, the positivity of the pairings $\langle \widetilde{\sw}_f^{\beta}([t, t']),  \Ind_H^G 1_H \rangle$,  for all subgroups $H$, is not enough.  However,  $\widetilde{\sw}_f^{\beta}([t, t'])$ has a geometric realization which is projective,  a stronger conclusion which relies on the following general result.

\begin{prop}[{Proposition \ref{Projective-TorsionCoeffs},  Theorem \ref{Rationality-of-NearbyCycles}}]
	\label{Intro-Perfection-Projectivity-Rationality}
Let $Y\to S=\Spec(\cO_K)$ be a normal relative curve and $x$ a closed point of the special fiber $Y_s$ such that $Y-\{x\}$ is smooth over $S$.
Let $G$ be a finite group with an admissible right action on $Y$ through $S$-automorphisms such that the induced action on the generic fiber $Y_{\eta}$ is free,  and denote by $\St_x$ the stabilizer of $x$.  Let $\ell\neq p$ be a prime number and denote by $P_{\bQ_\ell}(\St_x)$ the Grothendieck group of finitely generated and projective $\bQ_\ell[\St_x]$-modules.
\begin{itemize}
\item[(i)] The stalk $R\Psi_{Y/S}(\bQ_\ell)_x$ defines a perfect complex of $\bQ_\ell[\St_x]$-modules of amplitude in $[0, 1]$.
\item[(ii)] The resulting class $[R\Psi_{Y/S}(\bQ_\ell)_x[1]]$ in $P_{\bQ_\ell}(\St_x)$ is the class of an actual finitely generated and projective $\bQ_\ell[\St_x]$-module.
\item[(iii)] The class $[R\Psi_{Y/S}(\bQ_\ell)_x]$ has rational coefficients. 
\end{itemize}
\end{prop}

\noindent (See \ref{Z_ell-Coeffs} and \ref{l-adic NearbyCycles} for the definition of $R\Psi_{Y/S}(\bQ_\ell)_x$). Thanks to \cite[13.1,  Théorème 30]{Serre2}, the combination of the nearby cycles formula,  (the general form of) the identity \eqref{Intro-Strategy-II 2} and \ref{Intro-Perfection-Projectivity-Rationality} (iii) yields the following, aforementioned, geometric realization.

\begin{thm}[{Théorème \ref{Realization-NearbyCycles-Swan}}]
	\label{Intro-Projective Realization}
Denoting again by $Y/S'$ the algebraization of the compactification of the normalized integral model of $f^{-1}(A(t', t))$ over $S'$, we have the identity
\begin{equation}
	\label{Intro-Projective Realization 1}
[\bigoplus_{x} R\Psi_{Y/S'}(\bQ_\ell)_x]=\widetilde{\sw}_f^{\beta}([t, t']).
\end{equation}
\end{thm}

\noindent Let us mention that less general formulations of the statements in \ref{Intro-Perfection-Projectivity-Rationality} appear in the work of Ramero (in a rigid geometry setting) \cite[3.2.15, 3.2.17]{Ramero}.  We think that these statements may be of independent interest and thus deserve a more general algebraic formulation.
The proofs,  which adapt classical arguments from \cite{SGA4, SGA5, SGA7},  also follow the template in \cite{Ramero}. 
\newline
An identity similar to \eqref{Intro-Projective Realization} is also proved by Ramero \cite[3.3.16]{Ramero} but with a different,  and somewhat ad-hoc,  ramification theory due to Huber. \newline

\noindent Theorem \eqref{Intro-Projective Realization} shows that the class function $\widetilde{\sw}_f^{\beta}([t, t'])$ is in fact the character of a projective $\bQ_\ell[G]$-module,  which yields the convexity of the function $t\to \langle\widetilde{a}_f(t),  \chi \rangle$.\newline
Since the main result \ref{Intro-Thm-Faisceau-Meromorphe} is formulated in the language of the ramification theory of Abbes and Saito,  we need the following comparison result to conclude.

\begin{prop}[{\textit{cf}.  \ref{Abbes-Saito-Kato}}]
	\label{Intro-Comparison AbbesSaito Kato}
With the notation of \ref{Intro-Faisceau-Meromorphe},  let $\chi_{\cF}^{[t]}$ be the character of $\rho_{\cF}^{[t]}$.  Then,  the following statements hold.
\begin{itemize}
\item[(i)] For a rational number $t\in [r, r']$,  we have the identity
\begin{equation}
	\label{Intro-Comparison AbbesSaito Kato 1}
\langle\widetilde{a}_f(t),  \chi_{\cF}^{[t]} \rangle=\sw_{\AS}(\cF,  t).
\end{equation}
\item[(ii)] For rational numbers $r\leq t < t' \leq r'$,  we have the identity
\begin{equation}
	\label{Intro-Comparison AbbesSaito Kato 2}
\langle \widetilde{\sw}_f^{\beta}([t, t']),  \chi_{\cF}^{[t]} \rangle=\varphi_s(\cF, t) - \varphi_s(\cF, t').
\end{equation}
\end{itemize}
\end{prop}
The proof of these identities comes down to an application of \cite[8.25]{VCS},  a result which uses crucially the comparison of characteristic cycles in \cite[10.4]{Hu} linking Kato's ramification theory to Abbes and Saito's.

\subsection{} \label{Intro-Organisation-Texte}
The article is organized as follows. In \S \ref{Kato-Hurwitz formula}, we recall a few properties of the category of formal étale local rings studied in \cite[\S 3]{VCS}, which is a formal analogue of an algebraic category studied by Kato in \cite[\S 5]{K1}. We give a detailed proof of the Kato-Hurwitz formula for morphisms of this category, following closely Kato's algebraic proof.  Section \ref{TNCF} is devoted to the nearby cycles formula, after the requisite preparation.  Section \ref{Positivity} gives a proof the non vanishing of $R\Psi_{Y/S}(\Lambda)_x$ at non-smooth points (due to T.  Saito) and deduces some positivity consequences (\ref{Intro-Discriminant-Convex},  \ref{Euler-Poincaré-Negatif}); it also contains the proofs of Theorem \ref{Intro-Perfection-Projectivity-Rationality}.
In section \ref{Variation-Conductors},  the identities \eqref{Intro-Strategy-I 1} and \eqref{Intro-Strategy-II 2} are proved and the ensuing variational behavior for $t\to \langle\widetilde{a}_f(t), \chi\rangle$ is derived.  The last section hearkens back to the set up in \ref{Intro-Faisceau-Meromorphe}: the comparison \ref{Intro-Comparison AbbesSaito Kato} is established
and Theorem \ref{Intro-Faisceau-Meromorphe} is deduced from the aforementioned variation.

\noindent \subsection*{Acknowledgements} The bulk of this work was part of the author's PhD dissertation at the Université Paris-Saclay's \'Ecole Doctorale Mathématiques Hadamard and was prepared at the Institut des Hautes \'Etudes Scientifiques.  He thanks these institutions for their funding and hospitality respectively.  The author is very much indebted to his then advisor Ahmed Abbes for his careful reading,  numerous remarks and corrections.  He also thanks T. Saito for sharing his proofs of \ref{Saito-Inequality} and \ref{Saito-Vanishing} and allowing their inclusion here.

\section{Notations and Conventions} \label{Notations and Conventions}

Let $K$ be a complete discrete valuation field, $\cO_K$ its valuation ring, $\m_K$ its maximal ideal, $k$ its residue field, assumed to be algebraically closed of characteristic $p>0$, and $\pi$ a uniformizer of $\cO_K$. 
Let also $\oK$ be a separable closure of $K$, $\cO_{\oK}$ the integral closure of $\cO_K$ in $\oK$, $\overline{k}$ its residue field, $G_K$ the Galois group of $\oK$ over $K$, and $v_K:\oK^{\times}\to \bQ$ the valuation of $\oK$ normalized by $v_K(\pi)=1$. Let $D=\Sp(K\{\xi\})$ be the closed rigid unit disc over $K$.
For rational numbers $r'\geq r\geq 0$, we denote by $A(r', r)=A_K(r', r)$ the closed sub-annulus of $D$ defined by $r'\geq v_K(\xi)\geq r$ and by $A^{\circ}(r', r)$ the open sub-annulus given by $r'> v_K(\xi)>r$. We put $\scrS=\Spf(\cO_K)$ and denote by $s$ its unique point.
All formal schemes are assumed to be locally noetherian.

\section{Kato-Hurwitz formula}\label{Kato-Hurwitz formula}

\subsection{}\label{Property (P)-C_K}
For a formal relative curve $\fX/\scrS$ and a geometric point $\ox$ of $\fX$ \cite[2.5]{VCS} with image a closed point $x$ of $\fX$, we consider the following property : 
\begin{itemize}
\item[] {\hfil (P) \quad $\fX-\{x\}$ is smooth over $\scrS$ and $\fX$ is normal at $x$.}
\end{itemize}

We denote by $\widehat{\cC}_K$ the category whose objects are the rings that are isomorphic over $\cO_K$ to the formal étale local rings $\cO_{\fX, \ox}$ \cite[2.5]{VCS} for some couple $(\fX/\scrS, \ox)$ satisfying property (P), where $\fX$ is a formal relative curve over $\scrS$ and $\ox$ is a geometric point of $\fX$ over a closed point $x$ of $\fX$. The morphisms of $\widehat{\cC}_K$ are the finite $\cO_K$-homomorphisms inducing separable extensions of fields of fractions. By \cite[3.17]{VCS}, if $A$ is an object of $\widehat{\cC}_K$, then $A$ is a two-dimensional, henselian, normal, local ring with residue field $k$ and special fiber $A/\m_KA$ reduced and excellent.
We refer  the reader to \cite[2.5-2.12 and 3.17-3.23]{VCS} for other relevant properties of the category $\widehat{\cC}_K$ and its objects, when needed, as well as for the definition of its motivating algebraic analog $\cC_K$ introduced by Kato in \cite[\S 5]{K1}.

We briefly recall a few definitions associated to an object $A$ and a morphism $A\to B$ of $\widehat{\cC}_K$ (valid verbatim for $\cC_K$). We denote by $P(A)$ the set of height one prime ideals of $A$, by $P_s(A)$ its (finite) subset of prime ideals above $\m_K$ and by $P_{\eta}(A)$ the complement of $P_s(A)$ in $P(A)$. The quotient $A_0=A/\m_K A$ is a reduced ring \cite[3.17 (ii)]{VCS}; let $\widetilde{A_0}$ be its integral closure in its total ring of fractions. We denote by $\delta	(A)$ the $k$-dimension of the quotient vector space $\widetilde{A_0}/A_0$. The ring $A_K=A\otimes_{\cO_K}K$ is a Dedekind domain \cite[3.17 (i)]{VCS}, and the homomorphism $A_K\to B_K$ induced by $A\to B$ is an extension of Dedekind domains, hence projective. The corresponding bilinear trace map induces a well-defined $K$-linear determinant homomorphism
\begin{equation}
	\label{Property (P)-C_K 1}
 T_{B_K/A_K}: {\rm det}_{A_K} (B_K)\otimes_{A_K} {\rm det}_{A_K} (B_K)\to A_K,
\end{equation} 
where $\det_{A_K} (B_K)$ is the invertible $A_K$-module $\bigwedge_{A_K}^r B_K$, with $r$ the rank of $B_K$ as an $A_K$-module \cite[3.24-3.25]{VCS}. Following Kato \cite[\S 5]{K1}, we define the integer $d_{\eta}(B/A)$ to be the dimension of the cokernel of $T_{B_K/A_K}$.

For $\p\in P(A)$, $A_{\p}$ is a discrete valuation ring with residue field $\kappa(\p)$; we denote by $w_{\p}: {\rm Frac}(A)^{\times}\to \bZ$ the associated normalized valuation map. Moreover, $\kappa(\p)$ is a discrete valuation field. Indeed, if $\p\in P_s(A)$, the integral closure of the henselian Japanese ring $A/\p$ in $\kappa(\p)$ is a discrete valuation ring; if $\p\in P_{\eta}(A)$, $\kappa(\p)$ is a finite extension of $K$. In both cases, we denote by $\ord_{\p}: \kappa(\p)^{\times} \to \bZ$ the associated normalized valuation map.

For any $\p\in P_s(A)$, the couple $(A, \p)$ gives rise to a $\bZ^2$-valuation ring $V_A(\p)$ and to a henselian $\bZ^2$-valuation ring $V_A^h(\p)$, the henselization of $V_A(\p)$ \cite[3.18]{VCS}, whose field of fractions we denote by $\bK_A^h(\p)$. We let $v_{\p}: (\bK_A^h(\p))^{\times}\to \bZ^2$ be the associated normalized valuation map and $v_{\p}^{\alpha}$ (resp. $v_{\p}^{\beta}$) the composition of $v_{\p}$ with the first (resp. second) projection $\bZ^2\to \bZ$ \cite[3.7]{VCS}.

For any $\q\in P_s(B)$ above $\p$, the valuation ring $V_B^h(\q)$ is the integral closure of $V_A^h(\p)$ in $\bK^h_B(\q)$. It is a finite free $V_A^h(\p)$-module and $V_B^h(\q)=V_A^h(\p)[b]$ for some $b\in V_B^h(\q)$) \cite[3.23]{VCS}. Therefore, we have again a well-defined $K$-linear determinant homomorphism $T_{V_B^h(\q)/V_A^h(\p)}$, whose image is a nonzero principal ideal of $V_A^h(\p)$. Let $c(B/A, \p, \q)$ be a generator of this image. Following Kato \cite[\S 5]{K1}, we have a well-defined integer
\begin{equation}
	\label{Property (P)-C_K 2}
d_s(B/A)=\sum_{(\p, \q)} v^{\beta}_{\p}(c(B/A,\p,\q)), 
\end{equation}
where $\p$ runs over $P_s(A)$ and $\q$ runs over the elements of $P_s(B)$ above $\p$.

\begin{lem}[{\cite[Lemma 5.8]{K1}}]
	\label{K-theory}
Let $A$ be an object of $\mathcal{C}_K$ $($resp. $\widehat{\mathcal{C}}_K)$ and denote by $\bK$ its field of fractions. Then, for any $x\in \bK^{\times}$, we have
\begin{equation}
	\label{K-theory 1}
\sum_{\p\in P_s(A)} v_{\p}^{\beta}(x)=\sum_{\p\in P_{\eta}(A)} [\kappa(\p): K] w_{\p}(x).
\end{equation}
\end{lem}

\begin{proof}
We first note that the valuation map $w_{\p}$ coincide with $v_{\p}^{\alpha}$ for $\p\in P_s(A)$ \cite[3.7, 3.15]{VCS}. Denoting by $K_2(\bK)$ the second Milnor $K$-group of $\bK$, for any $p\in P(A)$, we have the Steinberg tame symbol $\partial_{\p}: K_2(\bK)\to \kappa(\p)^{\times}$ of $w_{\p}$. By definition, for any $x\in \bK^{\times}$, it satisfies
\begin{align}
\partial_{\p}(\{x, \pi\})=(-1)^{v_{\p}^{\alpha}(x)} x^{-1}\pi^{v_{\p}^{\alpha}(x)} \mod \p \quad {\rm if}~ \p\in P_s(A), \label{K-theory 2} \\
\partial_{\p}(\{x, \pi\})=\pi^{w_{\p}(x)} \mod \p \quad {\rm if}~ \p\in P_{\eta}(A). \label{K-theory 3}
\end{align}
Note that, for any $\p\in P(A)$, as $k$ is algebraically closed, the residue field of the discrete valuation ring associated to $\ord_{\p}: \kappa(\p)\to \bZ$ is of degree $d_{\p}=1$ over $k$.
If $\p\in P_s(A)$, then $\ord_{\p}(x\pi^{-v_{\p}^{\alpha}(x)}\mod \p)=v_{\p}^{\beta}(x)$ for any $x\in \bK^{\times}$ \cite[3.15]{VCS}.

The $K$-theory for $\Spec(A)$ yields a localization complex 
\begin{equation}
	\label{K-theory 4}
K_2(\bK)\xrightarrow{(\partial_{\p})_{\p}} \bigoplus_{\p\in P(A)}\kappa(\p)^{\times}\xrightarrow{(\partial'_{\p})_{\p}} \bZ=K_0(\kappa(\m_A)),
\end{equation}
where $\m_A$ is the maximal ideal of $A$ and $\partial'_{\p}: \kappa(\p)^{\times}\to \bZ$ is the composition of $\ord_{\p}$ with the multiplication map $\bZ\xrightarrow{\times d_{\p}} \bZ$. Then,
\begin{align}
\partial'_{\p}\circ\partial_{\p}(\{x, \pi\})=\ord_{\p}(x^{-1}\pi^{v_{\p}^{\alpha}(x)} \mod \p)=-v_{\p}^{\beta}(x) \quad {\rm if}~ \p\in P_s(A), \label{K-theory 5} \\
\partial'_{\p}\circ\partial_{\p}(\{x, \pi\})=[\kappa(\p): K] w_{\p}(x) \quad {\rm if}~ \p\in P_{\eta}(A). \label{K-theory 6}
\end{align}
Hence, equation \eqref{K-theory 1} follows from the relation $\sum_{\p\in P(A)} \partial'_{\p}\circ\partial_{\p}=0$ \eqref{K-theory 4}.
\end{proof}

\begin{lem}
	\label{Noether Normalization}
Let $\scrD=\Spf(\cO_K\{T\})$ be the formal closed unit disc over $\cO_K$ and $o$ the origin of its special fiber $\Spec(k[T])$. Then, for any object $A$ of $\widehat{\cC}_K$, there exists a morphism $\cO_{\scrD, o}\to A$ of the category $\widehat{\cC}_K$. 
\end{lem}

\begin{proof}
Let $\fX/\scrS$ be an affine formal relative curve and $\ox\to \fX$ a geometric point at a closed point of $\fX$ such that the couple $(\fX/\scrS, \ox)$ satisfies property (P) and $A\cong \cO_{\fX, \ox}$. As $\fX_s$ is a geometrically reduced $k$-curve, \textit{Noether normalization with separating transcendence basis} \cite[16.18]{Eisenbud} implies that we have a finite morphism $g_s : \fX_s\to \Spec(k[T])$ which is generically étale. As $\ox$ is over a closed point of $\fX$ and $k$ is algebraically closed, after replacing $T$ by a translation $T-z$, for some $z\in k$, we can assume that $g_s(\ox)=o$. Let $a\in \cO_{\fX}(\fX)$ be a lift of the image of $T$ in $\cO_{\fX_s}(\fX_s)$ by $g_s$. Then, the assignment $T\mapsto a$ defines an $\scrS$-morphism $g: \fX\to \scrD$ that lifts $g_s$ and is thus finite. As the $\scrS$-curve $\fX$ is normal, it is also Cohen-Macaulay \cite[Discussion below 0.16.5.1]{EGA.IV}. It follows that $g$ is flat \cite[0.17.3.5 (i)]{EGA.IV}. Then, as $g_s$ is étale over a nonempty open admissible subset of $\fX_s$, so is $g$ \cite[2.4.7, 2.4.8]{EGR}.
Hence, the homomorphism of formal étale stalks $\cO_{\scrD, o}\to \cO_{\fX, \ox}$ induced by $g$ gives the desired morphism of $\widehat{\cC}_K$.
\end{proof}

\begin{prop} [{\cite[5.7]{K1}}]
\label{Kato-Formula}
Let $A\to B$ be a morphism of $\widehat{\cC}_K$. Then, with the notation of \ref{Property (P)-C_K}, we have the following formula
\begin{equation}
	\label{Kato-Formula 1}
d_{\eta}(B/A)-d_s(B/A)= 2\delta(B)-2\deg(B/A)\delta(A),
\end{equation}
where $\deg(B/A)$ is the degree of the extension of fields of fractions induced by $A\to B$.
\end{prop}

\begin{proof}
Let $g: A\to B$ and $h: B\to C$ be morphisms of $\widehat{\cC}_K$. By \cite[III, Prop. 8]{Serre1}, we have
\begin{equation}
	\label{Kato-Formula 2}
d_{\eta}(C/A)=d_{\eta}(C/B) + \deg(C/B) d_{\eta}(B/A).
\end{equation}
As the formation of the two-dimensional henselian $\bZ^2$-valuation ring in \ref{Property (P)-C_K} is functorial \cite[3.22]{VCS}, \cite[III, Prop. 8]{Serre1} implies also that, for $\p\in P_s(A)$, $\q\in P_s(B)$ above $\p$ and $\r\in P_s(C)$ above $\q$, with the notation at the end of \ref{Property (P)-C_K}, we have
\begin{equation}
	\label{Kato-Formulal 3}
v_{\p}^{\beta}(c(C/A, \p, \r))=v_{\q}^{\beta}(c(C/B, \q, \r)) + \deg(B/A)v_{\p}^{\beta}(c(B/A, \p, \q)).
\end{equation}
Since, for each $\p\in P_s(A)$, the set of $\r\in P_s(C)$ above $\p$ is the disjoint union of the sets of $\r\in P_s(C)$ above $\q$, for $\q$ ranging over the subset of elements of $P_s(B)$ above $\p$, it follows that
\begin{equation}
	\label{Kato-Formulal 4}
d_s(C/A)=d_s(C/B) + \deg(C/B) d_s(B/A).
\end{equation}
Therefore, by \eqref{Kato-Formula 2} and \eqref{Kato-Formulal 4}, if \eqref{Kato-Formula 1} holds for two morphisms among $g$, $h$ and $h\circ g$, it holds also for the third. Hence, by \ref{Noether Normalization}, we can assume that $A=\cO_{\scrD, o}$ in the statement of the proposition. In particular, $A$ and $A_0=A/\m_KA$ are regular rings, $\delta(A)=0$ and $P_s(A)$ is a singleton. Then, the proposition follows from \ref{K-theory} and \cite[3.27]{VCS}.
\end{proof}

\begin{rem}
	\label{Kato-Strategy}
The above proof follows verbatim Kato's proof for the morphisms of $\cC_K$ \cite[5.7-5.8]{K1}.
\end{rem}

\section{The nearby cycles formula}\label{TNCF}

\subsection{}\label{Annulus-I}
For an integer $n\geq 0$, we let $C_n=C_{K, n}$ be the affinoid curve $\Sp(K\{\xi, \zeta\}/(\xi\zeta-\pi^n)$. 
It is the annulus in the closed unit disc $D$ defined by $n\geq v(\xi)\geq 0$. 
The admissible adic ring $\cO_K\{\xi, \zeta\}/(\xi\zeta-\pi^n)$ defines a formal model $\scrC_n=\scrC_{K, n}=\Spf(\cO_K\{\xi, \zeta\}/(\xi\zeta-\pi^n))$ of $C_n$ over $\cO_K$. 
The special fiber $\scrC_{n, s}=\Spec(k[\xi, \zeta]/(\xi\zeta))$ is a union of two copies of $\bA^1_k$ that intersect at the unique singular point $o=(0,0)$ of $\scrC_{n, s}$, an ordinary double point, and is geometrically reduced. 
It follows that $\cO_K\{\xi, \zeta\}/(\xi\zeta-\pi^n)$ is the unit ball of $K\{\xi, \zeta\}/(\xi\zeta-\pi^n)$
 for the sup-norm \cite[4.1]{A.S.1} and that $\scrC_n$ is the \textit{normalized integral model of} $C_n$ \textit{defined over} $\cO_K$ \cite[2.18-2.19]{VCS}. 
In particular, $\scrC_n$ is normal. 
Moreover, $\scrC_n-\{o\}$ is smooth over $\scrS$. 
Hence, the couple $(\scrC_n, o)$ satisfies property (P) in \ref{Property (P)-C_K}. 
Therefore, the formal étale local ring $\cO_{\scrC_n, o}$ is an object of the category $\widehat{\cC}_K$. 
If $\overline{\p}$ is a geometric generic point of $\scrC_{n, s}$, then its image $\p$ in $\scrC_{n, s}$ corresponds to a minimal prime ideal of $\cO_{\scrC_{n, s}, o}$; thus, the inverse image of $\p$ through the natural reduction map $\cO_{\scrC_n, o}\to \cO_{\scrC_{n, s}, o}$ of geometric stalks \cite[2.7.2]{VCS} is a height one prime ideal denoted $\p$ again. 

\subsection{}\label{Annulus-II} For the remainder of this section, we fix rational numbers $r'\geq r\geq 0$ and put $C=A(r', r)$, $C^{[r]}=A(r, r)$, $C^{[r']}=A(r', r')$, $C^{\circ}=A^{\circ}(r', r)$. There exist a finite extension $K'$ of $K$ and integers $m\geq n\geq 0$ such that $r, r' \in v(K')$ and 
\begin{equation}
	\label{Annulus-II 1}
C_{K'}=C \otimes_K K'\simeq A_{K'}(m, n)=\{x\in \oK~\lvert ~ m\geq v_{K'}(x)\geq n\},
\end{equation}
where $v_{K'}$ is the valuation of $K'$ normalized by $v_{K'}(\pi')=1$, for $\pi'$ a uniformizer of $K'$.
By the change of coordinate $\xi\mapsto \frac{\xi}{\pi_{K'}^n}$, we get an isomorphism $C_{m-n}\xrightarrow{\sim} C_{K'}$ over $K'$ and deduce the following formal isomorphism over $\scrS'=\Spf(\cO_{K'})=\{s'\}$
\begin{equation}
	\label{Annulus-II 2}
\scrC_{K', m-n}\xrightarrow{\sim} \scrC_{K'}.
\end{equation}
With \ref{Annulus-I}, the isomorphism \eqref{Annulus-II 2} shows that we have a distinguished geometric point $o$ of $\scrC_{K'}$, that $\scrC_{K'}$ is normal, that $\scrC_{K'}-\{o\}$ is smooth over $\scrS'$ and that the special fiber $\scrC_{s'}$ has exactly two irreducible components corresponding, via specialization, to the annuli with $0$-thickness $C_{K'}^{[r]}$ and $C_{K'}^{[r']}$, and intersecting at the unique singular point $o$ of $\scrC_{s'}$. 
Let $\p$ and $\p'$ be the corresponding generic points. 
Let $\overline{\p}$ and $\overline{\p}'$ be geometric generic points at $\p$ and $\p'$.
Let $V^h=V^h_{\cO_{\scrC, o}}(\p)$ (resp. $V^{\prime h}=V^h_{\cO_{\scrC, o}}(\p')$) be the henselian $\bZ^2$-valuation ring induced by the couple $(\cO_{\scrC, o}, \p)$ (resp. $(\cO_{\scrC, o}, \p')$) \eqref{Property (P)-C_K} and denote by $\bK^h$ (resp. $\bK^{\prime h}$) its field of fractions. 
Let $v: (\bK^h)^{\times}\to \bQ\times \bZ$ (resp. $v': (\bK^{\prime h})^{\times}\to \bQ\times \bZ$) be the composition of the normalized $\bZ^2$-valuation map $v_{\p}$ (resp. $v_{\p'}$) \eqref{Property (P)-C_K} with the injection $\bZ^2\to \bQ\times \bZ, ~ (a, b)\mapsto (a/e, b)$, where $e$ is the ramification index of $K'/K$. Let $v^{\alpha}$ and $v^{\beta}$ (resp. $v^{\prime\alpha}$ and $v^{\prime\beta}$) be the composition of $v$ (resp. $v'$) with the first and second projections $\bQ\times \bZ\to \bQ$ and $\bQ\times\bZ\to \bZ$.

\subsection{} \label{Models Formels}
We keep the notation of \ref{Annulus-II}.
Let $X$ be a smooth $K$-affinoid space and $f: X\to C$ a finite flat morphism.
As in \cite[4.9]{VCS}, with the same justifications (we don't need the étaleness assumption made in \textit{loc. cit.}), there exists a finite extension $K'$ of $K$ (taken to be larger than in \ref{Annulus-II}) such that $r, r'\in v_K(K')$ and, as for $C_{K'}$, the formal spectrum $\fX_{K'}=\Spf(\cO^{\circ}(X_{K'}))$ of the unit ball $\cO^{\circ}(X_{K'})$ of $\cO(X_{K'})$ for its sup-norm is an admissible formal model of $X_{K'}$ over $\cO_{K'}$, with a geometrically reduced special fiber. 
In other words, $\fX_{K'}$ is the normalized integral model of $X$ defined over $\cO_{K'}$.
In particular, $\fX_{K'}$ is normal.
Moreover, $\fX_{K'}$ is smooth over $\scrS'$ outside a finite set of its closed points.
Let $\hf_{K'} : \fX_{K'}\to \scrC_{K'}$ be the adic morphism induced by $f$, called \textit{the normalized integral model of $f$ over} $\cO_{K'}$. 
Then, the formation of $\scrC_{K'}$, $\fX_{K'}$ and $\hf_{K'}$ commute with further finite extensions of $K'$. 
An extension $K'$ of $K$ as above is said to be \textit{admissible for} $f$; from the aforementioned commutation, we see that any further extension of $K'$ is also admissible for $f$.

\begin{lem}
	\label{Generic-Étale}
Let $\fX$ and $\fY$ be formal relative curves over $\scrS$ and $\ox$ and $\oy$ respective geometric points of $\fX$ and $\fY$ over closed points such that the couples $(\fX/\scrS, \ox)$ and $(\fY/\scrS, \oy)$ satisfy property $($P$)$. 
Let $g: \fX\to \fY$ be a finite flat $\scrS$-morphism such that
$g(\ox)=\oy$.
\begin{itemize}
\item[(i)] If the generic fiber $g_{\eta}: \fX_{\eta}\to \fY_{\eta}$ is étale over a nonempty open admissible subset of $\fY_{\eta}$ contained in the tube of $\oy$, then $g$ induces a morphism $\cO_{\fY, \oy} \to \cO_{\fX, \ox}$ of the category $\widehat{\cC}_K$.
\item[(ii)] Let $G$ be a finite group with a right action on $\fX$. If $g: \fX \to \fY$ is $G$-equivariant and $\cO(\fX)^G\cong \cO(\fY)$, then $g$ induces a morphism $\cO_{\fY, \oy} \to \cO_{\fX, \ox}$ of $\widehat{\cC}_K$ whose extension of field of fractions is Galois of group a subgroup of $G$.
\end{itemize}
\end{lem}

\begin{proof}
We can assume that $\fX=\Spf(B)$ and $\fY=\Spf(A)$ are affine. Denote by $\bK_{\oy}$ (resp. $\bK_{\ox}$) the field of fractions of $\cO_{\fY, \oy}$ (resp. $\cO_{\fX, \ox}$). By \cite[2.12]{VCS}, we have
\begin{equation}
	\label{Generic-Étale 1}
\cO_{\fY, \oy}\otimes_A B\cong \prod_{\ox\in \fX(\oy)}\cO_{\fX, \ox},
\end{equation}
where $\fX(\oy)$ is the set of geometric points of $\fX$ above $\oy$. This shows the finiteness of the homomorphism $\cO_{\fY, \oy}\to \cO_{\fX, \ox}$.

(i)\quad Let $U$ be an admissible open affinoid subset of $\fY_{\eta}$ such that $f: f^{-1}(U)\to U$ is étale.
It is enough to show that $\bK_{\oy}$ defines a point in $\Spec(\cO(U))$, i.e. that
$\bK_{\oy}\otimes_{\cO(\fY)} \cO(U) \neq 0$.
Indeed, if $\bK_{\oy}\otimes_A \cO(U) \neq 0$, then there exists a fields extension $\bL/\bK_{\oy}$ and a commutative diagram
\begin{equation}
	\label{Generic-Étale 2}
\xymatrix{
A \ar[r] \ar[d] & \cO(U)\ar[d]\\
\bK_{\oy} \ar[r] &  \bL.
}
\end{equation}
Using the canonical isomorphism $\bL\otimes_{\cO(U)}\cO(U)\otimes_A B\cong \bL\otimes_{\cO_{\fY, \oy}}\cO_{\fY, \oy}\otimes_A B$, we deduce from \eqref{Generic-Étale 1} that
\begin{equation}
	\label{Generic-Étale 3}
\bL\otimes_{\cO(U)}\cO(f^{-1}(U))=\bL\otimes_{\cO(U)}\cO(U)\otimes_A B \cong \prod_{\ox\in \fX(\oy)}\bL\otimes_{\cO_{\fY, \oy}}\cO_{\fX, \ox}=\prod_{\ox\in \fX(\oy)} \bL\otimes_{\bK_{\oy}}\bK_{\ox}.
\end{equation}
We deduce that the canonical homomorphism
\begin{equation}
	\label{Generic-Étale 4}
\bL\to \bL\otimes_{\bK_{\oy}}\bK_{\ox}
\end{equation}
is étale, and thus $\bK_{\oy}\to \bK_{\ox}$ is a separable extension.

Now the admissible open immersion $U\hookrightarrow \fY_{\eta}$ lifts to a formal morphism $\cU \hookrightarrow \fY'\to \fY$, where $\fY'\to \fY$ is an admissible formal blow-up and $\cU\hookrightarrow \fY'$ is a formal open immersion. The geometric point $\oy$ is a specialization of a rig-point $z\in U$. The latter defines a rig-point $z'$ of $\fY'$ by a property of admissible formal blow-ups \cite[8.3, Prop. 5]{Bosch}. Then, as $\cU_{\eta}=U$ is contained in the tube of $\oy$, $z'$ lies in $\cU$ and induces a geometric point $\oy'\to \cU$ above $\oy$. Let $\cV$ be a formal affine open subset of $\cU$ containing $y'$. Then, $\cO_{\fY', \oy'}=\cO_{\cV, \oy'}$ and we have a commutative square
\begin{equation}
	\label{Generic-Étale 5}
\xymatrix{
\cO(\fY) \ar[r] \ar[d] & \cO(U)\ar[r] & \cO(\cV)\otimes_{\cO_K} K \ar[d] \\
\bK_{\oy} \ar[rr] & & \bK_{\oy'}, 
}
\end{equation}
where $\bK_{\oy'}$ is the field of fractions of $\cO_{\cV, \oy'}$. It follows that $\bK_{\oy}\otimes_{\cO(\fY)} \cO(U) \neq 0$.

(ii) \quad We use the following isomorphism deduced from \eqref{Generic-Étale 1} by tensoring with $\bK_{\oy}$ 
\begin{equation}
	\label{Generic-Étale 6}
\bK_{\oy}\otimes_{\cO(\fY)}\cO(\fX)\xrightarrow{\sim} \prod_{\ox\in \fX(\oy)} \bK_{\ox}.
\end{equation}
If $\cO(\fX)^G\cong \cO(\fY)$, then $(\prod_{\ox\in \fX(\oy)} \bK_{\ox})^G \cong \bK_{\oy}$, and thus, for any $\ox\in \fX(\oy)$, $\bK_{\ox}/\bK_{\oy}$ is Galois of group a subgroup of $G$.
\end{proof}

\begin{defi}
	\label{S_f & d_f}
We keep the notation and assumptions of \ref{Annulus-II} and \ref{Models Formels}.  
Let $K'$ be a finite extension of $K$ which is admissible for $f$.
\begin{itemize}
\item[(1)] We say that the finite flat morphism $f$ is \textit{generically étale} if $f$ is étale over a nonempty open admissible subset of $C$. If $f$ is generically étale, then it has only a finite number of (closed) ramification points in $C$. Indeed, by the Jacobian criterion \cite[6.4.21]{EGR}, the closed subset of $X$ where $f$ is not étale is the vanishing locus of a nonzero convergent power series in one variable, which is finite by the Weierstrass preparation theorem.
In particular, $f$ is étale over an open admissible subset of $C^{\circ}$.
\item[(2)] We define $S_f$ (resp. $S'_f$) to be the set of couples $\tau=(\ox_{\tau}, \q_{\tau})$, where $\ox_{\tau}$ is a geometric point of $\fX_{K'}$ above $o$ and $\q_{\tau}$ is a height $1$ prime ideal of $B_{\tau}=\cO_{\fX_{K'}, \ox_{\tau}}$ above the height $1$ prime ideal $\p$ (resp. $\p'$) of $A=\cO_{\scrC, o}$. 
This is a nonempty finite set which is independent of the choice of the large enough extension $K'$ of $K$ which is admissible for $f$ (see \cite[4.22]{VCS}).
\item[(3)] We assume that $f$ is generically étale. Then, for each $\tau\in S_f$ (resp. $S'_f$) the induced homomorphism $A\to B_{\tau}$ is a morphism of $\widehat{\cC}_K$ (\ref{Generic-Étale} (i)); we denote by $V^h_{\tau}$ (resp. $V^{\prime h}_{\tau}$) the associated henselian $\bZ^2$-valuation ring (\ref{Property (P)-C_K}) and by $\bK_{\tau}^h$ (resp. $\bK_{\tau}^{\prime h}$) its field of fractions. 
With the notation at the end of \ref{Property (P)-C_K}, we define the integers
\begin{align}
d_{f, s}=\sum_{\tau\in S_f} v^{\beta}(c(B_{\tau}/A, \p, \q_{\tau})), \label{S_f & d_f 1} \\
d'_{f, s}=\sum_{\tau\in S'_f} v^{\prime\beta}(c(B_{\tau}/A, \p', \q_{\tau})). \label{S_f & d_f 2}
\end{align}
\end{itemize}
\end{defi}

\begin{prop}
	\label{Sum-Kato-Formula}
We resume the notation and assumptions of \ref{Annulus-II} and \ref{Models Formels}. We further assume that the finite flat morphism $f: X\to C$ is generically étale.
Let $K'$ be a finite extension of $K$ which is admissible for $f$. 
We denote by $\ox_1, \ldots, \ox_N$ the geometric points of $\fX_{K'}$ above $o$ and put $A=\cO_{\scrC, o}$, $B_j=\cO_{\fX_{K'}, \ox_j}$, for $j=1, \ldots, N$. 
Then, the homomorphisms $A\to B_j$ induced by $\hf_{K'}$ are morphisms in $\widehat{\cC}_{K'}$ and we have
\begin{equation}
	\label{Sum-Kato-Formula 1}
\sum_{j=1}^N (d_{\eta}(B_j/A) - 2\delta(B_j))= d_{f, s} + d'_{f, s} - 2\deg(f) .
\end{equation}
\end{prop}

\begin{proof}
By \ref{Annulus-II} and \ref{Models Formels}, the following couples
\begin{equation}
	\label{Sum-Kato-Formula 2}
(\scrC_{K'}/\scrS', o)\quad {\rm and}\quad ((\fX_{K'} -\{\ox_i, ~ i\neq j\})/\scrS', \ox_j)_{1\leq j\leq N}
\end{equation}
satisfy property (P) (\ref{Property (P)-C_K}).
Hence, $A$ and the $B_j$ are objects of the category $\widehat{\cC}_{K'}$ (\ref{Property (P)-C_K}).
As $f_{K'}$ is étale over an nonempty admissible open subset of the tube $C_{K'}^{\circ}$ of $o$ (\ref{S_f & d_f}(1)), we see from \ref{Generic-Étale} that the homomorphisms $A\to B_j$ induced by $\hf_{K'}$ are indeed morphisms in $\widehat{\cC}_{K'}$. 
As it is readily seen from the definitions that 
\begin{equation}
	\label{Sum-Kato-Formula 3}
\delta(A)=1 \quad {\rm and}\quad \sum_{j=1}^N d_s(B_j/A) = d_{f, s} + d'_{f, s},
\end{equation}
equation \eqref{Sum-Kato-Formula 1} follows from \ref{Kato-Formula}.
\end{proof}

\begin{prop}\label{Nearby Cycles Lutke}
We resume the notation and assumptions of \ref{Annulus-II} and \ref{Models Formels}. 
We further assume that $X$ \textit{has trivial canonical sheaf}, that $f$ is étale over $C^{[r]}$ and $C^{[r']}$ and that
\begin{equation}
	\label{Nearby Cycles Lutke 1}
f^{-1}(C^{[r]})=\coprod_{j=1}^{\delta_f} \Delta_j \quad {\rm and}\quad 
f^{-1}(C^{[r']})=\coprod_{j=1}^{\delta'_f} \Delta'_j,
\end{equation}
where $\Delta_j=A(r/d_j, r/d_j)$ and $\Delta'_j=A(r'/d'_j, r'/d'_j)$ with the integer $d_j\geq 1$ (resp. $d'_j\geq 1$) the order of $f$ on $\Delta_j$ $($resp. $\Delta'_j)$ \cite[4.2]{VCS}.
Let $K'$ be a finite extension of $K$ which is admissible for $f$. Denote by $\ox_1, \ldots, \ox_N$ the geometric points of $\fX_{K'}$ above $o$ \eqref{Models Formels} and put $A=\cO_{\scrC, o}$, $B_j=\cO_{\fX_{K'}, \ox_j}$, for $j=1, \ldots, N$. 
Then, we have a \emph{nearby cycles formula}
\begin{equation}
	\label{Nearby Cycles Lutke 2}
	\sum_{j=1}^N \left(d_{\eta}(B_j/A) - 2\delta(B_j)+ \lvert P_s(B_j)\lvert\right)= \sigma + \delta_f - (\sigma' + \delta'_f),
\end{equation}
where $\lvert P_s(B_j)\lvert$ denotes the cardinality of $P_s(B_j)$ and $\sigma$ $($resp. $\sigma')$ is the total order of the derivative of the restriction $f_{\lvert f^{-1}(C^{[r]})}$ $($resp. $f_{\lvert f^{-1}(C^{[r']})})$ of $f$ \cite[4.5]{VCS}.
\end{prop}

\begin{proof}
We put $X^{[r]}=f^{-1}(C^{[r]})$, $X^{[r']}=f^{-1}(C^{[r']})$ and denote by $\scrC^{[r]}$, $\scrC^{[r']}$, $\fX^{[r']}$, $\fX^{[r']}$, $\fX$ and $\hf: \fX\to \scrC$ the respective normalized integral models of $C^{[r]}$, $C^{[r']}$, $X^{[r]}$, $X^{[r']}$, $X$ and $f:X\to C$ defined over $\cO_{K'}$ (\ref{Models Formels}).
We have a Cartesian diagram
\begin{equation}
	\label{Nearby Cycles Lutke 3}
\xymatrix{\ar @{} [dr] | {\Box} 
X^{[r']}\ar[r]\ar[d] & \ar @{} [dr] | {\Box}  X\ar[d]^f & X^{[r]} \ar[l]\ar[d] \\
C^{[r']}\ar[r] & C & C^{[r]}\ar[l],
}
\end{equation}
where the horizontal arrows are inclusions, and the following decompositions into disjoint unions
\begin{equation}
\label{Nearby Cycles Lutke 4}
\fX^{[r]}=\coprod_j \widehat{\Delta}_j \quad {\rm and}\quad \fX^{[r']}=\coprod_j \widehat{\Delta}'_j,
\end{equation} 
where the normalized integral model $\widehat{\Delta}_j=\Spf(\cO^{\circ}(\Delta_j))$ of $\Delta_j$ is the formal annulus of radius $\lvert\pi\lvert^{r/d_j}$ with $0$-thickness, defined over $K'$, and is isomorphic to $\Spf(\cO_{K'}\{T_j, T_j^{-1}\})$; and $\Delta'_j$ is similarly defined and isomorphic to $\Spf(\cO_{K'}\{T'_j, T_j^{\prime -1}\})$. As $\fX\times_{\scrC}\scrC^{[r]}$ (resp. $\fX\times_{\scrC}\scrC^{[r]}$) is an affine formal model of $X_{K'}^{[r]}$ (resp. $X_{K'}^{[r']}$) \eqref{Nearby Cycles Lutke 3}, with a geometrically reduced special fiber, by \cite[4.1(ii)]{VCS}, we have have a Cartesian diagram
\begin{equation}
	\label{Nearby Cycles Lutke 5}
\xymatrix{\ar @{} [dr] | {\Box} 
\fX^{[r']}\ar[r]\ar[d] & \ar @{} [dr] | {\Box}  \fX \ar[d]^{\widehat{f}} & \fX^{[r]} \ar[d] \ar[l]\\
\scrC^{[r']}\ar[r] & \scrC & \scrC^{[r]} \ar[l].
}
\end{equation}
The horizontal arrows above are formal open immersions. It follows that $\fX_{s'} - (\fX_{s'}^{[r]}\sqcup \fX_{s'}^{[r']})$ lies over the singular point $\scrC_{s'}-(\scrC_{s'}^{[r]}\sqcup \scrC_{s'}^{[r']})=\{o\}$.\newline

For each $1\leq j\leq \delta_f$,  we glue $\fX_{K'}$ and a formal closed disc $\fD_j=\Spf(\cO_{K'}\{S_j\})$ along the boundary $\scrC_j^{[r]}=\Spf(\cO_{K'}\{S_j, S_j^{-1}\})$ with gluing map $T_j\mapsto S_j^{-1}$.
For each $1\leq j\leq \delta'_f$, we also glue $\fX_{K'}$ and a formal closed disc $\fD'_j=\Spf(\cO_{K'}\{S'_j\})$ along the boundary $\scrC_j^{[r']}=\Spf(\cO_{K'}\{S'_j, S_j^{\prime -1}\})$ with gluing map $T'_j\mapsto S'_j$.
The resulting formal relative curve 
\begin{equation}
	\label{Nearby Cycles Lutke 6}
\fY_{K'}=(\coprod_{j=1}^{\delta'_f} \fD'_j)\cup \fX_{K'}\cup (\coprod_{j=1}^{\delta_f}\fD_j)\big/\sim ~\to \scrS'=\Spf(\cO_{K'}).
\end{equation}
has smooth rigid fiber and contains $\fX_{K'}$ as a formal open subscheme.
As $\fX_{K'}$ is normal, $\fY_{K'}$ is also normal.
Its special fiber $\fY_{s'}$ is the gluing of $\fX_{s'}$ and $\delta_f + \delta'_f$ copies of $\bA_k^1$; for each $1\leq j\leq \delta_f$ (resp. $1\leq j\leq \delta'_f$), the copy $\Spec(k[S_j])$ (resp. $\Spec(k[S'_j])$) is glued with $\Spec(k[T_j, T_j^{-1}]$ (resp. $\Spec(k[T'_j, T_j^{\prime -1}])$ by the identification $T_j\mapsto S_j^{-1}$ (resp. $T'_j\mapsto S'_j$). It follows that $\fY_{s'}$ is a projective $k$-curve.
Moreover, by construction, the singular locus of $\fY_{s'}$ is contained in the set $\fX_{s'}-(\fX_{s'}^{[r]}\sqcup \fX_{s'}^{[r']})$.
By Grothendieck's algebraization theorem, there exists a relative proper algebraic curve $Y'$ over $S'=\Spec(\cO_{K'})$ whose formal completion along its special fiber is $\fY_{K'}$ \cite[5.4.5]{EGA.III}.
As the rigid fiber $\fY_{\eta'}$ of $\fY_{K'}$ is smooth, so is the generic fiber $Y'_{\eta'}$ of $Y'$.
Since the canonical sheaf of $X$ is trivial, there exists a global section $\omega \in \Gamma(X_{K'}, \Omega^1_{X_{K'}/K'})$ inducing an isomorphism $\cO_{X_{K'}}\xrightarrow{\sim} \Omega^1_{X_{K'}/K'}$.
Thus we can write $df=f^{\dagger}\omega$, where $f^{\dagger}\in \Gamma(X_{K'}, \cO_X)$.
For each $j$, both $\omega\lvert \Delta_j$ and $dT_j$ (resp. $\omega\lvert \Delta'_j$ and $dT'_j$) induce a basis of $\Omega^1_{X_{K'}/K'}$ on $\Delta_j$ (resp. $\Delta'_j$); thus, we have $\omega\lvert \Delta_j=u_j(T_j) dT_j$ and $\omega\lvert \Delta'_j=u_j(T'_j) dT'_j$, for some $u_j(T_j)\in \Gamma(\Delta_j, \cO_{\Delta_j})^{\times}$ and $u_j(T'_j)\in \Gamma(\Delta'_j, \cO_{\Delta'_j})^{\times}$.
Hence, we deduce that $f'(T_j)=u_j(T_j) f^{\dagger}\lvert \Delta_j$ and $f'(T'_j)=u_j(T'_j) f^{\dagger}\lvert \Delta'_j$.
We choose a point $y_j$ (resp. $y'_j$) in the generic fiber $D_j$ (resp. $D'_j$) of $\fD_j$ (resp. $\fD'_j$) that is not in $\Delta_j$ (resp. $\Delta'_j$).
By the rigid Runge theorem \cite[3.5.2]{Raynaud-Abh}, we can then approximate $\hf: \fX_{K'}\to \scrC_{K'}$ by the formal completion $\hg:\fY_{K'}\to \widehat{\mathbb{P}}_{S'}^1$ of an algebraic morphism $g: Y'\to \mathbb{P}_{S'}^1$ satisfying $g^{-1}(\infty)\subset \{y_j, y'_j\}$, such that the induced morphism $g_{\eta'}: \fY_{\eta'}\to \mathbb{P}_{K'}^{1, {\rm rig}}$ on rigid fibers has poles at most at the $y_j$ and $y'_j$ and, on each $\Delta_j$ (resp. $\Delta'_j$), we have
\begin{equation}
	\label{Nearby Cycles Lutke 7}
\lvert g_{\eta'} - f \lvert_j ~< \lvert f^{\dagger}\lvert_j/\lvert u_j^{-1} (T_j)\lvert_{\sup} \quad ({\rm resp.}\quad  \lvert g_{\eta'} - f\lvert'_j ~< \lvert f^{\dagger}\lvert'_j/\lvert u_j^{-1} (T'_j)\lvert_{\sup}),
\end{equation}
where $\lvert \cdot \lvert_j$ (resp. $\lvert \cdot \lvert'_j$) is defined as the sup-norm of the restriction to $\Delta_j$ (resp. $\Delta'_j$).
As for $f$, we have $dg_{\eta'}\lvert X_{K'}=g^{\dagger}\omega$, for some $g^{\dagger}\in \Gamma(X_{K'}, \cO_X)$, $g'_{\eta'}(T_j)=u_j(T_j) g^{\dagger}\lvert \Delta_j$ and $g'_{\eta'}(T'_j)=u_j(T'_j) g^{\dagger}\lvert \Delta'_j$.
Since $Y'_{\eta'}$ is a smooth projective curve, and $dg$ is a meromorphic section of the canonical sheaf $\Omega^1_{Y'_{\eta'}/K'}$ which is nonzero (by the equality $\lvert g'_{\eta'}(T_j) \lvert_{\sup} =\lvert f'(T_j)\lvert_{\sup}$ on $\Delta_j$ established just below \eqref{Nearby Cycles Lutke 10}), we have
\begin{equation}
	\label{Nearby Cycles Lutke 8}
	2g(Y'_{\overline{\eta}'})-2\lvert \pi_0(Y'_{\overline{\eta}'})\lvert=\deg({\rm div}(dg_{\eta'})),
\end{equation}
where $g(Y'_{\overline{\eta}'})$ is the total genus of $Y'_{\overline{\eta}'}$, i.e. the sum of the genera of its connected components. Let us compute the right-hand side of \eqref{Nearby Cycles Lutke 8}.
On $\Delta_j$, taking the derivative of a power series expansion of $g_{\eta'}-f$ and using the strong triangle inequality gives
\begin{equation}
	\label{Nearby Cycles Lutke 9}
\lvert g'_{\eta'}(T_j) - f'(T_j) \lvert_{\sup}\leq \lvert g_{\eta'} - f\lvert_j.
\end{equation}
Since $\lvert g^{\dagger} - f^{\dagger}\lvert_j\leq \lvert u_j^{-1}(T_j)\lvert_{\sup}\lvert (g_{\eta'} - f)'(T_j) \lvert_{\sup}$ and $\lvert f^{\dagger}\lvert_j\leq \lvert u_j^{-1}(T_j)\lvert_{\sup}\lvert f'(T_j) \lvert_{\sup}$, equations \eqref{Nearby Cycles Lutke 7} and \eqref{Nearby Cycles Lutke 9} yield both following inequalities
\begin{equation}
	\label{Nearby Cycles Lutke 10}
\lvert g'_{\eta'}(T_j) - f'(T_j) \lvert_{\sup} < \lvert f'(T_j)\lvert_{\sup}\quad {\rm and}\quad \lvert g^{\dagger} - f^{\dagger}\lvert_j < \lvert f^{\dagger}\lvert_j.
\end{equation}
Therefore, we also have $\lvert g'_{\eta'}(T_j) \lvert_{\sup} =\lvert f'(T_j)\lvert_{\sup}$ and $\lvert g^{\dagger} \lvert_j =\lvert f^{\dagger}\lvert_j$.
On $\Delta'_j$, the same argument gives $\lvert g'_{\eta'}(T'_j) \lvert_{\sup} =\lvert f'(T'_j)\lvert_{\sup}$ and $\lvert g^{\dagger} \lvert'_j =\lvert f^{\dagger}\lvert_j'$.
Hence, at each point of the normalization $\widetilde{\fY}_{s'}$ of $\fY_{s'}$, $f^{\dagger}$ and $g^{\dagger}$ have the same order as defined in \cite[2.20]{VCS}, and so do $f'(T_j)$ and $g'_{\eta'}(T_j)$ (resp. $f'(T'_j)$ and $g'_{\eta'}(T'_j)$).  Denoting by $C_+(y)$ the fiber of a point $y\in \fY_{s'}$ under the specialization map $\fY_{\eta'}\to \fY_{s'}$, it follows from \cite[2.21]{VCS}, that, for each $x_j\in \fX_{s'}-(\fX_{s'}^{[r]}\sqcup \fX_{s'}^{[r']})$, we have $\deg({\rm div}(g^{\dagger})\lvert C_+(x_j))=\deg({\rm div}(f^{\dagger})\lvert C_+(x_j))$. Hence, because we have
\begin{equation}
	\label{Nearby Cycles Lutke 11}
{\rm div}(dg_{\eta'})\lvert C_+(x_j)={\rm div}(g^{\dagger})\lvert C_+(x_j) + {\rm div}(\omega)\lvert C_+(x_j),
\end{equation}
and similarly for $df$ and $f^{\dagger}$, we obtain
$\deg({\rm div}(dg_{\eta'})\lvert C_+(x_j))=\deg({\rm div}(df\lvert C_+(x_j))$.
Moreover, as $f$ is étale on $X_{K'}^{[r]}$, so is $g_{\eta'}$; hence, ${\rm div}(dg_{\eta'}\lvert X_{K'})$ is supported in the tube of $\fX_{s'}-(\fX_{s'}^{[r]}\sqcup \fX_{s'}^{[r']})$. Therefore, we have
\begin{equation}
	\label{Nearby Cycles Lutke 12}
\deg({\rm div}(dg_{\eta'}\lvert X_{K'}))=\sum_{j=1}^N \deg({\rm div}(df)\lvert C_+(x_j))=\sum_{j=1}^N d_{\eta}(B_j/A).
\end{equation}
We denote by $\Delta_j^{-}$ the annulus $\Delta_j$  seen as the boundary of the disc $D_j$, with coordinate $S_j=T_j^{-1}$. Since $g_{\eta'}$ is étale over $\Delta_j^{-}$ (resp. $\Delta'_j$), ${\rm div}(dg_{\eta'})$ (resp. ${\rm div}(dg_{\eta'})$) is supported on $C_+(y_j)$ (resp. $C_+(y'_j)$). As $D_j-\Delta_j^{-}=C_+(y_j)$ (resp. $D'_j-\Delta'_j=C_+(y'_j)$), and $g'(T_j)$ and $f'(T_j)$ (resp. $g'(T'_j)$ and $f'(T'_j)$) have the same order $\sigma_j$ (resp. $\sigma'_j$) on the annulus $\Delta_j$ (resp. $\Delta'_j$), \cite[2.21]{VCS} again yields
\begin{equation}
	\label{Nearby Cycles Lutke 13}
\deg({\rm div}(dg_{\eta'}\lvert D_j-\Delta_j^{-}))={\rm ord}_{y_j}(g'_{\eta'}(T_j^{-1}))=-2-\sigma_j.
\end{equation}
\begin{equation}
	\label{Nearby Cycles Lutke 14}
\deg({\rm div}(dg_{\eta'}\lvert D'_j-\Delta'_j))={\rm ord}_{y'_j}(g'_{\eta'}(T'_j))=\sigma'_j.
\end{equation}
Summing \eqref{Nearby Cycles Lutke 13} over the $j$'s and adding \eqref{Nearby Cycles Lutke 12}, we find at last that the total degree is
\begin{equation}
	\label{Nearby Cycles Lutke 15}
\deg({\rm div}(dg_{\eta'}))=\sum_{j=1}^N d_{\eta}(B_j/A) + \sigma' - \sigma -2\delta_f.
\end{equation}
Now, let $R\Psi$ be the nearby cycles functor associated to the proper structure morphism $Y'\to S'$ and let $\Lambda$ be a finite field of characteristic different from $p$. Denoting by $Z$ the closed subset $\fX_{s'}-(\fX_{s'}^{[r]}\sqcup \fX_{s'}^{[r']})$ of the special fiber $\fY_{s'}\cong Y'_{s'}$, $i : Z\to Y'_{s'}$ the closed immersion and $j: U=Y'_{s'} - Z \to Y'_{s'}$ the inclusion of the complement, the long exact sequence of cohomology induced by the short exact sequence $0\to j_{!}(\Lambda_{\lvert U})\to \Lambda\to i_{*}(\Lambda_{\lvert Z})\to 0$ of sheaves on  $Y'_{s'}$ gives the following equality of Euler-Poincaré characteristics
\begin{equation}
	\label{Nearby Cycles Lutke 16}
\chi(Y'_{s'},\Lambda)=\chi_c (U, \Lambda_{\lvert U}) + \chi(Y'_{s'}, i_{*}(\Lambda_{\lvert Z})),
\end{equation}
where $\chi_c (\cdot)$ is the Euler-Poincaré characteristic with compact support. As the residue fields of the points in $Z$ coincide with the algebraically closed field $k$, we get $\chi(Y'_{s'}, i_{*}(\Lambda_{\lvert Z}))=\dim_{\Lambda} H_{\textrm{ét}}^0 (Z, \Lambda_{\lvert Z})=\lvert Z\lvert= N$. As $U$ is a disjoint union of $\delta_f+\delta'_f$ copies of $\bA_k^1$, we see that
\begin{equation}
	\label{Nearby Cycles Lutke 17}
\chi_c (U, \Lambda_{\lvert U}) =(\delta_f+\delta'_f) \chi_c (\bA_k^1,\Lambda)=\delta_f+\delta'_f.
\end{equation}
Since $Y'$ is normal, the strict localization$Y'_{(\ox)}$ at any geometric point $\ox \to Y'_{s'}$ is also normal; hence, $Y'_{(\ox)}\times\eta'$ is reduced. Moreover, as $Y'_{s'}\cong\fY_{s'}$ is reduced, so is $Y'_{(\ox)}\times s'=(Y'_{s'})_{(\ox)}$. Therefore, applying \cite[18.9.8]{EGA.IV} to the flat local homomorphism $Y'_{(\ox)}\to S'$, we see that the Milnor tube $Y'_{(\ox)}\times \overline{\eta}'$ is connected. As $R^i\Psi(\Lambda_{\lvert Y'_{\overline{\eta}'}})_{\ox}=H_{\textrm{ét}}^i(Y'_{(\ox)}\times \overline{\eta}', \Lambda)$ \cite[XIII, 2.1.4]{SGA7}, the sheaf $R^0\Psi(\Lambda_{\lvert Y'_{\overline{\eta}'}})$ is thus isomorphic to $\Lambda_{\lvert Y'_{s'}}$ and $R^i\Psi(\Lambda_{\lvert Y'_{\overline{\eta}'}})=0$ for $i >1$ \cite[I, Théorème 4.2]{SGA7}. Moreover, $R^1\Psi(\Lambda_{\lvert Y'_{\overline{\eta}'}})$ is concentrated in the singular locus of $Y'_{s'}$ \cite[XIII, 2.1.5]{SGA7}, located in $Z$. It thus follows from \eqref{Nearby Cycles Lutke 16} and \eqref{Nearby Cycles Lutke 17} that
\begin{equation}
	\label{Nearby Cycles Lutke 18}
N + \delta_f+\delta'_f -\sum_{j=1}^N \dim_{\Lambda} H_{\textrm{ét}}^1(Y'_{(\ox_j)}\times \overline{\eta}', \Lambda)=\chi(Y'_{s'}, R\Psi(\Lambda)).
\end{equation}
By the proper base change theorem, we also have the equality
\begin{equation}
	\label{Nearby Cycles Lutke 19}
\chi(Y'_{s'}, R\Psi (\Lambda))=\chi(Y'_{\overline{\eta}'}, \Lambda)=2\lvert \pi_0(Y'_{\overline{\eta}'})\lvert-2g(Y'_{\overline{\eta}'}).
\end{equation}
It remains to link the cohomology group in \eqref{Nearby Cycles Lutke 18} to $\delta(B_j)$ and $P_s(B_j)$ in the following way. As $\fX_{K'}$ is a formal open subscheme of $\fY_{K'}$, we have $\cO_{\fX_{K'}, \ox_j}=\cO_{\fY_{K'}, \ox_j}$. Then, \cite[2.8]{VCS} gives that 
\begin{equation}
	\label{Nearby Cycles Lutke 20}
B_j/\m_{K'} \xrightarrow{\sim} \cO_{\fY_{s'}, \ox_j}=\cO_{Y'_{s'}, \ox_j}.
\end{equation}
Since, for $A\in {\rm Obj}(\cC_{K'})$ (resp. ${\rm Obj}(\widehat{\cC}_{K'}))$, $P_s(A)$ identifies with the set of minimal prime ideals of $A/\m_{K'}$, it then follows that
\begin{equation}
	\label{Nearby Cycles Lutke 21}
\delta (B_j)=\delta(\cO_{Y', \ox_j}) \quad {\rm and}\quad \lvert P_s(B_j)\lvert=\lvert P_s(\cO_{Y', \ox_j})\lvert.
\end{equation}
As the couple $((Y'-\{\ox_i, i\neq j\})/S', \ox_j)$ satisfy property (P) in \ref{Property (P)-C_K}; then, \cite[Prop. 5.9]{K1}, in conjunction with \eqref{Nearby Cycles Lutke 21}, implies that
\begin{equation}
	\label{Nearby Cycles Lutke 22}
2\delta(B_j) - \lvert P_s(B_j)\lvert + 1= \dim_{\Lambda} H_{\textrm{ét}}^1(Y'_{(\ox_j)}\times \overline{\eta}', \Lambda).
\end{equation}
Finally, combining \eqref{Nearby Cycles Lutke 8}, \eqref{Nearby Cycles Lutke 15}, \eqref{Nearby Cycles Lutke 18}, \eqref{Nearby Cycles Lutke 19} and \eqref{Nearby Cycles Lutke 22} yields \eqref{Nearby Cycles Lutke 2}, which concludes the proof.
\end{proof}

\subsection{} \label{Decomposition SemiStable}
We resume the notation and assumptions of \ref{Annulus-II} and \ref{Models Formels}. We furthermore assume that the finite flat morphism $f:X\to A(r', r)$ is generically étale. We know from \cite[2.3]{Lutke} (see also \cite[4.4]{VCS}), via the semi-stable reduction theorem, that there exist a finite extension $K'$ of $K$ and a sequence of rational numbers $r'=r_0> r_1>\cdots > r_n > r_{n+1}=r$ in $v_K(K')$ such that $f_{K'}^{-1}(A^{\circ}_{K'}(r_{i-1}, r_i))$ is a finite disjoint union of open annuli $A_{K'}^{\circ}(r_{i-1}/d_{ij}, r_i/d_{ij})$ and the restrictions of $f$ to the latter annuli are étale of the form 
\begin{equation}
A_{K'}^{\circ}(r_{i-1}/d_{ij}, r_i/d_{ij})\to A^{\circ}_{K'}(r_{i-1}, r_i), \quad \xi_{ij}\mapsto \xi_{ij}^{d_{ij}}(1+h_{ij}(\xi_{ij})),
\end{equation} 
for some integers $d_{ij}\geq 1$, the order of $f$ on the annulus $A_{K'}^{\circ}(r_{i-1}/d_{ij}, r_i/d_{ij})$, and functions $h_{ij}$ on the same annulus satisfying $\lvert h_{ij}\lvert_{\rm sup} < 1$. 
A radius, i.e. an element of $[r, r']\cap \bQ$, which is different from all the $r_i$'s is said to be \emph{non-critical}; being non-critical is stable under base change.
It follows that the assumption \eqref{Nearby Cycles Lutke 1} is satisfied if we restrict $f$ over a sub-annulus $A(t', t)\subset C$, where $t'\geq t$ are non-critical radii of $f$.

\begin{lem}
	\label{AnnulationDe-d-Eta}
We use the notation of \ref{Sum-Kato-Formula}.
If the morphism $f: X\to C$ is étale, then, for every $j=1, \ldots, N$, we have $d_{\eta}(B_j/A)=0$. 
\end{lem}

\begin{proof}
As $\cO(C_{K'})$ is reduced and the $K'$-affinoid algebra $\cO(X_{K'})$ is finite over $\cO(C_{K'})$, $\cO(\fX)$ is also a finite algebra over $\cO(\scrC)$ \cite[6.4.1/6]{BGR}. Moreover, since $f$ is étale, $\cO(\fX)$ is rig-étale over $\cO(\scrC)$ \cite[6.4.12]{EGR}. It follows from  \cite[1.14.15]{EGR} that $\cO(\fX)[1/\pi']=\cO(X_{K'})$ is étale over $\cO(\scrC)[1/\pi']=\cO(C_{K'})$, where $\pi'$ is a unformizer of $\cO_{K'}$. By \cite[2.12]{VCS}, we also have
\begin{equation}
	\label{AnnulationDe-d-Eta 1}
A_{K'}\otimes_{\cO(C_{K'})}\cO(X_{K'})\cong K'\otimes_{\cO_{K'}}(A\otimes_{\cO(\scrC)}\cO(\fX))\cong K'\otimes_{\cO_{K'}}(\prod_{j=1}^N B_j)=\prod_{j=1}^N B_{j, K'}.
\end{equation}
Whence we deduce that, for each $j=1, \ldots, N$, the extension of Dedekind domains $A_{K'}\to B_{j, K'}$ is étale and thus $d_{\eta}(B_j/A)=0$.
\end{proof}

\begin{lem}
	\label{Singular-Points}
The fiber $\hf_{s'}^{-1}(o)$ of the morphism $\hf_{s'}: \fX_{s'}\to \scrC_{s'}$ induced by the normalized integral model of $f$ over $\cO_{K'}$ lies in the non-smooth locus of $\fX_{s'}$.
\end{lem}

\begin{proof}
Since $f:X\to C$ is finite, the induced morphism $\hf: \fX_{K'}\to \scrC_{K'}$ is also finite \cite[6.4.1/6]{BGR}. Since $\scrC_{K'}$ is $\cO_{K'}$-flat and $f$ is surjective, $\hf$ is surjective. It follows that the induced morphism on special fibers $\hf_{s'}: \fX_{s'}\to \scrC_{s'}$ is also finite and surjective. We deduce that, for a (geometric) point $x$ of $\fX_{s'}$ above $o\in \scrC_{s'}$, the morphism
\begin{equation}
	\label{Singular-Points 1}
(\fX_{s'})_{(x)} \to (\scrC_{s'})_{(o)}
\end{equation}
on strict localizations is surjective. If $x$ was smooth, then \eqref{Singular-Points 1} would factor through the normalization map $\widetilde{(\scrC_{s'})_{(o)}}\to (\scrC_{s'})_{(o)}$ and the image of $(\fX_{s'})_{(x)} \to \widetilde{(\scrC_{s'})_{(o)}}$ would lie in one of the two connected components of $\widetilde{(\scrC_{s'})_{(o)}}$, contradicting the surjectivity of \eqref{Singular-Points 1}. 
\end{proof}

\section{Positivity, projectivity and rationality of nearby cycles} \label{Positivity}
The notation and assumptions are as in \S \ref{Notations and Conventions}.

\begin{prop}[{T. Saito \cite{SaitoBook}}]
	\label{Saito-Inequality}
Let $Y\to S=\Spec(\cO_K)$ be a normal relative curve and $x$ a closed point of the special fiber $Y_s$ such that $Y-\{x\}$ is smooth over $S$. Then, we have
\begin{equation}
	\label{Saito-Inequality 1}
\dim_{\Lambda} R^1\Psi_{Y/S} (\Lambda)_{x} \geq \lvert P_s(\cO_{Y, x}^{\rm sh})\lvert - 1,
\end{equation}	
where $\Lambda$ is a finite field of characteristic different from $p$ and $P_s(\cO_{Y, x}^{\rm sh})$ denotes the set of height $1$ prime ideals, above the closed point of $S$, of the strict henselization $\cO_{Y, x}^{\rm sh}$ of $\cO_{Y, x}$ with respect to $k$ \eqref{Property (P)-C_K}.
\end{prop}

\begin{proof} (T. Saito)\quad
Since the formation of nearby cycles commutes with dominant base change of traits $S'\to S$ \cite[Th. finitude, 3.7]{SGA4demi} and $Y_{s'}\xrightarrow{\sim} Y_s$, we may assume, by the local semi-stable reduction theorem \cite[3.2.2, 4.9]{Saito87}, that we have a proper $S$-morphism $f : W \to  Y$ such that $W$ is regular and semi-stable, the exceptional divisor $E=f^{-1}(x)$ is a strict normal crossing divisor and $f$ induces an isomorphism $W-E\xrightarrow{\sim} Y-\{x\}$ ($W$ is the MN-model of $Y$ in the terminology of \textit{loc. cit.}, which exists by \cite{Lipman}). The special fiber $W_s$ identifies with $E\cup \widetilde{Y_s}$ and $P_s(\cO_{Y, x}^{\rm sh})= \pi^{-1}(x)=E\cap \widetilde{Y_s}$, where $\pi: \widetilde{Y_s}\to Y_s$ is the normalization map.

As $f$ is proper, the base change morphism
\begin{equation}
	\label{Saito-Inequality 2}
R\Psi_{Y/S}(Rf_{\ast} \Lambda)\to Rf_{s\ast}(R\Psi_{W/S}(\Lambda))
\end{equation}
is an isomorphism. Observe that $f$ induces an isomorphism on the generic fibers.
Hence, by the proper base change theorem, the spectral sequence of the hypercohomology of the functor $\Gamma(E, -)$ with respect to the complex $R\Psi_{W/S}(\Lambda)$ yields a spectral sequence
\begin{equation}
	\label{Saito-Inequality 3}
E_2^{p, q}= H^p(E, R^q\Psi_{W/S}(\Lambda)) ~ \Rightarrow ~ R^{p+q}\Psi_{Y/S}(\Lambda)_x.
\end{equation}
For any $q\geq 2$, $R^q\Psi_{W/S}(\Lambda)=0$ \cite[I, Théorème 4.2]{SGA7} and $R^1\Psi_{W/S}(\Lambda)$ is concentrated at the non-smooth locus of $W_s$ \cite[XIII, 2.1.5]{SGA7}, which is $P_s(\cO_{Y, x}^{\rm sh}) ~\cup ~{\rm Sing}(E)$. Therefore, \eqref{Saito-Inequality 3} yields the exact sequence
\begin{equation}
	\label{Saito-Inequality 4}
R^1\Psi_{Y/S}(\Lambda)_x \to \bigoplus_{w\in P_s(\cO_{Y, x}^{\rm sh}) ~\cup ~{\rm Sing}(E)} R^1\Psi_{W/S}(\Lambda)_w \to H^2(E, \Lambda)\to 0.
\end{equation}
On the one hand, the $\Lambda$-dimension of the top cohomology group $H^2(E, \Lambda)$ is the number irreducible components of $E$ \cite[IX, 4.7]{SGA4}; as $E$ is connected, by induction on this number, one sees that it is $\leq \lvert {\rm Sing}(E)\lvert +1$.
On the other hand, as $R^1\Psi_{W/S}(\Lambda)$ coincides with the corresponding tame nearby cycles \cite[3.5]{Illusie1}, we know from \cite[I, Théorème 3.3]{SGA7} that the $\Lambda$-dimension of every $R^1\Psi_{W/S}(\Lambda)_w$ is $1$. Then, \eqref{Saito-Inequality 1} follows from \eqref{Saito-Inequality 4}.
\end{proof}

\begin{cor} [{T. Saito \cite{SaitoBook}}]
	\label{Saito-Vanishing}
We keep the notation and assumption of \ref{Saito-Inequality}.
\begin{itemize}
\item[1.] We have the following inequalities
\begin{equation}
	\label{Saito-Vanishing 1}
\dim_{\Lambda} R^1\Psi_{Y/S} (\Lambda)_{x} \geq \delta(\cO_{Y, x}^{\rm sh})\geq \lvert P_s(\cO_{Y, x}^{\rm sh})\lvert - 1\geq 0,
\end{equation}
where the notation $\delta(\cO_{Y, x}^{\rm sh})$ is as introduced in \ref{Property (P)-C_K}.
\item[2.] The following conditions are equivalent :
\begin{itemize}
\item[(i)] $Y$ is smooth over $S$;
\item[(ii)] $Y_s$ is smooth over $k$;
\item[(iii)] $\delta(\cO_{Y, x}^{\rm sh})=0$;
\item[(iv)] $R^1\Psi_{Y/S} (\Lambda)_{x}=0$.
\end{itemize}
\end{itemize}
\end{cor}

\begin{proof}
1. \quad We know from Kato \cite[5.9]{K1} that 
\begin{equation}
	\label{Saito-Vanishing 2}
\dim_{\Lambda} R^1\Psi_{Y/S} (\Lambda)_{x} + \lvert P_s(\cO_{Y, x})\lvert - 1 = 2\delta(\cO_{Y, x}).
\end{equation}
With \eqref{Saito-Inequality 1}, this implies \eqref{Saito-Vanishing 1}.

2. \quad Since $Y\to S$ is flat, (i) and (ii) are equivalent. The latter  is also equivalent to $Y_s$ being normal, that is $\delta(\cO_{Y, x})=0$. Finally, it is well-known that (i) implies (iv), and the implication (iv) $\Rightarrow$ (iii) follows from \eqref{Saito-Vanishing 1}.
\end{proof}

\subsection{}\label{Discriminant}
We resume the notation and assumptions of \ref{Annulus-II} and \ref{Models Formels} for the morphism $f: X \to C$.
For a rational number $t\in [r', r]$, let $f^{[t]}: X^{[t]}\to C^{[t]}$ be the restriction of $f$ over the annulus $C^{[t]}$ of radius $\lvert \pi\lvert^t$ with $0$-thickness.
It follows from \cite[4.2, 4.4]{VCS} that, for an extension $K'$ of $K$ admissible for $f^{[t]}$ \eqref{Models Formels}, the unit ball $\cO^{\circ}(X_{K'}^{[t]})$ of the affinoid $K'$-algebra $\cO(X_{K'}^{[t]})$ of $X^{[t]}$, with respect to its sup-norm, is a locally free module over the unit ball $\cO(\cC_{K'}^{[t]})$ (see also \cite[4.16]{VCS}). The associated discriminant
\begin{equation}
	\label{Discriminant 1}
\d_f[t]=\d_{\cO^{\circ}(X_{K'}^{[t]})/\cO^{\circ}(\cC_{K'}^{[t]})}
\end{equation}
is an invertible ideal of $\cO^{\circ}(\cC_{K'}^{[t]})$. Hence, it has a well-defined sup-norm $\lvert \d_f [t] \lvert_{\rm sup}$ \cite[4.15]{VCS}, which can be written as
\begin{equation}
	\label{Discriminant 2}
\lvert \d_f[t]\lvert_{\rm sup}=\lvert \pi\lvert^{\partial_f(t)},  \quad {\rm for ~ some ~ unique} \quad \partial_f(t) \in \bQ,
\end{equation}
thus defining the \textit{discriminant function} $\partial_f: [r', r]\cap \bQ{\geq 0}\to \bQ$ associated to $f$.\newline

\noindent We note that, by \cite[4.16, 4.17]{VCS}, the definition is consistent with the definition given in \cite[4.22]{VCS} for the discriminant function associated to a rigid morphism $X\to D$ which is finite and generically étale. (We note also that $\partial_f$ here was denoted by $\partial_f^{\alpha}$ in \textit{loc. cit}.)\newline 

\noindent Lütkebohmert \cite[2.3, 2.6]{Lutke} proved the following variational result for $\partial_f$.

\begin{prop}
	\label{Variation-Lutke}
Let $r=r_{n+1}< r_n <\cdots < r_1< r_0=r'$ be a sequence containing all critical radii of $f$ \eqref{Decomposition SemiStable} and $\Delta_i=\coprod_j \Delta_{ij}$ the associated decomposition of $\Delta_i=f^{-1}\left(A^{\circ}(r_{i-1}, r_i)\right)$ into a disjoint union of open annuli such that the restriction $f\lvert \Delta_i$ is étale. Then, the function $\partial_f$ is affine on $\rbrack r_i, r_{i-1}\lbrack \cap \bQ$ and its right slope at $t\in \lbrack r_i, r_{i-1}\lbrack \cap \bQ$ is
\begin{equation}
	\label{Variation-Lutke 1}
	\frac{d}{dt}\partial_f(t^+)=\sigma_i - d + \delta_f(i),
\end{equation}
where $\sigma_i$ is the total order of the derivative of $f\lvert \Delta_i$ \cite[4.5]{VCS} and $\delta_f(i)$ is the number of connected components of $\Delta_i$ $($i.e. the number of $\Delta_{ij}$'s$)$.
\end{prop}

\begin{proof}
Although this result is stated in \cite[2.6]{Lutke} for the closed unit disc $D$ instead of $C$, i.e. for $r=\infty$ and $r'=0$, the proof, detailed in \cite[4.23]{VCS}, works for the above more general statement.
\end{proof}

\begin{prop}
	\label{Discriminant-Convex}
Assume that $f : X\to C$ is étale. Then, the discriminant function $\partial_f$ is convex.
\end{prop}

\begin{proof}
We use the notation in the proof of \ref{Nearby Cycles Lutke}. For every $j=1, \ldots, N$, we have from \eqref{Nearby Cycles Lutke 22}
\begin{equation}
	\label{Discriminant-Convex 1}
1 - \dim_{\Lambda} H_{\textrm{ét}}^1(Y'_{(\ox_j)}\times \overline{\eta}', \Lambda)= - 2\delta(B_j) + \lvert P_s(B_j)\lvert.
\end{equation}
As $\ox_j$ is singular \eqref{Singular-Points}, the stalk $R^1\Psi_{Y'/S'}(\Lambda)_{\ox_j}= H_{\textrm{ét}}^1(Y'_{(\ox_j)}\times \overline{\eta}', \Lambda)$ is nonzero \eqref{Saito-Vanishing}. It follows that $ - 2\delta(B_j) + \lvert P_s(B_j)\lvert\leq 0$. By \ref{AnnulationDe-d-Eta}, \eqref{Nearby Cycles Lutke 2} and \eqref{Variation-Lutke 1}, we deduce that 
\begin{equation}
	\label{Discriminant-Convex 2}
\frac{d}{dt}\partial_f(r+) - \frac{d}{dt}\partial_f(r'-)\leq 0.
\end{equation}
By \ref{Decomposition SemiStable}, we can apply this to the restriction of $f$ above the annulus $A(t, t')$ for all non-critical radii $t <t'$, hence the convexity of $\partial_f$.
\end{proof}

\subsection{}\label{Euler-Characteristic-Complex}
If $\Lambda$ is an artinian ring and $C^{\bullet}$ is a bounded complex of $\Lambda$-modules whose cohomology groups are finitely generated $\Lambda$-modules, then, following Ramero \cite[3.2.5]{Ramero}, we define its Euler characteristic as
\begin{equation}
	\label{Euler-Characteristic-Complex 1}
\chi(C^{\bullet})= \sum_{i\in \bZ} (-1)^{i} \frac{{\rm length}_{\Lambda} (H^{i}(C^{\bullet}))}{{\rm length}_{\Lambda}(\Lambda)}.
\end{equation}
We note that, if $C^{\bullet}$ is a complex of free $\Lambda$-modules, we have the identity
\begin{equation}
	\label{Euler-Characteristic-Complex 2}
\chi(C^{\bullet})= \sum_{i\in \bZ} (-1)^{i}{\rm rk}_{\Lambda}(C^{i}).
\end{equation}

\begin{lem}
	\label{Euler-Poincaré-Negatif}
Let $C\to S$ be a normal relative curve and $o$ a closed point of the special fiber $C_s$ such that $C-\{o\}$ is smooth over $S$.
Then, for every integer $n\geq 1$ and every lisse sheaf of $\bZ/\ell^n\bZ$-modules $\cF$ on $C_{\eta}$, we have
\begin{equation}
	\label{Euler-Poincaré-Negatif 1}
\chi(R\Psi_{C/S}(\cF)_o)\leq 0.
\end{equation}
\end{lem}

\begin{proof}
As $\chi$ is additive, the descending filtration $\cF\supset \ell\cF\supset \ell^2\cF\supset \cdots \supset \ell^n\cF$, whose quotients are lisse $\bF_{\ell}$-sheaves, reduces us to the case where $\cF$ is a lisse $\bF_\ell$-sheaf. By the same argument, we can assume that the restriction $\cF_{\lvert C_{(o)}\times_S\overline{\eta}}$ is irreducible. Then, if $V$ is the associated representation of the étale fundamental group $\pi_1$ of the connected $k(\overline{\eta})$-curve $C_{(o)}\times_S \overline{\eta}$, its invariant subspace $V^{\pi_1}$ is either $0$ or $V$.
If $\cF$ is non-constant, it follows that
\begin{equation}
	\label{Euler-Poincaré-Negatif 2}
R^0\Psi_{C/S}(\cF)_o= H^0(C_{(o)}\times_S\overline{\eta}, \cF_{\lvert C_{(o)}\times_S\overline{\eta}})=V^{\pi_1}=0,
\end{equation}
which proves the desired inequality. Otherwise, we can assume that $\cF$ is the constant sheaf with stalk $\bF_{\ell}$. Then, $R^0\Psi_{C/S}(\cF)_o\cong \bF_{\ell}$ and, by \ref{Saito-Vanishing}, $R^1\Psi_{C/S}(\cF)_o\neq 0$, which yields \eqref{Euler-Poincaré-Negatif 1}.
\end{proof}

\subsection{}\label{Groupes-de-Grothendieck}
For $\Lambda$ a finite field,  or a complete discrete valuation ring or a complete discrete valuation field,  and $G$ a finite group, we denote by $R_{\Lambda}(G)$ (resp. $P_{\Lambda}(G)$) the Grothendieck group of the category of finitely generated (resp. finitely generated projective) $\Lambda[G]$-modules. We denote by $\overline{R}_{\bQ}(G)$ the subgroup of $R_{\bQ_{\ell}}(G)$ formed by elements with rational coefficients.

\noindent A complex of modules over a (unitary but not necessarily commutative) ring is \textit{perfect} if it is quasi-isomorphic to a bounded complex of finitely generated and projective modules.
The class $[C^{\bullet}]\in P_{\Lambda}(G)$ of a perfect complex of $\Lambda[G]$-modules $C^{\bullet}$ is defined as follows. Let $P^{\bullet}\xrightarrow{\sim} C^{\bullet}$ be a quasi-isomorphism, where $P^{\bullet}$ is a bounded complex of finitely generated projective $\Lambda[G]$-modules, and put $[C^{\bullet}]=\sum_i (-1)^{i}[P^{i}]$. This definition is independent of $P^{\bullet}$ and the chosen quasi-isomorphism $P^{\bullet}\xrightarrow{\sim} C^{\bullet}$ \cite[VIII.2]{SGA5}; it thus extends to $C^{\bullet}$ in the essential image $D^b_{\rm Perf}(\Lambda[G])$,  in the derived category $D^b(\Lambda[G])$ of bounded complexes of $\Lambda[G]$-modules,  of the homotopy category of perfect complexes $K_{\rm Perf}(\bZ_\ell [G])$ \cite[VIII.7]{SGA5}.

\subsection{}\label{Stalk-of-Nearby-Cycles-Complexes}
Let $Y/S$ be a normal relative curve and $x$ a closed point of the special fiber $Y_s$ such that $Y-\{x\}$ is smooth over $S$.
Let $G$ be a finite group with an admissible right action on $Y$ through $S$-automorphisms such that the induced action on the generic fiber $Y_{\eta}$ is free,  and denote by $\St_x$ the stabilizer of $x$.  We denote by $\phi: Y\to Y/G=C$ the quotient morphism and put $\phi(x)=o$. 
Let $\ell\neq p$ be a prime number and $n\geq 1$ an integer,  and put $\Lambda=\bZ/\ell^n\bZ$.  Then,  the constant sheaf $\Lambda_{Y_{\eta}}$ is naturally endowed with a canonical structure of a $G$-sheaf.  \newline
As $\phi$ is finite,  we have the following commutative diagram
\begin{equation}
	\label{Stalk-of-Nearby-Cycles-Complexes 1}
\xymatrix{
R\Psi_{C/S}(R\phi_{\eta\ast}\Lambda_{Y_{\eta}}) \ar[r]^{\sim} & R\phi_{s\ast}(R\Psi_{Y/S}(\Lambda_{Y_{\eta}})) \ar[d]^{\wr} \\
R\Psi_{C/S}(\phi_{\eta\ast}\Lambda_{Y_{\eta}}) \ar[u]^{\wr} \ar[r]^{\sim} & \phi_{s\ast}(R\Psi_{Y/S}(\Lambda_{Y_{\eta}})).
}
\end{equation}
where the top horizontal arrow is the proper base change isomorphism.  Taking the stalk at $o$ of the below horizontal arrow in \eqref{Stalk-of-Nearby-Cycles-Complexes 1} yields an isomorphism
\begin{equation}
	\label{Stalk-of-Nearby-Cycles-Complexes 2}
R\Gamma(C_{(o)}\times_S\overline{\eta}, \phi_{\eta\ast}\Lambda_{Y_{\eta}}) \xrightarrow{\sim} R\Psi_{Y/S}(\Lambda_{Y_{\eta}})_x \quad {\rm in} ~ D^+(\Lambda).
\end{equation}
As $G$ acts trivially on $C_{\eta}$,  $\phi_{\eta\ast}\Lambda_{Y_{\eta}}$ is a sheaf of $\Lambda[\St_x]$-modules on $C_{\eta}$,  and thus the left-hand-side of \eqref{Stalk-of-Nearby-Cycles-Complexes 2} lives in $D^+(\Lambda[\St_x])$.  Through \eqref{Stalk-of-Nearby-Cycles-Complexes 2},  $R\Psi_{Y/S}(\Lambda_{Y_{\eta}})_x$ thus defines an element of $D^+(\Lambda[\St_x])$.

\begin{prop}
	\label{Projective-TorsionCoeffs}
With the above notation and assumptions,  the following statements hold.
\begin{itemize}
\item[(i)] The stalk $R\Psi_{Y/S}(\Lambda_{Y_{\eta}})_x$ is (isomorphic to) a perfect complex of $\Lambda[\St_x]$-modules of amplitude in $[0, 1]$.
\item[(ii)] The element $[R\Psi_{Y/S}(\bF_{\ell, Y_{\eta}})_x[1]]$ of $P_{\bF_{\ell}}(\St_x)$ is the class of a finitely generated and  projective $\bF_{\ell}[\St_x]$-module.
\end{itemize}
\end{prop}

\begin{proof}
Replacing $\St_x$ by $G$,  we may assume that $\St_x=G$; then, $\phi^{-1}(o)=x$.
The sheaf of $\Lambda[G]$-modules $\phi_{\eta\ast}\Lambda_{Y_{\eta}}$ on $C_{\eta}$ is flat.  Indeed,  since $G$ acts freely on $Y_{\eta}$,  the stalk of $\phi_{\eta\ast}\Lambda_{Y_{\eta}}$ at any geometric point $\overline{c}\to C_{\eta}$ is 
\begin{equation}
	\label{Projective-TorsionCoeffs 1}
(\phi_{\eta\ast}\Lambda_{C_{\eta}})_{\overline{c}}\cong \bigoplus_{\overline{y}\in \phi_{\eta}^{-1}(\overline{c})} \Lambda_{Y_{\eta}, \overline{y}}\cong \Lambda[G];
\end{equation}
whence we see that $\phi_{\eta\ast}\Lambda_{Y_{\eta}}$ is fiber-wise flat, hence flat, as a sheaf of $\Lambda[G]$-modules. 
As in the proof of \cite[I, 4.2]{SGA7}, being a projective limit of affine schemes of dimension $\leq 1$, $C_{(o)}\times_S\overline{\eta}$ has cohomological dimension $\leq 1$ by Artin's affine Lefschetz theorem \cite[XIV, 3.2]{SGA4}. Therefore, the functor $R\Gamma(C_{(o)}\times_S\overline{\eta}, \cdot)$ has cohomological dimension $\leq 1$ on both derived categories $D^+(C_{(o)}\times_S\overline{\eta}, \Lambda[G])$ and $D^+(C_{(o)}\times_S\overline{\eta}, Z(\Lambda[G]))$, where $Z(\Lambda[G])$ is the center of the Noetherian ring $\Lambda[G]$.  (Thus,  the functor extends to the corresponding unbounded derived categories).  As $\phi_{\eta\ast}\Lambda_{Y_{\eta}}$ is flat,  it follows from \cite[XVII, 5.2.11]{SGA4},  applied to the structure morphism $C_{(o)}\times_S\overline{\eta}\to \overline{\eta}$ that the complex of $\Lambda[G]$-modules $R\Gamma(C_{(o)}\times_S\overline{\eta}, \phi_{\eta\ast}\Lambda_{Y_{\eta}})$ has Tor-dimension $\leq 0$. Hence,  through \eqref{Stalk-of-Nearby-Cycles-Complexes 2},  the Tor-amplitude of $R\Psi_{Y/S}(\Lambda)_x$ is in $[0, 1]$.  From \cite[VIII, 8.1]{SGA5},  it follows also that $R\Psi_{Y/S}(\Lambda)_x$ is perfect, as a complex of $\Lambda[G]$-modules,  if and only if it is pseudo-coherent.  Now,  by \cite[Th. finitude, 3.2]{SGA4demi},  the $R^{i}\Psi_{Y/S}(\Lambda)_x$ are finite modules over $\Lambda$. Then, $\Lambda[G]$ being Noetherian,  $R\Psi_{Y/S}(\Lambda)_x$ is indeed pseudo-coherent by \cite[I, 3.5, 3.6]{SGA6},  thus perfect, which proves (i).\newline

\noindent Let $P^{\bullet}=(P^0\to P^1)$ be a complex of finitely generated and projective $\bF_{\ell}[G]$-modules such that we have a quasi-isomorphism 
\begin{equation}
	\label{Projective-TorsionCoeffs 2}
P^{\bullet}\xrightarrow{\sim} R\Psi_{Y/S}(\bF_{\ell})_x.
\end{equation}
By \cite[\S 14.3, Corollaires 1, 2]{Serre2}, $P^0$ and $P^1$ are direct sums of projective envelopes of simple $\bF_{\ell}[G]$-modules
\begin{equation}
	\label{Projective-TorsionCoeffs 3}
P^0=\bigoplus_N \alpha_N P_N, \quad P^1=\bigoplus_N \beta_N P_N.
\end{equation}
So, to prove (ii), it is enough to show that for any simple $\bF_{\ell}[G]$-module $N$, we have $\alpha_N\leq \beta_N$.
For any such $N$ and $N'$, $\Hom_{\bF_{\ell}[G]}(P_N, N')=0$ if $N\neq N'$: if $f: P_N\to N'$ is nonzero, then it is surjective and the image of $\ker f\subsetneq P_N$ through the canonical essential map $P_N\to N$ is $0$; this induces a nonzero map $N'\cong P_N/\ker f \to N$, which is an isomorphism.
Therefore, it is enough to show that, for any $\bF_{\ell}[G]$-module $M$ of finite length, we have
\begin{equation}
	\label{Projective-TorsionCoeffs 4}
\rk_{\bF_{\ell}}(\Hom_{\bF_{\ell}[G]}(P^0, M))\leq \rk_{\bF_{\ell}}(\Hom_{\bF_{\ell}[G]}(P^1, M)).
\end{equation}
The dual $M^{\vee}=\Hom_{\bF_{\ell}}(M, \bF_{\ell})$ of a $\bF_{\ell}[G]$-module $M$ is also an $\bF_{\ell}[G]$-module, and if $P$ is a finitely generated and projective $\bF_{\ell}[G]$-module, then the map
\begin{equation}
	\label{Projective-TorsionCoeffs 5}
P^{\vee}\otimes_{\bF_{\ell}[G]} M^{\vee} \to (P\otimes_{\bF_{\ell}[G]} M)^{\vee}, \varphi\otimes\psi \mapsto (a\otimes b\mapsto \varphi(a)\psi(b)) 
\end{equation}
is a $\bF_{\ell}$-isomorphism. Combined with \cite[X, 3.8]{SGA5}, this yields, for $i=0, 1$, $\bF_{\ell}$-isomorphisms
\begin{equation}
	\label{Projective-TorsionCoeffs 6}
\Hom_{\bF_{\ell}[G]}(P^i, M^{\vee})\cong (P^{i})^{\vee}\otimes_{\bF_{\ell}[G]} M^{\vee} \cong (P^{i}\otimes_{\bF_{\ell}[G]} M)^{\vee}.
\end{equation} 
Since a finitely generated $\bF_{\ell}[G]$-module and its dual have the same rank over $\bF_{\ell}$, it follows from \eqref{Projective-TorsionCoeffs 4} and \eqref{Projective-TorsionCoeffs 6} that it is enough to show that
\begin{equation}
	\label{Projective-TorsionCoeffs 7}
\rk_{\bF_{\ell}}(P^0\otimes_{\bF_{\ell}[G]} M) \leq \rk_{\bF_{\ell}}(P^1\otimes_{\bF_{\ell}[G]} M),
\end{equation}
for any $\bF_{\ell}[G]$-module $M$ of finite length.
The proof of \cite[XVII, 5.2.11]{SGA4} also establishes that the base change map
\begin{equation}
	\label{Projective-TorsionCoeffs 8}
M\otimes_{\bF_{\ell}[G]}^{\bL} R\Gamma(C_{(o)}\times_S\overline{\eta}, \phi_{\eta\ast}\bF_{\ell, Y_{\eta}}) \to R\Gamma(C_{(o)}\times_S\overline{\eta}, M\otimes_{\bF_{\ell}[G]}\phi_{\eta\ast}\bF_{\ell, Y_{\eta}})
\end{equation}
is an isomorphism.
Therefore, with \eqref{Stalk-of-Nearby-Cycles-Complexes 2} and \eqref{Projective-TorsionCoeffs 2}, we are reduced to proving that 
\begin{equation}
	\label{Projective-TorsionCoeffs 9}
\chi(C_{(o)}\times_S\overline{\eta}, M\otimes_{\bF_{\ell}[G]}\phi_{\eta\ast}\bF_{\ell, Y_{\eta}})\leq 0.
\end{equation}
Since $G$ acts freely on $Y_{\eta}$, the quotient morphism $\phi_{\eta}$ is (finite) étale; thus $\phi_{\eta\ast}\bF_{\ell, Y_{\eta}}$ is locally constant \cite[XVI, 2.2]{SGA4}, and \ref{Euler-Poincaré-Negatif} applies, which yields \eqref{Projective-TorsionCoeffs 9} and finishes the proof.
\end{proof}

\subsection{} \label{Z_ell-Coeffs}
We keep the notation and assumptions of \ref{Stalk-of-Nearby-Cycles-Complexes}.
As mentioned in the proof of \ref{Projective-TorsionCoeffs},  for any $n\geq 1$,  derived global section functor $R\Gamma$ on the $D^+(C_{(o)}\times_S\overline{\eta}, \bZ/\ell^n\bZ[\St_x])$ has finite cohomological dimension and thus,  as established in the proof of \cite[XVII, 5.2.11]{SGA4},  it satisfies the projection formula.  In particular,  we have an isomorphism \eqref{Stalk-of-Nearby-Cycles-Complexes 2}
\begin{equation}
	\label{Z_ell-Coeffs 1}
\bZ/\ell^n\bZ[\St_x]\otimes_{\bZ/\ell^{n+1}\bZ[\St_x]}^{\bL} R\Psi_{Y/S}(\bZ/\ell^{n+1}\bZ)_x \xrightarrow{\sim} R\Psi_{Y/S}(\bZ/\ell^n\bZ)_x \quad {\rm in}~ D(\bZ/\ell^n\bZ[\St_x]).
\end{equation}
Then,  by induction,  \cite[XIV,  \S 3,  n$^{\circ}3$,  Lemme 1]{SGA5} implies the existence of a projective system $(P^{\bullet}_n)_n$, where each $P^{\bullet}_n$ is a complex of finitely generated projective $\bZ/\ell^n\bZ[\St_x]$-modules concentrated in degrees $0$ and $1$,  the transition maps are $\bZ/\ell^n\bZ[\St_x]$-isomorphisms
\begin{equation}
	\label{Z_ell-Coeffs 2}
\bZ/\ell^n\bZ[\St_x]\otimes_{\bZ/\ell^{n+1}\bZ[\St_x]} P_{n+1}^{\bullet}\xrightarrow{\sim} \bZ/\ell^n\bZ\otimes_{\bZ/\ell^{n+1}\bZ} P^{\bullet}_{n+1} \xrightarrow{\sim} P^{\bullet}_n,
\end{equation}
and $P_n^{\bullet}$ is isomorphic to $R\Psi_{Y/S}(\bZ/\ell^n\bZ)_x$ in $D^b(\bZ/\ell^n\bZ[\St_x])$.\newline
(We note that the commutativity assumption in \textit{loc. cit.}, which is not satisfied by the family of rings $(\bZ/\ell^n\bZ[\St_x])_n$, is not necessary in our setting.  Indeed,  in \textit{loc. cit.},  this assumption intervenes only,  through the reference to \cite[III, \S 4, n$^{\circ} 6$, Lemme 2]{Bourbaki1}, in the proof of the intermediate result \cite[XIV,  \S 3,  n$^{\circ}3$,  Lemme 3 (ii)]{SGA5}; here,  the latter result is implied by \cite[14.4,  Proposition 42]{Serre2} since each $R\Psi_{Y/S}(\bZ/\ell^n\bZ)_x$, being perfect,  is already isomorphic to a complex of finitely generated projective $\bZ/\ell^n\bZ[\St_x]$-modules in $D^b(\bZ/\ell^n\bZ[\St_x])$.)\newline

\noindent We deduce a complex $P^{\bullet}=\varprojlim_n P^{\bullet}_n$ in  $K_{\rm Perf}(\bZ_\ell [\St_x])$.  Any other projective system with the same properties as $(P^{\bullet}_n)_n$ yields a complex quasi-isomorphic to $P^{\bullet}$ in  $K_{\rm Perf}(\bZ_\ell [\St_x])$. 
Therefore,  we have a well-defined element
\begin{equation}
	\label{Z_ell-Coeffs 3}
R\Psi_{Y/S}(\bZ_\ell)_x =\varprojlim_n P^{\bullet}_n \quad {\rm in} ~ D^b_{\rm Perf}(\bZ_\ell [\St_x]).
\end{equation}

\begin{rem}
	\label{Definition-Stalk-Nearby-Cycles}	
Through the forgetful functor $D^b_{\rm Perf}(\bZ_\ell [\St_x])\to D^b_{\rm Perf}(\bZ_\ell)$,  \eqref{Z_ell-Coeffs 3} defines also an $\ell$-adic nearby cycles stalk at $x$
\begin{equation}
	\label{Definition-Stalk-Nearby-Cycles 1}
R\Psi_{Y/S}(\bZ_\ell)_x ~ \in ~ D^b_{\rm Perf}(\bZ_\ell).
\end{equation}
Alternatively,  let us note that,  forgetting the action of $\St_x$,  the arguments given in the proof of \ref{Projective-TorsionCoeffs} (i) also yield that each $R\Psi_{Y/S}(\bZ/\ell^n\bZ)_x$ lies in $D^b_{\rm Perf}(\bZ/\ell^n\bZ)$ with amplitude in $[0,  1]$,  and satisfies a projection formula analogous to \eqref{Z_ell-Coeffs 1}.  (We can also deduce this from \ref{Projective-TorsionCoeffs} (i) as $\bZ/\ell^n\bZ[\St_x]$ is a free $\bZ/\ell^n\bZ$-modules of finite rank).  As in \ref{Z_ell-Coeffs},  we can then define $R\Psi_{Y/S}(\bZ_\ell)_x$ in $D^b_{\rm Perf}(\bZ_\ell)$ as the isomorphism class of the colimit of an inverse system $(P_n^{\bullet})_n$,  where $P_n^{\bullet}$ is a complex of finitely generated and projective $\bZ/\ell^n\bZ$-modules.  In fact,  we can take the later system to be the same as in \ref{Z_ell-Coeffs} where we forget the action of $\St_x$. 
\end{rem}

\begin{cor}
	\label{Projective-Z_l-Coeffs}
With the notation and assumptions of \ref{Z_ell-Coeffs},  the element $[R\Psi_{Y/S}(\bZ_{\ell})_x[1]]$ of $P_{\bZ_{\ell}}(\St_x)$ is the class of a finitely generated and projective $\bZ_{\ell}[\St_x]$-module.
\end{cor}

\begin{proof}
This follows from \ref{Projective-TorsionCoeffs} and \cite[\S 14.4, Corollaire 3.3]{Serre2}, which, for a finite group $H$, identifies $P_{\bZ_{\ell}}(H)$ with $P_{\bF_{\ell}}(H)$ through reduction $\mod \ell$.
\end{proof}

\subsection{}\label{l-adic NearbyCycles}
We keep the notation and assumptions of \ref{Z_ell-Coeffs}. The stalk at $x$ of the $\ell$-adic nearby cycles complex $R\Psi_{Y/S}(\bQ_{\ell})$ is defined as
\begin{equation}
	\label{l-adic NearbyCycles 1}
R\Psi_{Y/S}(\bQ_{\ell})_x=R\Psi_{Y/S}(\bZ_{\ell})_x\otimes_{\bZ_{\ell}}\bQ_{\ell}.
\end{equation}
It is (a quasi-isomorphism class of) a complex of finitely generated and projective $\bQ_{\ell}[\St_x]$-modules \eqref{Projective-TorsionCoeffs},  and thus defines an element in $P_{\bQ_{\ell}}(\St_x)$,  which is in fact the class of a finitely generated and projective $\bQ_{\ell}[\St_x]$-module \eqref{Projective-Z_l-Coeffs}.  We have the following rationality result.

\begin{thm}
	\label{Rationality-of-NearbyCycles}
The class $[R\Psi_{Y/S}(\bQ_{\ell})_x] \in P_{\bQ_{\ell}}(\St_x)$ lies in $\overline{R}_{\bQ}(\St_x)$.
\end{thm}

\begin{proof}
As in the proof of \ref{Saito-Inequality}, we can assume that we have a proper $S$-morphism $f: W\to Y$ such that $W$ is regular and semi-stable, $E=f^{-1}(x)$ is a strict normal crossing divisor of $W$ and $f$ induces an isomorphism $W-E \xrightarrow{\sim} Y-\{x\}$. We can moreover assume that $G$ acts on $W$ as a group of $S$-automorphisms such that $f$ is $G$-equivariant.
The proper base change isomorphism \cite[XIII, 2.1.7.1]{SGA7}
\begin{equation}
	\label{Rationality-of-NearbyCycles 1}
R\Psi_{Y/S}(Rf_{\eta \ast}\bQ_{\ell})\xrightarrow{\sim} Rf_{s \ast}(R\Psi_{W/S}(\bQ_{\ell}))
\end{equation}
is $G$-equivariant. By the proper base change theorem again, its stalk at $x$ yields an $\St_x$-equivariant isomorphism
\begin{equation}
	\label{Rationality-of-NearbyCycles 2}
R\Psi_{Y/S}(Rf_{\eta \ast}\bQ_{\ell})_x\xrightarrow{\sim} R\Gamma(E, R\Psi_{W/S}(\bQ_{\ell})\lvert E).
\end{equation}
As $W_{\eta}\xrightarrow{\sim} Y_{\eta}$, we have
\begin{equation}
	\label{Rationality-of-NearbyCycles 3}
Rf_{\eta \ast}\bQ_{\ell}\simeq \bQ_{\ell\lvert Y_{\eta}}.
\end{equation}
Denoting by $i: {\rm Sing}(W_s)\hookrightarrow W_s$ the canonical closed immersion, we have \cite[I, 3.3, 4.2]{SGA7}
\begin{equation}
	\label{Rationality-of-NearbyCycles 4}
R^0\Psi_{W/S}\bQ_{\ell}\simeq \bQ_{\ell\lvert W_s}, ~ R^1\Psi_{W/S}\bQ_{\ell}\simeq i_{\ast}\bQ_{\ell}(-1)_{\lvert {\rm Sing}(W_s)} ~ {\rm and}~ R^q\Psi_{W/S}\bQ_{\ell}=0 \quad {\rm for} ~ q > 1,
\end{equation}
where $(-1)$ indicates the Tate twist.
Note that, although the isomorphisms \eqref{Rationality-of-NearbyCycles 1} to \eqref{Rationality-of-NearbyCycles 4} are proved for torsion coefficients in \cite{SGA7}, they remain valid for $\bQ_{\ell}$. Indeed, since the derived direct image of a proper morphism to a quasi-compact scheme has finite cohomological dimension \cite[XVII, 5.2.8.1]{SGA4}, all the functors involved in the constructions of these isomorphisms, for torsion coefficients, have finite cohomological dimension; hence, the latter isomorphisms satisfy the projection formula for $\bZ/\ell^{n+1}\bZ\to \bZ/\ell^n\bZ$ \cite[6.5.7]{Fu} and thus pass to the limit.
Combining them yields
\begin{equation}
	\label{Rationality-of-NearbyCycles 5}
[R\Psi_{Y/S}(\bQ_{\ell})_x]= [R\Gamma(E, \bQ_{\ell})] - [\Gamma({\rm Sing}(E), \bQ_{\ell}(-1))].
\end{equation}
We note that the actions of $\St_x$ on the terms of the RHS of \eqref{Rationality-of-NearbyCycles 5} are functorially induced by the actions of $\St_x$ on $\bQ_{\ell\lvert W_s}$ and $i_{\ast}\bQ_{\ell}(-1)_{\lvert {\rm Sing}(Ws)}$, and the latter are defined, via \eqref{Rationality-of-NearbyCycles 4}, by the action on $R\Psi_{W/S}(\bQ_{\ell})$. In particular, as the first isomorphism in \eqref{Rationality-of-NearbyCycles 4} stems from the canonical $\St_x$-equivariant map $\bQ_{\ell\lvert W_s} \to R\Psi_{W/S}(\bQ_{\ell\lvert W_s})$, the induced action on $\bQ_{\ell\lvert W_s}$ is the trivial one deduced from the initial trivial action on $\bQ_{\ell}$.

We are now reduced to showing that the two terms on the RHS of \eqref{Rationality-of-NearbyCycles 5} lie  $\overline{R}_{\bQ}(\St_x)$.

Let $\pi: \widetilde{E}\to E$ be the normalization map of $E$. By dévissage, we have 
\begin{equation}
	\label{Rationality-of-NearbyCycles 6}
[R\Gamma(E, \bQ_{\ell})]= [R\Gamma(\widetilde{E}, \bQ_{\ell})] - [\Gamma(\pi^{-1}({\rm Sing}(E)), \bQ_{\ell})] + [\Gamma({\rm Sing}(E), \bQ_{\ell})].
\end{equation}
The action of $\St_x$ on each term of the RHS of \eqref{Rationality-of-NearbyCycles 6} is functorially induced by the the trivial action on $\bQ_{\ell}$. Let $\widetilde{E}_1, \ldots, \widetilde{E}_r$ be the irreducible components of $\widetilde{E}$. Then, we have a decomposition 
\begin{equation}
	\label{Rationality-of-NearbyCycles 7}
R\Gamma(\widetilde{E}, \bQ_{\ell})\simeq\bigoplus_{i=1}^r R\Gamma(\widetilde{E}_i, \bQ_{\ell})
\end{equation}
For $g\in \St_x$, we denote by $\widetilde{g}: \widetilde{E}\xrightarrow{\sim}\widetilde{E}$, $\widetilde{g}_{\bQ_{\ell}}: \widetilde{g}^{\ast}\bQ_{\ell}\xrightarrow{\sim}\bQ_{\ell}$ and 
\begin{equation}
	\label{Rationality-of-NearbyCycles 8}
R\Gamma(\widetilde{E}, \widetilde{g}_{\bQ_{\ell}}): R\Gamma(\widetilde{E}, \widetilde{g}^{\ast}\bQ_{\ell})\xrightarrow{\sim} R\Gamma(\widetilde{E}, \bQ_{\ell})
\end{equation}
the induced isomorphisms. Restricting \eqref{Rationality-of-NearbyCycles 8} with respect to \eqref{Rationality-of-NearbyCycles 7} yields isomorphisms
\begin{equation}
	\label{Rationality-of-NearbyCycles 9}
\varepsilon_i: R\Gamma(\widetilde{g}(\widetilde{E}_i), \bQ_{\ell}) \xrightarrow{\sim}  R\Gamma(\widetilde{E}_i, \bQ_{\ell}).
\end{equation}
It follows that the trace of $R\Gamma(\widetilde{E}, \widetilde{g}_{\bQ_{\ell}})$ is the sum of the traces of the $\varepsilon_i$ such that $\widetilde{g}(\widetilde{E}_i)=\widetilde{E}_i$. As such a proper smooth $k$-curve $\widetilde{E}_i$ is irreducible, the automorphism $\widetilde{g}: \widetilde{E}_i\xrightarrow{\sim}\widetilde{E}_i$ has only isolated fixed points (in finite number); thus, the Lefschetz fixed point formula \cite[Rapport, 5.3]{SGA4demi} applies, showing that the trace of $\varepsilon_i$ is an integer. Therefore, we have $R[\Gamma(\widetilde{E}, \bQ_{\ell})] \in \overline{R}_{\bQ}(\St_x)$.\newline
Similarly, for the trace of the automorphism of $\Gamma({\rm Sing}(E), \bQ_{\ell})$ induced by $g\in \St_x$, we can restrict to the subset ${\rm Sing}(E)^g$ of points fixed by $g$; then, for $z\in {\rm Sing}(E)^g$, the trace on $\Gamma(\{z\}, \bQ_{\ell})$ is $1$. Thus, we have
$[\Gamma({\rm Sing}(E), \bQ_{\ell})] \in \overline{R}_{\bQ}(\St_x)$.\newline
If $z\in {\rm sing}(E)$, then $\pi^{-1}(z)=\{\widetilde{z}_1, \widetilde{z}_2\}$ and $\Gamma(\pi^{-1}(z), \bQ_{\ell})=\bQ_{\ell}^2$. For $g\in \St_x$, $\widetilde{g}$ either fixes or permutes $\widetilde{z}_1$ and $\widetilde{z}_2$; hence, the trace of $g$ on $\Gamma(\pi^{-1}(z), \bQ_{\ell})$ is either $2$ or $0$. It follows that $[\Gamma(\pi^{-1}({\rm Sing}(E)), \bQ_{\ell})] \in \overline{R}_{\bQ}(\St_x)$.

For $[\Gamma({\rm Sing}(E), \bQ_{\ell}(-1))]$, as in the foregoing argument, if $g\in \St_x$, it is enough to compute the trace of $\Gamma({\rm Sing}(E)^g, g_{\bQ_{\ell}(-1)})$. In fact, since the action on $\Gamma({\rm Sing}(E), \bQ_{\ell}(-1))$ is functorially deduced from the action on $i_{\ast}\bQ_{\ell}(-1)_{\lvert {\rm Sing}(Ws)}$ defined by \eqref{Rationality-of-NearbyCycles 4}, it is enough to show that $[R^1\Psi_{W/S}(\bQ_{\ell})_z] \in \overline{R}_{\bQ}(\St_x)$ for any $z\in {\rm Sing}(E)^g$.
By \cite[XV, 2.2.5, C]{SGA7}, with the notation therein, Poincaré duality for the smooth generic fiber $W_{\eta}\xrightarrow{\sim} Y_{\eta}$ induces a perfect pairing
\begin{equation}
	\label{Rationality-of-NearbyCycles 10}
R^1\Psi_{W/S}(\bQ_{\ell\lvert W_{\eta}})_z \times H^1_{\{z\}}(W_s, R\Psi_{\eta}\bQ_{\ell}(1)) \to \bQ_{\ell}.
\end{equation}
By \cite[XV, 2.2.7]{SGA7}, we have a short exact sequence
\begin{equation}
	\label{Rationality-of-NearbyCycles 11}
0 \to H^0(\{z\}, \bQ_{\ell})(1) \to H^0(\pi^{-1}(z), \bQ_{\ell})(1) \to H^1_{\{z\}}(W_s, R\Psi_{\eta}\bQ_{\ell}(1)) \to 0.
\end{equation}
Endowing $\bQ_{\ell\lvert \{z\}}$ and $\bQ_{\ell\lvert \pi^{-1}(z)}$ with the trivial action of $\St_x$, \eqref{Rationality-of-NearbyCycles 11} is $\St_x$-equivariant for the functorially induced actions on the $H^0$'s. It follows that, for $g\in \St_x$, the trace induced on $ H^1_{\{z\}}(W_s, R\Psi_{\eta}\bQ_{\ell}(1))$ is $2-1=1$ if $\widetilde{g}$ fixes the two elements of $\pi^{-1}(z)$ and $0-1=-1$ otherwise. Therefore, we have $[\Gamma({\rm Sing}(E), \bQ_{\ell}(-1))] \in \overline{R}_{\bQ}(\St_x)$, which finishes the proof.
\end{proof}

\section{Variation of conductors of a morphism to a rigid annulus}\label{Variation-Conductors}

\subsection{} \label{Hypotheses}
Let $0\leq r < r'$ be rational numbers and put $C=A(r', r)$ and $C^{\circ}=A^{\circ}(r', r)$.
Let $X$ be a smooth and connected $K$-affinoid space and $f: X\to C$ a finite flat morphism which is generically étale \eqref{S_f & d_f}.
Let $G$ be a finite group with a right action on $X$ such that $f$ is invariant under $G$.
By \cite[6.3.3/3]{BGR}, $\cO(X)^G$ is a $K$-affinoid closed sub-algebra of $\cO(X)$ and $\cO(X)$ is finite over $\cO(X)^G$. 
We assume that $\cO(X)^G\cong \cO(C)$, i.e. $X/G=\Sp(\cO(X)^G)\cong C$.
In this section, we recall and extend the results of \cite[\S 7]{VCS} in the above setting.

\subsection{}\label{Group action}
We use the notation of \ref{S_f & d_f}.
Let $K'$ be a finite extension of $K$ which is admissible for $f$ \eqref{Models Formels}.
The right action of $G$ on $X$ induces a canonical action of $G$ on the normalized integral model $\fX_{K'}$ defined in \ref{Models Formels}. We deduce an action of $G$ on $S_f$ (resp $S'_f$) given as follows : for $g\in G$ and $\tau=(\ox_{\tau}, \q_{\tau}) \in S_f$ (resp. $S'_f$), we put $g\cdot\tau=(g\circ\overline{x}_{\tau}, g_{\#}^{-1}(\q_{\tau}))$, where $g_{\#}$ is the induced isomorphism
\begin{equation}
	\label{Group action 1}
\cO_{\fX_{K'}, g\circ \ox_{\tau}}\xrightarrow{\sim} \cO_{\fX_{K'}, \ox_{\tau}}.
\end{equation}

\begin{lem} \label{ActionTransitive-ExtenionsGaloisiennes-RevetementQuotient}
\begin{itemize}
\item[(1)] The above action of $G$ on the set $S_f$ $($resp. $S'_f)$ is transitive.
\item[(2)] For every $\tau\in S_f$ $($resp. $S'_f)$, the extension of fields of fractions $\bK_{\tau}^h/\bK^h$ $($resp. $\bK_{\tau}^{\prime h}/\bK^{\prime h})$ induced by $V^h_{\tau}/V^h$ $($resp. $V_{\tau}^{\prime h}/V^{\prime h})$ is finite and Galois of group $G_{\tau}$ $($resp. $G'_{\tau})$ isomorphic to the stabilizer of $\tau$ under the action of $G$ on $S_f$ $($resp. $S'_f)$.
\item[(3)] Let $H$ be a subgroup of $G$.
The quotient $X_H=X/H=\Sp(\cO(X)^H)$ is a smooth $K$-rigid space and the moprhism $f_H : X_H \to C$ induced by $f$ is finite, flat and generically étale.
The canonical maps $g_S: S_f\to S_{f_H}$ and $g'_S: S'_f\to S'_{f_H}$ induced by the quotient morphism $X\to Y$ are surjective and their fibers are the orbits of $H$ under the action of $G$ on $S_f$ and $S'_f$ respectively.
\end{itemize}
\end{lem}

\begin{proof}
The items (1) and (3) are respectively \cite[7.3]{VCS} and \cite[7.11]{VCS} extended from the target space $D$ to $C$. 
The proofs go through verbatim. Note that the $\ox$ in \textit{loc. cit.} corresponds to our geometric point $o$ \eqref{Annulus-II}. Item (2) follows from \ref{Generic-Étale} (see also \cite[7.4]{VCS}). Indeed, as the normalized integral model $\hf_{K'}: \fX_{K'}\to \scrC_{K'}$ of $f$ over $\O_{K'}$ \eqref{Models Formels} is finite and $\scrC_{K'}$ is normal and integral, the isomorphism $\cO(X)^G\cong \cO(C)$ implies that $\cO(\fX_{K'})^G\cong \cO(\scrC_{K'})$, and thus \ref{Generic-Étale}(ii) applies.
\end{proof}

\subsection{}\label{Characters}
As in \cite[7.5]{VCS}, for $\tau \in S_f$, the ramification theory of $\bZ^2$-valuation rings gives rise to a $\bQ$-valued class function $a_{G_{\tau}}^{\alpha}$ and a $\bZ$-valued class function $\sw_{G_{\tau}}^{\beta}$ on $G_{\tau}$ \cite[6.16]{VCS}. 
The functions $a_{G_{\tau}}^{\alpha}$ are conjugate to each other as $\tau$ varies in $S_f$; so are the functions $\sw_{G_{\tau}}^{\beta}$ for $\tau$ ranging in $S_f$. They induce on $G$ the class functions
\begin{equation}
	\label{Characters 1}
a_{f, K'}^{\alpha}=\Ind_{G_{\tau}}^G a_{G_{\tau}}^{\alpha} 
\quad {\rm and}\quad 
\sw_{f, K'}^{\beta}=\Ind_{G_{\tau}}^G \sw^{\beta}_{G_{\tau}}
\end{equation}
which are independent of the choice of $\tau\in S_f$. In the same way, for $\tau \in S'_f$, we also have functions $a_{G'_{\tau}}^{\alpha}$ and $\sw_{G'_{\tau}}^{\beta}$ on $G'_{\tau}$ inducing $a_{f, K'}^{\prime\alpha}$ and $\sw_{f, K'}^{\prime\beta}$ on $G$ respectively.

\begin{lem}
	\label{Independence of Base Change}
Let $L$ be a finite extension of $K'$. Then, we have
\begin{equation}
	\label{Independence of Base Change 1}
a_{f, L}^{\alpha}=a_{f, K'}^{\alpha}
\quad ({\rm resp.}\quad
a_{f, L}^{\prime\alpha}=a_{f, K'}^{\prime\alpha}),
\end{equation}
\begin{equation}
	\label{Independence of Base Change 2}
\sw_{f, L}^{\beta}=\sw_{f, K'}^{\beta}
\quad ({\rm resp.}\quad
\sw_{f, L}^{\prime\beta}=\sw_{f, K'}^{\prime\beta}).
\end{equation}
In particular, $a_{f, K'}^{\alpha}$ $($resp. $a_{f, K'}^{\prime\alpha})$ and $\sw_{f, K'}^{\beta}$ $($resp $\sw_{f, K'}^{\prime\beta})$ are independent of the choice of the extension $K'$ of $K$ which is admissible for $f$; we denote them by $a_f^{\alpha}$ $($resp. $a_f^{\prime\alpha})$ and $\sw_f^{\beta}$ $($resp $\sw_f^{\prime\beta})$.
\end{lem}

\begin{proof}
This is \cite[7.6]{VCS} extended from $D$ to $C$. The proof goes through verbatim.
\end{proof}

\subsection{}\label{Normalized Conductors}
We keep the notation of \ref{Group action}.
We modify $a_f^{\alpha}$ and $\sw_f^{\beta}$ in the following way. 
We put
\begin{equation}
	\label{Normalized Conductors 1}
\widetilde{a}_f^{\alpha}=\frac{\lvert G\lvert}{\lvert S_f\lvert} a_f^{\alpha}
\quad {\rm and} \quad
\widetilde{\sw}_f^{\beta}=\frac{\lvert G\lvert}{\lvert S_f\lvert} \sw_f^{\beta}.
\end{equation}
We consider similar modifications $\widetilde{a}_f^{\prime\alpha}$ and $\widetilde{\sw}_f^{\prime\beta}$ of $a_f^{\prime\alpha}$ and $\sw_f^{\prime\beta}$ respectively, by replacing $S_f$ with $S'_f$.

\begin{prop} \label{DiscVar}
We denote by $\langle \cdot, \cdot\rangle$ the usual pairing of class functions on $G$ \cite[2.2, Remarques]{Serre2}.
Let $H$ be a subgroup of $G$ and $\partial_{f_H}$ the discriminant function \eqref{Discriminant} associated to the quotient morphism $f_H: X/H\to C$ induced by $f$.
\begin{itemize}
\item[(i)] We have the following identities
\begin{equation}
	\label{DiscVar 1}
\partial_{f_H}(r) =\langle \widetilde{a}_f^{\alpha}, \bQ[G/H]\rangle,
\end{equation}
\begin{equation}
	\label{DiscVar 2}
\partial_{f_H}(r') =\langle \widetilde{a}_f^{\prime\alpha}, \bQ[G/H]\rangle,
\end{equation}
where the representation $\bQ[G/H]=\Ind_H^G 1_H$ of $G$ stands in abusively for its character.
\item[(ii)] We assume that $X/H$ has \emph{trivial canonical sheaf} and that $f_H$ satisfies the decomposition hypothesis \eqref{Nearby Cycles Lutke 1}. Then, we also have the following identity for the right and left derivatives of $\partial_{f_H}^{\alpha}$ at $r$ and $r'$ respectively
\begin{equation}
	\label{DiscVar 3}
\frac{d}{dt}\partial_{f_H}(r+) - \frac{d}{dt}\partial_{f_H}(r'-)=\langle \widetilde{\sw}_f^{\beta}, \bQ[G/H]\rangle + \langle \widetilde{\sw}_f^{\prime\beta}, \bQ[G/H]\rangle,
\end{equation}
\end{itemize}
\end{prop}

\begin{proof}
We first note that, as $G$ acts transitively on $S_f$, we have $\lvert G_{\tau}\lvert =\lvert G\lvert /\lvert S_f\lvert$. By Frobenius reciprocity, we have the following identities
\begin{equation}
	\label{DiscVar 4}
\langle a_f^{\alpha}, \bQ[G/H]\rangle=\langle a_{G_{\tau}}^{\alpha}, \bQ[G/H]\lvert G_{\tau} \rangle,
\end{equation}
\begin{equation}
	\label{DiscVar 5}
\langle \sw_f^{\beta}, \bQ[G/H]\rangle=\langle \sw_{G_{\tau}}^{\beta}, \bQ[G/H]\lvert G_{\tau} \rangle.
\end{equation}
Let $R$ be a set of representatives in $G$ of the double cosets $G_{\tau}\backslash G/H$. From \cite[\S 7.3, Prop. 22]{Serre2}, for $\tau \in S_f$, we have the identity
\begin{equation}
	\label{DiscVar 6}
\bQ[G/H]\lvert G_{\tau}= \bigoplus_{\sigma \in R} \Ind_{H_{\sigma}}^{G_{\tau}} 1_{H_{\sigma, \tau}},
\end{equation}
where $H_{\sigma, \tau}=\sigma H\sigma^{-1}\cap G_{\tau}$. If $\sigma\in R$ and $g\sigma h$ is another representative of the double coset $G_{\tau}\sigma H$, then $H_{g\sigma h, \tau}=H_{\sigma, \tau}$. Hence, $H_{\sigma, \tau}$ depends only on the double coset $G_{\tau}\sigma H$ and the sum \eqref{DiscVar 6} is taken over $G_{\tau}\backslash G/H$. Therefore, we have
\begin{equation}
	\label{DiscVar 7}
\langle a_f^{\alpha}, \bQ[G/H]\rangle=\sum_{\sigma\in G_{\tau}\backslash G/H} \langle a_{G_{\tau}}^{\alpha}, \bQ[G_{\tau}/H_{\sigma, \tau}]\rangle,
\end{equation}
\begin{equation}
	\label{DiscVar 8}
\langle \sw_f^{\beta}, \bQ[G/H]\rangle=\sum_{\sigma\in G_{\tau}\backslash G/H} \langle \sw_{G_{\tau}}^{\beta}, \bQ[G_{\tau}/H_{\sigma, \tau}]\rangle.
\end{equation}
From \cite[(6.18.4) and (6.18.5)]{VCS}, we get
\begin{equation}
	\label{DiscVar 9}
\lvert G_{\tau}\lvert \langle a_{G_{\tau}}^{\alpha}, \bQ[G_{\tau}/H_{\sigma, \tau}]\rangle =v^{\alpha}(\mathfrak{d}_{V^h(\sigma, \tau)/V^h}),
\end{equation}
\begin{equation}
	\label{DiscVar 10}
\lvert G_{\tau}\lvert \langle \sw_{G_{\tau}}^{\beta}, \bQ[G_{\tau}/H_{\sigma, \tau}]\rangle =v^{\beta}(\mathfrak{d}_{V^h(\sigma, \tau)/V^h}) -\frac{\lvert G_{\tau}\lvert}{\lvert H_{\sigma, \tau}\lvert}+ 1,
\end{equation}
where $V^h(\sigma, \tau)=(V^h_{\tau})^{H_{\sigma, \tau}}$. The subgroup $\sigma^{-1} H_{\sigma, \tau}\sigma$ of $\sigma^{-1} G_{\tau}\sigma=G_{\sigma^{-1}\cdot\tau}$ is $H_{{\rm id}, \sigma^{-1}\cdot\tau}$. Then, $\sigma^{-1}$ yields an isomorphism of $\bZ^2$-valuation rings $V^h(\sigma^{-1}\cdot\tau)\xrightarrow{\sigma} V^h_{\tau}$, via \eqref{Group action 1}, which induces an isomorphism
\begin{equation}
	\label{DiscVar 11}
V^h({\rm id}, \sigma^{-1}\cdot\tau)=(V^h(\sigma^{-1}\cdot\tau))^{\sigma^{-1}H_{\sigma, \tau}\sigma} \xrightarrow{\sim} V^h(\sigma, \tau).
\end{equation}
By \ref{ActionTransitive-ExtenionsGaloisiennes-RevetementQuotient} (3), the map $C: G_{\tau}\backslash G/H \to S_{f_H}, ~  G_{\tau}\sigma H\mapsto g_S (\sigma^{-1}\cdot \tau)$ is well-defined since, for $g\in G_{\tau}$ and $h\in H$, $(g\sigma h)^{-1}\cdot \tau=h^{-1}\cdot \tau$.
Moreover, as $G$ acts transitively on $S_f$ (\ref{ActionTransitive-ExtenionsGaloisiennes-RevetementQuotient} (1)) and $g_S$ is surjective, $C$ is also surjective.
If $C( G_{\tau}\sigma H)=C( G_{\tau}\sigma' H)$, then there exists $h\in H$ such that $\sigma^{-1}\cdot \tau=h\sigma^{'-1}\cdot \tau$; so $\sigma h \sigma^{'-1} \in G_{\tau}$ and thus $\sigma'\in  G_{\tau}\sigma H$. Hence, $C$ is also injective, hence a bijection.
We also have $V^h({\rm id}, \sigma^{-1}\cdot\tau)=V^h(g_S(\sigma^{-1}\cdot\tau))$.
It follows that, if $\sigma, \sigma' \in R$ represent double cosets such that $g_S(\sigma^{-1}\cdot\tau)=g_S(\sigma^{'-1}\cdot\tau)$, then $V^h({\rm id}, \sigma^{-1}\cdot\tau)=V^h({\rm id}, \sigma^{'-1}\cdot\tau)$.

Combining all this with \eqref{DiscVar 7}, \eqref{DiscVar 8}, \eqref{DiscVar 9} and \eqref{DiscVar 10} yields
\begin{equation}
	\label{DiscVar 12}
\langle \widetilde{a}_f^{\alpha}, \bQ[G/H]\rangle =\sum_{\tau'\in S_{f_H}} v^{\alpha}(\mathfrak{d}_{V^h_{\tau'}/V^h})=v^{\alpha}\left(\prod_{\tau'\in S_{f_H}} \mathfrak{d}_{V^h_{\tau'}/ V^h}\right)=\partial_{f_H}(r),
\end{equation}
\begin{equation}
	\label{DiscVar 13}
\langle \widetilde{\sw}_f^{\beta}, \bQ[G/H]\rangle=v^{\beta} \left(\prod_{\tau'\in S_{f_H}}\mathfrak{d}_{V^h_{\tau'}/ V^h}\right) - \deg(f_H) + \lvert S_{f_H}\lvert.
\end{equation}
Equation \eqref{DiscVar 12} yields the identities \eqref{DiscVar 1}. The same arguments work for $\widetilde{a}_f^{\prime\alpha}$ and $\widetilde{\sw}_f^{\prime\beta}$, producing identities similar to \eqref{DiscVar 12} and \eqref{DiscVar 13}, yielding also \eqref{DiscVar 2}. \newline
From \eqref{DiscVar 13}  and the corresponding formula for $\widetilde{\sw}_f^{\prime\beta}$, we see by \eqref{S_f & d_f 1} and \eqref{S_f & d_f 2} that
\begin{equation}
	\label{DiscVar 14}
\langle \widetilde{\sw}_f^{\beta}, \bQ[G/H]\rangle=d_{f_H, s}- \deg(f_H)+ \lvert S_{f_H}\lvert.
\end{equation}
\begin{equation}
	\label{DiscVar 15}
\langle \widetilde{\sw}_f^{\prime\beta}, \bQ[G/H]\rangle=d'_{f_H, s}- \deg(f_H)+ \lvert S'_{f_H}\lvert.
\end{equation}
Whence we deduce that
\begin{equation}
	\label{DiscVar 16}
\langle \widetilde{\sw}_f^{\beta} + \widetilde{\sw}_f^{\prime\beta}, \bQ[G/H]\rangle= d_{f_H, s} + d'_{f_H, s} - 2\deg(f_H) + \lvert S_{f_H} \lvert + \lvert S'_{f_H}\lvert.
\end{equation}
As $\lvert S_{f_H}\lvert + \lvert S'_{f_H}\lvert$ is also the sum over $j=1, \ldots, N$ of the integers $\lvert P_s(B_j)\lvert$ (notation of \ref{Sum-Kato-Formula}), by \ref{Sum-Kato-Formula}, applied to $f_H$, we now see that
\begin{equation}
	\label{DiscVar 17}
\langle \widetilde{\sw}_f^{\beta} + \widetilde{\sw}_f^{\prime\beta}, \bQ[G/H]\rangle=\sum_{j=1}^N (d_{\eta}(B_j/A) - 2\delta(B_j)+ \lvert P_s(B_j)\lvert).
\end{equation}
Now, if $X/H$ has a trivial canonical sheaf and $f_H$ satisfies the decomposition hypothesis \eqref{Nearby Cycles Lutke 1}, then we can apply to it the formula \eqref{Nearby Cycles Lutke 2}. The latter, combined with
\ref{Variation-Lutke 1} applied to $f_H$, and \eqref{DiscVar 17}, yields
\begin{equation}
	\label{DiscVar 18}
\langle \widetilde{\sw}_f^{\beta} + \widetilde{\sw}_f^{\prime\beta}, \bQ[G/H]\rangle=\sigma + \delta_{f_H} - (\sigma' + \delta'_{f_H})=\frac{d}{dt}\partial_{f_H}(r+) - \frac{d}{dt}\partial_{f_H}(r'-).
\end{equation}
\end{proof}

\subsection{}\label{Artin-Intrinseque}
Let $r \leq t \leq r'$ be a rational number, put $C^{[t]}=A(t, t)$, $X^{[t]}=f^{-1}(C^{[t]})$ and denote by $f^{[t]}: X^{[t]}\to C^{[t]}$ the morphism induced by $f$.
Let $K'$ be a finite extension of $K$ such that the normalized integral models $\scrC_{K'}^{[t]}$, $\fX_{K'}^{[t]}$ and $\hf_{K'}^{[t]}: \fX_{K'}^{[t]}\to \scrC_{K'}^{[t]}$ of $C^{[t]}$, $X^{[t]}$ and $f^{[t]}$ respectively are defined over $\cO_{K'}$ \eqref{Models Formels}. 
Let $\p^{(t)}$ be the generic point of the special fiber $\scrC_{s'}^{[t]}$ of $\scrC_{K'}^{[t]}$ and $\overline{\p}^{(t)}\to \scrC_{s'}^{[t]}$ a geometric generic point.
The latter defines a geometric point of $\scrC_{K'}^{[t]}$ \cite[2.5]{VCS}.
Let $\fX_{K'}^{[t]}(\overline{\p}^{(t)})$ be the set of geometric generic points of $\fX_{K'}^{[t]}$ above $\overline{\p}^{(t)}$. The natural right action of $G$ on $X$ induces an action of $G$ on $\fX_{K'}^{[t]}$, hence an action of $G$ on $\fX_{K'}^{[t]}(\overline{\p}^{(t)})$. The latter action is transitive.
For $\overline{\q}^{(t)}\in \fX_{K'}^{[t]}(\overline{\p}^{(t)})$, the natural map
\begin{equation}
	\label{Artin-Intrinseque 1}
A_{\overline{\p}^{(t)}}=\cO_{\scrC_{K'}^{[t]}, \overline{\p}^{(t)}}\to A_{\overline{\q}^{(t)}}=\cO_{\fX_{K'}^{[t]}, \overline{\q}^{(t)}}
\end{equation}
is a finite homomorphism of henselian discrete valuation rings \cite[2.8, 2.12]{VCS}. The induced extension of fields of fractions $\bK_{\overline{\q}^{(t)}}/\bK_{\overline{\p}^{(t)}}$ is Galois of group $G_{\overline{\q}^{(t)}}$ isomorphic to the stabilizer of $\overline{\q}^{(t)}$ under the action of $G$ on $\fX_{K'}^{[t]}(\overline{\p}^{(t)})$ (\textit{cf}. proof of \ref{Generic-Étale} (ii)).
As the normalized integral models have (geometrically) reduced special fibers, the extension $\bK_{\overline{\q}^{(t)}}/\bK_{\overline{\p}^{(t)}}$ has ramification index $1$. Moreover,
its residue extension $\kappa(\overline{\q}^{(t)})/\kappa(\overline{\p}^{(t)})$ is a monogenic extension of discrete valuation fields, with trivial residue extension as $k$ is algebraically closed.
To properly justify this claim, one picks arbitrary specializations of $\overline{\p}^{(t)}$ and $\overline{\q}^{(t)}$ to (geometric) closed points of the special fibers of $\scrC_{K'}^{[t]}$ and $\fX_{K'}^{[t]}$ respectively, above one another, and applies \cite[3.18]{VCS}, followed by \cite[III, \S 6, Prop. 12]{Serre1}.
Let $\bK'_{\overline{\p}^{(t)}}$ be the maximal unramified sub-extension of $\bK_{\overline{\q}^{(t)}}/\bK_{\overline{\p}^{(t)}}$. Then, $\bK_{\overline{\q}^{(t)}}/\bK'_{\overline{\p}^{(t)}}$ is of type (II), i.e. it has ramification index $1$ and a purely inseparable and monogenic residue extension $\kappa(\overline{\q}^{(t)})/\kappa'(\overline{\p})$.
Let $A'_{\overline{\p}^{(t)}}$ be the valuation ring of $\bK'_{\overline{\p}^{(t)}}$ and $G'_{\overline{\q}^{(t)}}$ the Galois group of $\bK_{\overline{\q}^{(t)}}/\bK'_{\overline{\p}^{(t)}}$. Then, the $A'_{\overline{\p}^{(t)}}$-algebra $A_{\overline{\q}^{(t)}}$ is monogenic (so is the $A_{\overline{\p}^{(t)}}$-algebra $A'_{\overline{\p}^{(t)}}$) \cite[III, \S 6, Prop. 12]{Serre1}. (Note that \textit{loc. cit.}, with the same proof, applies here because $\kappa(\overline{\q}^{(t)})/\kappa'(\overline{\p})$, despite being purely inseparable, is monogenic.) Then, just as in the classical setting, we have an \emph{Artin class function}
$a_{G'_{\overline{\q}^{(t)}}}: G'_{\overline{\q}^{(t)}}\to \bZ$ on $G'_{\overline{\q}^{(t)}}$ defined as follows. 
Let $b$ be a generator of $A_{\overline{\q}^{(t)}}$ over $A'_{\overline{\p}^{(t)}}$, denote by $v_{\overline{\q}^{(t)}}: \bK_{\overline{\q}^{(t)}}^{\times}\to \bZ$ its normalized valuation map associated to $A_{\overline{\q}^{(t)}}$ and put
\begin{equation}
	\label{Artin-Intrinseque 2} 
a_{G'_{\overline{\q}^{(t)}}}(\sigma)=- v_{\overline{\q}^{(t)}}(\sigma(b) - b) \quad {\rm if}\quad \sigma\neq 1,
\end{equation}
\begin{equation}
	\label{Artin-Intrinseque 3}
a_{G'_{\overline{\q}^{(t)}}}(1)= - \sum_{\sigma\neq 1} a_{G'_{\overline{\q}^{(t)}}}(\sigma).
\end{equation}
This definition is independent of the chosen generator $a$ because \cite[IV, \S 1, Lemme 1]{Serre1}
\begin{equation}
	\label{Artin-Intrinseque 4}
v_{\overline{\q}^{(t)}}(\sigma(b)-b) =\min\{v_{\overline{\q}^{(t)}}(\sigma(x)-x) ~\lvert ~ x \in A_{\overline{\q}^{(t)}}\}.
\end{equation}
Since, for $\overline{\q}^{(t)}$ varying in $\fX_{K'}^{[t]}(\overline{\p}^{(t)})$, the subgroups $G'_{\overline{\q}^{(t)}}\subset G_{\overline{\q}^{(t)}}$ of $G$, as well as the functions $a_{G'_{\overline{\q}^{(t)}}}$, are all conjugate, we have a well-defined class function
\begin{equation}
	\label{Artin-Intrinseque 5}
\widetilde{a}_{f^{[t]}}=\Ind_{G'_{\overline{\q}^{(t)}}}^G (\lvert G'_{\overline{\q}^{(t)}}\lvert \cdot a_{G'_{\overline{\q}^{(t)}}}).
\end{equation}

\subsection{}\label{FunctionsArtinSwan}
We let $\ell$ be a prime number different form the residue characteristic $p$ of $K$, fix $\overline{\bQ}_{\ell}$ a separable closure of $\bQ_{\ell}$ and denote by $R_{\overline{\bQ}_{\ell}}(G)$ the Grothendieck group of finitely generated $\overline{\bQ}_{\ell}[G]$-modules. We denote by $\cI([r, r']\cap \bQ)$ the set of intervals of $[r', r]$ bounded by rational numbers and \textit{which are not singletons}. For $[t, t']\in \cI([r, r']\cap \bQ)$ (with $t\neq t'$), we denote the restriction of $f$ over $A(t, t')$ by 
\begin{equation}
	\label{FunctionsArtinSwan 1}
f^{[t, t']}: f^{-1}(A(t', t))\to A(t', t).
\end{equation}
Then, by \ref{Normalized Conductors}, we have well-defined associated class functions on $G$
\begin{equation}
	\label{FunctionsArtinSwan 2}
\widetilde{a}_{f^{[t, t']}}^{\alpha}, ~ \widetilde{a}_{f^{[t, t']}}^{\prime\alpha}, ~ \widetilde{\sw}_{f^{[t, t']}}^{\beta} ~{\rm and}~ \widetilde{\sw}_{f^{[t, t']}}^{\prime\beta}. 
\end{equation}
By \ref{DiscVar} (i), the rational number $\langle \widetilde{a}_{f^{[t, t']}}^{\alpha}, \bQ[G/H]\rangle$ (resp. $\langle \widetilde{a}_{f^{[t, t']}}^{\prime\alpha}, \bQ[G/H]\rangle$)
depends only on $t$ (resp. $t'$), not on the interval $[t, t']$. 
In fact, by taking $\tau=(\ox_{\tau}, \q_{\tau})\in S_{f^{[t, t']}}$ (resp. $\tau=(\ox_{\tau'}, \q_{\tau'})\in S'_{f^{[t, t']}}$) and letting $\overline{\q}^{(t)}$ (resp. $\overline{\q}^{(t')}$) be the geometric generic point of $\fX_{s'}^{[t]}$ (resp. $\fX_{s'}^{[t]}$) above $\overline{\p}^{(t)}$ (resp. $\overline{\p}^{(t')}$) in \ref{Artin-Intrinseque} corresponding to $\q_{\tau}$ (resp. $\q_{\tau'}$) \cite[3.21]{VCS}, we see that $G_{\tau}=G_{\overline{\q}^{(t)}}$ and $G_{\tau'}=G_{\overline{\q}^{(t')}}$ \eqref{Nearby Cycles Lutke 5}. Then, by \cite[(6.19.1)]{VCS}, we have
\eqref{Artin-Intrinseque 5}
\begin{equation}
	\label{FunctionsArtinSwan 3}
 a_{G_{\tau}}^{\alpha}\lvert G'_{\overline{\q}^{(t)}}=a_{G'_{\overline{\q}^{(t)}}} \quad {\rm and}\quad a_{G_{\tau'}}^{\alpha} \lvert G'_{\overline{\q}^{(t)}}=a_{G'_{\overline{\q}^{(t')}}};
\end{equation}
\begin{equation}
	\label{FunctionsArtinSwan 4}
\widetilde{a}_{f^{[t]}}=\widetilde{a}_{f^{[t, t']}}^{\alpha} 
\quad {\rm and}\quad 
\widetilde{a}_{f^{[t']}}=\widetilde{a}_{f^{[t, t']}}^{\prime\alpha},
\end{equation}
where the last equation follows from \eqref{FunctionsArtinSwan 3} and \cite[\S 7.2, Rem. 3]{Serre2}.
To $\chi\in R_{\overline{\bQ}_{\ell}}(G)$, we associate the functions
\begin{equation}
	\label{FunctionsArtinSwan 5}
\widetilde{a}_f(\chi, \cdot) : ~ [r, r'] \cap \bQ \to \bQ, \quad t\mapsto \langle \widetilde{a}_{f^{[t]}}, \chi\rangle,
\end{equation}
\begin{equation}
	\label{FunctionsArtinSwan 6}
\widetilde{\sw}_f^{\beta}(\chi, \cdot) : ~ \cI([r, r']\cap \bQ) \to \bQ, \quad [t, t']\mapsto \langle \widetilde{\sw}_{f^{[t, t']}}^{\beta} + \widetilde{\sw}_{f^{[t, t']}}^{\prime\beta}, \chi\rangle.
\end{equation}

\begin{prop}
	\label{Var-Reg}
	Let $H$ be a subgroup of $G$.
\begin{itemize}
\item[(i)] The function $\widetilde{a}_f(\bQ[G/H], \cdot)$ above is continuous and piecewise linear with finitely many slopes which are all integers. 
\item[(ii)] We assume that $X/H$ has \emph{trivial canonical sheaf}. Then, for rational numbers $r\leq t < t'\leq r'$, the difference of the right and left derivatives of $\widetilde{a}_f(\bQ[G/H], \cdot)$ at $t$ and $t'$ respectively is
\begin{equation}
	\label{Var-Reg 1}
\frac{d}{dt}\widetilde{a}_f(\bQ[G/H], t+) - \frac{d}{dt}\widetilde{a}_f(\bQ[G/H], t'-)=\widetilde{\sw}_f^{\beta}(\bQ[G/H], [t, t']).
\end{equation}
\end{itemize}
\end{prop}

\begin{proof}
We know form \cite[(4.17.4), (4.22.2)]{VCS} that
\begin{equation}
	\label{Var-Reg 2}
\partial_{f_H^{[t, t']}}(t)=\partial_{f_H}(t)\quad {\rm and}\quad \partial_{f_H^{[t, t']}}(t')=\partial_{f_H}(t')
\end{equation}
Hence, (i) follows from \eqref{Var-Reg 2}, \eqref{FunctionsArtinSwan 4}, \ref{DiscVar}(i) and \ref{Variation-Lutke}. By \ref{Decomposition SemiStable}, the decomposition hypothesis \eqref{Nearby Cycles Lutke 1} is satisfied by $f_H$ for all but a finite number of radii. As we are computing left and right derivatives (of a piecewise linear function), \eqref{DiscVar 3} applies and, with \eqref{DiscVar}(i), \eqref{FunctionsArtinSwan 4} and \eqref{Var-Reg 2}, yield (ii).
\end{proof}

\begin{thm}
	\label{Variation}
We assume that, for every subgroup $H\subset G$, the quotient $X/H$ has a trivial canonical sheaf.
Let $\chi\in R_{\overline{\bQ}_{\ell}}(G)$. Then, the function $\widetilde{a}_f(\chi, \cdot)$ \eqref{FunctionsArtinSwan 5}
is continuous and piecewise linear with finitely many slopes which are all integers. For rational numbers $r\leq t < t'\leq r'$, the difference of the right and left derivatives of $\widetilde{a}_f(\chi, \cdot)$ at $t$ and $t'$ respectively is
\begin{equation}
	\label{Variation 1}
\frac{d}{dt}\widetilde{a}_f(\chi, t+) - \frac{d}{dt}\widetilde{a}_f(\chi, t'-)=\widetilde{\sw}_f^{\beta}(\chi, [t, t']).
\end{equation}
\end{thm}

\begin{proof}
Mutatis mutandis, the proof is the same as the one for \cite[Theorem 7.16]{VCS} with the key statement \cite[7.13]{VCS} replaced by \ref{Var-Reg} above.
\end{proof}

\subsection{}\label{Coeffs-finis}
Let $\Lambda$ be a finite extension of $\bQ_{\ell}$ inside $\overline{\bQ}_{\ell}$ and $\overline{\Lambda}$ its residue field. By \cite[16.1, Théorème 33]{Serre2}, we have a surjective homomorphism $d_G: R_{\Lambda}(G)\to R_{\overline{\Lambda}}(G)$, the Cartan homomorphism. Let $\overline{\chi}\in R_{\overline{\Lambda}}(G)$ and $\chi\in R_{\Lambda}(G)$ a pre-image of $\overline{\chi}$ by $d_G$. Then, for rational numbers $r\leq t < t'\leq r'$, we put 
\begin{equation}
	\label{Coeffs-finis 1}
\widetilde{a}_f(\overline{\chi}, t)= \widetilde{a}_f(\chi, t) \quad {\rm and}\quad \widetilde{\sw}_f^{\beta}(\overline{\chi}, [t, t'])=\widetilde{\sw}_f^{\beta}(\chi, [t, t']).
\end{equation}
These quantities are independent of the chosen pre-image $\chi$ : the proof is the same as the one given for the analogous statement \cite[7.18]{VCS}.

\begin{cor}
	\label{Variation-Coeffs-finis}
We resume the assumptions of \ref{Variation}.
Let $\overline{\chi} \in R_{\overline{\Lambda}}(G)$. Then, the function
\begin{equation}
\label{Variation-Coeffs-finis 1}
\widetilde{a}_f(\overline{\chi}, \cdot): ~ [r, r']\cap \bQ_{\geq 0}\to \bQ, \quad t\mapsto \widetilde{a}_f(\overline{\chi}, t)
\end{equation}
is continuous and piecewise linear with finitely many slopes which are all integers. For rational numbers $r\leq t < t' \leq r'$, the difference of the right and left derivatives of $\widetilde{a}_f(\overline{\chi}, \cdot)$ at $t$ and $t'$ respectively is
\begin{equation}
	\label{Variation-Coeffs-finis 2}
\frac{d}{dt}\widetilde{a}_f(\overline{\chi}, t+) - \frac{d}{dt}\widetilde{a}_f(\overline{\chi}, t'-)=\widetilde{\sw}_f^{\beta}(\overline{\chi}, [t, t']).
\end{equation}
\end{cor}

\begin{rem}
	\label{Convexity-Expected}
Assume that $f: X\to C$ is étale. Then, so is $f_H: X/H\to C$ for any subgroup $H$ of $G$ and thus the "trivial canonical sheaf" hypothesis of \ref{Variation} is satisfied. Moreover, by \ref{Discriminant-Convex} and \ref{DiscVar}, for any subgroup $H$ of $G$, the function $\widetilde{a}_f(\bQ[G/H], \cdot)$ is convex. However, this convexity statement is not enough to prove the convexity of $\widetilde{a}_f(\chi, \cdot)$ (resp. $\widetilde{a}_f(\overline{\chi}, \cdot)$) for all $\chi\in R_{\overline{\bQ}_{\ell}}(G)$ (resp. $\overline{\chi}\in R_{\overline{\Lambda}}(G)$) (in the routine manner in which \ref{Var-Reg} implies \ref{Variation}). To achieve such a general convexity result, we establish, in what follows, a projective realization of $-\widetilde{\sw}_f^{\beta}-\widetilde{\sw}_f^{\prime\beta}$. 
\end{rem}

\subsection{}\label{DimensionOfFixedPart}
If $C^{\bullet}$ is a complex of $\bQ_{\ell}[G]$-modules and $H$ is a subgroup of $G$, then $(C^{\bullet})^H$ denotes the complex of $\bQ_{\ell}$-vector spaces whose $i$-th term is $(C^{i})^H$. The classical identity
\begin{equation}
	\label{DimensionOfFixedPart 1}
\langle V, \bQ_{\ell}[G/H]\rangle_G = \dim_{\bQ_{\ell}}(V^H),
\end{equation}
where $V$ is any finite dimensional $\bQ_{\ell}$-representation of $G$, extends straightforwardly to the identity
\begin{equation}
	\label{DimensionOfFixedPart 2}
\langle [C^{\bullet}], \bQ_{\ell}[G/H]\rangle_G=\chi((C^{\bullet})^H),
\end{equation}
where $C^{\bullet}$ is any perfect complex of $\bQ_{\ell}[G]$-modules.

\begin{lem}
	\label{NearbyCycles-QuotientMap}
We use the notation of \ref{Nearby Cycles Lutke} and its proof, but write $Y$ instead $Y'$. 
Let $H$ be a subgroup of $G$, denote by $Y_H$ the quotient of $Y$ by $H$, by $\phi: Y\to Y_H$ the quotient morphism and by $\ox'_1, \ldots, \ox'_{N'}$ the images of $\ox_1, \ldots, \ox_N$ by $\phi$. 
\begin{itemize}
\item[(i)] If $j=1, \ldots, N$ and $H\subset \St_{\ox_j}$, we have an isomorphism in the derived category $D(\bQ_\ell)$ of $\bQ_\ell$-vector spaces
\begin{equation}
	\label{NearbyCycles-QuotientMap 1}
R\Psi_{Y_H/S}(\bQ_{\ell})_{\phi(\ox_j)} \xrightarrow{\sim} \left(R\Psi_{Y/S}(\bQ_{\ell})_{\ox_j}\right)^H .
\end{equation}
\item[(ii)] If the subgroup $H$ is abelian, then we have a canonical identification
\begin{equation}
	\label{NearbyCycles-QuotientMap 2}
\bigoplus_{j=1}^{N'} R\Psi_{Y_H/S}(\bQ_{\ell})_{\ox'_j} \cong \left(\bigoplus_{j=1}^N R\Psi_{Y/S}(\bQ_{\ell})_{\ox_j}\right)^H.
\end{equation}
\end{itemize}
\end{lem}

\begin{proof}
We put $\ox'_j=\phi(\ox_j)$ and $\Lambda_n=\bZ/\ell^n\bZ$, for $n\geq 1$.\newline
(i) ~ The functor $\underline{\Gamma}^H: \cF\mapsto \cF^H$ from the category of sheaves of $\Lambda_n[H]$-modules to the category of sheaves of $\Lambda_n$-modules, both on $(Y_H)_{(\ox'_j)}\times_S\overline{\eta}$, is left-exact and derived as
\begin{equation}
	\label{NearbyCycles-QuotientMap 3}
R\underline{\Gamma}^H: D^+((Y_H)_{(\ox'_j)}\times_S\overline{\eta}, \Lambda_n[H])\to D^+((Y_H)_{(\ox'_j)}\times_S\overline{\eta}, \Lambda_n).
\end{equation} 
Similarly, the functor $\Gamma^H: M\mapsto M^H$ on the category of $\Lambda_n[H]$-modules is derived as 
\begin{equation}
	\label{NearbyCycles-QuotientMap 4}
R\Gamma^H : D^+(\Lambda_n[H])\to D^+(\Lambda_n).
\end{equation} 
We note that, as $\Lambda_n$ is a self-injective module (Baer's criterion) and $H$ a finite group, the (left) Noetherian ring $\Lambda_n[H]$ is also self-injective \cite[Cor.9']{EN55}; thus every finitely generated and projective (left) $\Lambda_n[H]$-module is injective \cite[Theorem 15]{EN55}.  Since $R\Psi_{Y/S}(\Lambda_n)_{\ox_j}$ is perfect of amplitude in $[0, 1]$ as a complex of $\Lambda_n[H]$-modules (\ref{Projective-TorsionCoeffs}(i)),  it is quasi-isomorphic to a complex $P_n^{\bullet}$ of finitely generated and projective $\Lambda_n[H]$-modules, concentrated in degrees $0$ and $1$,  and it follows that
\begin{equation}
	\label{NearbyCycles-QuotientMap 5}
(P_n^{\bullet})^H \xrightarrow{\cong} R\Gamma^H(R\Psi_{Y/S}(\Lambda_n)_{\ox_j}).
\end{equation}
Likewise, since $\phi_{\eta\ast}\Lambda_n$ is a sheaf of finitely generated and projective $\Lambda_n[H]$-modules (in fact, locally free, see \eqref{Projective-TorsionCoeffs 3}), the canonical map
\begin{equation}
	\label{NearbyCycles-QuotientMap 6}
\Lambda_n \to R\underline{\Gamma}^H \phi_{\eta\ast}\Lambda_n
\end{equation}
is an isomorphism of sheaves of $\Lambda_n$-modules on $(Y_H)_{(\ox'_j)}\times_S\overline{\eta}$.
Moreover, as the functor sending a $\Lambda_n$-sheaf $\cG$ to $\cG$ endowed with the trivial action is exact and left-adjoint to $\underline{\Gamma}^H$, it follows that $\underline{\Gamma}^H$ preserves injective objects and thus $R(\Gamma\circ\underline{\Gamma}^H)=R\Gamma\circ R\underline{\Gamma}^H$, where $\Gamma$ is the global sections functor on the category of $\Lambda_n$-sheaves on $(Y_H)_{(\ox'_j)}\times_S\overline{\eta}$ \cite[6.3.3]{Fu}. Likewise, since any $H$-equivariant direct image functor preserves injective $H$-sheaves (as right-adjoint to the corresponding exact pull-back functor), so does the global sections functor $\Gamma$ on the category of sheaves of $\Lambda_n[H]$-modules, and thus $R(\Gamma^H\circ\Gamma)=R\Gamma^H\circ\Gamma$. As we clearly have $\Gamma\circ\underline{\Gamma}^H=\Gamma^H\circ \Gamma$, it follows that $R\Gamma\circ R\underline{\Gamma}^H=R\Gamma^H\circ R\Gamma$.  Now, applying the derived functor $R\Gamma: D((Y_H)_{(\ox'_j)}\times_S\overline{\eta}, \Lambda_n)\to D(\Lambda_n)$ to \eqref{NearbyCycles-QuotientMap 6}, yields
\begin{equation}
	\label{NearbyCycles-QuotientMap 7}
R\Psi_{Y_H/S}(\Lambda_n)_{\ox'_j} \xrightarrow{\cong} R\Gamma^H(R\Gamma((Y_H)_{(\ox'_j)}\times_S\overline{\eta}, \phi_{\eta\ast}\Lambda_n))\cong R\Gamma^H(R\Gamma(Y_{(\ox_j)}\times_S\overline{\eta}, \Lambda_n).
\end{equation}
Hence,  with \eqref{NearbyCycles-QuotientMap 5},  we get
\begin{equation}
	\label{NearbyCycles-QuotientMap 8}
R\Psi_{Y_H/S}(\Lambda_n)_{\ox'_j} \xrightarrow{\cong} (P_n^{\bullet})^H.
\end{equation}
For $i=0,  1$,  every $P_n^{i}$ is a finite set,  being finitely generated over the finite ring $\Lambda_n[H]$; hence,  $(P_n^i)^H$ is also finite. Therefore,  the inverse system $((P_n^i)^H)_n$  is Mittag-Leffler,  and thus $R^1\varprojlim_n (P_n^i)^H=0$ \cite[1.15]{Jannsen}.  Then,  with \eqref{NearbyCycles-QuotientMap 8},  we have
\begin{equation}
	\label{NearbyCycles-QuotientMap 9}
\begin{split}
R\Psi_{Y_H/S}(\bQ_{\ell})_{\phi(\ox_j)} & \xrightarrow{\sim} \left( R\varprojlim_n R\Psi_{Y_H/S}(\Lambda_n)_{\phi(\ox_j)}\right)\otimes_{\bZ_\ell}^{\bL}\bQ_\ell \cong (R\varprojlim_n (P_n^{\bullet})^H)\otimes_{\bZ_\ell}\bQ_\ell \\
& \cong (\varprojlim_n (P_n^{\bullet})^H)\otimes_{\bZ_\ell}\bQ_\ell \cong (\varprojlim_n P_n^{\bullet}\otimes_{\bZ_\ell}\bQ_\ell )^H \cong \left(R\Psi_{Y/S}(\bQ_\ell)_{\ox_j}\right)^H,
\end{split}
\end{equation}
where the second-to-last isomorphism is the commutation of the functors $\varprojlim_n$ and $\Gamma^H$ and the last isomorphism is the definition \eqref{Z_ell-Coeffs 3} (the last term is well-defined because $\Gamma^H$ is exact for $\bQ_{\ell}$-coefficients).
\newline

\noindent (ii) ~ Let $j=1, \ldots, N'$. Since $\phi^{-1}(\ox'_j)$ is an orbit of $H$ (under the action of $G$), the direct sum of the complexes $R\Psi_{Y/S}(\bQ_\ell)_{\ox}$, over $\ox\in \phi^{-1}(\ox'_j)$ is stable under $H$. Let $H_j$ be the stabilizer of an element of $\phi^{-1}(\ox'_j)$, which is also the stabilizer of every such element as $H$ is abelian. It follows from \eqref{NearbyCycles-QuotientMap 1} that
\begin{equation}
	\label{NearbyCycles-QuotientMap 10}
\left(\bigoplus_{\ox\in\phi^{-1}(\ox'_j)} R\Psi_{Y/S}(\bQ_\ell)_{\ox}\right)^{H_j} \cong \bigoplus_{\ox\in\phi^{-1}(\ox'_j)} R\Psi_{Y_{H_j}/S}(\bQ_\ell)_{\phi_j(\ox)},
\end{equation}
where $\phi_j$ is the quotient morphism $Y\to Y/H_j=Y_{H_j}$. 
The quotient group $H/H_j$ acts on the right-hand-side sum in \eqref{NearbyCycles-QuotientMap 10} by permuting the summands.
Therefore, we indeed have the identification
\begin{equation}
	\label{NearbyCycles-QuotientMap 11}
\begin{split}
\left(\bigoplus_{j=1}^N R\Psi_{Y/S}(\bQ_{\ell})_{\ox_j}\right)^H & \cong \left(\bigoplus_{j=1}^{N'} \bigoplus_{\ox\in\phi^{-1}(\ox'_j)} R\Psi_{Y/S}(\bQ_{\ell})_{\ox_j}\right)^H \\
& \cong \bigoplus_{j=1}^{N'} \left(\bigoplus_{\ox\in\phi^{-1}(\ox'_j)} R\Psi_{Y_{H_j}/S}(\bQ_\ell)_{\phi_j(\ox)}\right)^{H/H_j}
\cong \bigoplus_{j=1}^{N'} R\Psi_{Y_H/S}(\bQ_{\ell})_{\ox'_j}.
\end{split}
\end{equation}
\end{proof}

\begin{lem}
	\label{Realization-NearbyCycles-Swan-Regular-Pairing}
We use the notation \ref{Nearby Cycles Lutke} and its proof. We assume that the morphism $f: X\to C$ is étale.
For every abelian subgroup $H$ of $G$, we have the identity
\begin{equation}
	\label{Realization-NearbyCycles-Swan-Regular-Pairing 1}
\langle [\bigoplus_{j=1}^N R\Psi_{Y/S}(\bQ_{\ell})_{\ox_j}], \bQ[G/H]\rangle=\langle\widetilde{\sw}_f^{\beta} + \widetilde{\sw}_f^{\prime\beta}, \bQ[G/H]\rangle.
\end{equation}
\end{lem}

\begin{proof}
From \eqref{DiscVar 3}, \eqref{Variation-Lutke 1} and \eqref{Nearby Cycles Lutke 2}, we see that the right-hand-side of \eqref{Realization-NearbyCycles-Swan-Regular-Pairing 1} is the sum $\sum_j (d_{\eta}(B_j/A) - 2\delta(B_j) + \lvert P_s(B_j)\lvert$, where the local rings $A$ and $B_j$ are the appropriate formal étale local rings associated to the normalized integral model of the morphism $X/H\to C$. On the one hand, as $X\to C$ is étale, so is $X/H\to C$ and thus all the $d_{\eta}(B_j/A)$ vanish \eqref{AnnulationDe-d-Eta}. On the other hand, as the Milnor tube $(Y_H)_{(\ox'_j)}\times_S\overline{\eta}$ is connected, we have  
$\dim_{\bQ_{\ell}} R^0_{Y_H/S}(\bQ_{\ell})_{\ox'_j}=1$; moreover, Kato's formula \eqref{Nearby Cycles Lutke 22}, applied to $Y_H\to S$, gives
\begin{equation}
	\label{Realization-NearbyCycles-Swan-Regular-Pairing 2}
\dim_{\bQ_{\ell}} R^1\Psi_{Y_H/S}(\bQ_{\ell})_{\ox'_j}=2\delta(B_j)-\lvert P_s(B_j)\lvert +1.
\end{equation}
Note that \eqref{Realization-NearbyCycles-Swan-Regular-Pairing 2} was formulated with $\bF_\ell$-coefficients on its left-hand-side; but  as $[R\Psi_{Y_H/S}(\bF_{\ell})_{\ox'_j}]$ is identified with $[R\Psi_{Y_H/S}(\bQ_{\ell})_{\ox'_j}]$ through the direct injection $e: P_{\bF_\ell}(\St_{x'_j}) \to R_{\bQ_\ell}(\St_{x'_j})$ \cite[16.1, Théorème 34]{Serre2},  we have $\chi(R\Psi_{Y_H/S}(\bF_{\ell})_{\ox'_j})=\chi(R\Psi_{Y_H/S}(\bQ_{\ell})_{\ox'_j})$,  which implies \eqref{Realization-NearbyCycles-Swan-Regular-Pairing 2}. \newline
It follows that the right-hand-side of \eqref{Realization-NearbyCycles-Swan-Regular-Pairing 1} is also the Euler characteristic of the complex of $\bQ$-vector spaces $\bigoplus_{j=1}^{N'} R\Psi_{Y_H/S}(\bQ_{\ell})_{\ox'_j}$
\begin{equation}
	\label{Realization-NearbyCycles-Swan-Regular-Pairing 3}
\langle\widetilde{\sw}_f^{\beta} + \widetilde{\sw}_f^{\prime\beta}, \bQ[G/H]\rangle= \chi\left(\bigoplus_{j=1}^{N'} R\Psi_{Y_H/S}(\bQ_{\ell})_{\ox'_j}\right).
\end{equation}
Then, \eqref{Realization-NearbyCycles-Swan-Regular-Pairing 1} follows from \eqref{Realization-NearbyCycles-Swan-Regular-Pairing 3}, \eqref{NearbyCycles-QuotientMap 1} and \eqref{DimensionOfFixedPart 2}.
\end{proof}

\begin{thm}
	\label{Realization-NearbyCycles-Swan}
We use the notation \ref{Nearby Cycles Lutke} and its proof.
We have the following identity in $\overline{R}_{\bQ}(G)$
\begin{equation}
	\label{Realization-NearbyCycles-Swan 1}
[\bigoplus_{j=1}^N R\Psi_{Y/S}(\bQ_{\ell})_{\ox_j}]=\widetilde{\sw}_f^{\beta} + \widetilde{\sw}_f^{\prime\beta}. 
\end{equation}
\end{thm}

\begin{proof}
The theorem follows from \ref{Rationality-of-NearbyCycles}, \ref{Realization-NearbyCycles-Swan-Regular-Pairing}, and \cite[\S 13.1, Théorème 30]{Serre2} applied to
\begin{equation}
	\label{Realization-NearbyCycles-Swan 2}
\theta=[\bigoplus_{j=1}^N R\Psi_{Y/S}(\bQ_{\ell})_{\ox_j}] - (\widetilde{\sw}_f^{\beta} + \widetilde{\sw}_f^{\prime\beta}) ~ \in  \overline{R}_{\bQ}(G).
\end{equation}
\end{proof}

\begin{cor}
	\label{Convexity}
Assume that $f: X\to C$ is étale. If $\rho$ is a representation of $G$ with coefficients in $\overline{\bQ}_{\ell}$ $($resp. $\overline{\Lambda})$ and $\chi_{\rho}$ its class in $R_{\overline{\bQ}_{\ell}}(G)$ $($resp. $R_{\overline{\Lambda}}(G))$, then the function $\widetilde{a}_f(\chi_{\rho}, \cdot)$ is convex.
\end{cor}

\begin{proof}
It follows from \ref{Projective-Z_l-Coeffs} and \ref{Realization-NearbyCycles-Swan} that for rational numbers $r\leq t< t'\leq r'$,
the $\bQ$-valued function $-\widetilde{\sw}_{f^{[t, t']}}^{\beta} - \widetilde{\sw}_{f^{[t, t']}}^{\prime\beta}$ is the class of a finitely generated and projective $\bQ_{\ell}[G]$-module. Then, the result follows from \ref{Variation} (resp. \ref{Variation-Coeffs-finis}) and \cite[\S 14.5, b)]{Serre2}.
\end{proof}

\section{Swan conductor of a lisse torsion sheaf on a rigid annulus}
\label{Faisceau-Lisse-Couronne-Fermee}
\subsection{}\label{Faisceau}
Let $\overline{\Lambda}$ be a finite field of characteristic $\ell\neq p$ and $\psi: \bF^{\times}_p\to \overline{\Lambda}^{\times}$ a nontrivial character. 
Let $0\leq r < r'$ be rational numbers and $C=A(r', r)$ the closed sub-annulus of the closed unit disc $D$ defined by $r\leq v_K(\xi)\leq r'$, where $\xi$ is the coordinate of $D$.
Let $\cF$ be a lisse sheaf of $\overline{\Lambda}$-modules on $C$. By \cite[2.10]{deJong}, $\cF$ corresponds to a connected Galois étale cover $f: X\to C$ and a continuous finite dimensional $\overline{\Lambda}$-representation $\rho_{\cF}$ of $G=\Aut(X/C)$. 
We denote by $\chi_{\cF}$ the image of $\rho_{\cF}$ in $R_{\overline{\Lambda}}(G)$.

\subsection{}\label{SW&CC}
We use the notation of \ref{Artin-Intrinseque}.
Let $r \leq t \leq r'$ be a rational number and $\overline{\q}^{(t)}$ an element of $\fX_{K'}^{[t]}(\overline{\p}^{(t)})$. The stabilizer of $\overline{\q}^{(t)}$ under the action of $G$ on $\fX_{K'}^{[t]}(\overline{\p}^{(t)})$ is isomorphic to the Galois group $G_{\overline{\q}^{(t)}}$ of a finite extension of a henselian discrete valuation field which is of type (II) over an unramified sub-extension, with residue extension $\kappa(\overline{\p}^{(t)})\to \kappa(\overline{\q}^{(t)})$.
The ramification theory of Abbes and Saito, applied to the $G_{\overline{\q}^{(t)}}$-representation $M_{\overline{\q}^{(t)}}=\rho_{\cF}\vert G_{\overline{\q}^{(t)}}$ yields a rational number $\sw_{G_{\overline{\q}^{(t)}}}^{\AS}(M_{\overline{\q}^{(t)}})$, the \textit{Swan conductor of} $M_{\overline{\q}^{(t)}}$ \cite[(8.21.2)]{VCS}, and (a power of) a logarithmic differential form $\CC_{\psi}(M_{\overline{\q}^{(t)}})$, the \textit{characteristic cycle of} $M_{\overline{\q}^{(t)}}$ \cite[4.12]{Hu}, which in our setting of type (II) extension, lies in fact in $(\Omega^1_{\kappa(\overline{\p}^{(t)})})^{\otimes m_t}$, where $m_t=\dim_{\overline{\Lambda}}(M_{\overline{\q}^{(t)}}/(M_{\overline{\q}^{(t)}})^{(0)})$, with $(M_{\overline{\q}^{(t)}})^{(0)}$ denoting the part of $M_{\overline{\q}^{(t)}}$ fixed by the tame inertia subgroup of $G_{\overline{\q}^{(t)}}$ \cite[10.5]{Hu}.

The residue field $\kappa(\overline{\p}^{(t)})$ coincides with $\cO_{\scrC_{s'}^{[t]}, \overline{\p}^{(t)}}$. We denote by
$\ord_{0^{(t)}} : \kappa(\overline{\p}^{(t)})^{\times}\to \bZ$ (resp. $\ord_{\infty^{(t)}} : \kappa(\overline{\p}^{(t)})^{\times}\to \bZ$) the normalized valuation map defined by $\ord_{0^{(t)}}(\xi)=1$ (resp. $\ord_{\infty^{(t)}}(\xi)=-1$). We again denote by $\ord_{0^{(t)}} : (\Omega^1_{\kappa(\overline{\p}^{(t)})})^{\otimes m_t} - \{0\}\to \bZ$ the multiplicative extension of $\ord_{0^{(t)}}$ defined by $\ord_{0^{(t)}}(b{\rm d} a)=\ord_{0^{(t)}}(b)$, for any $a, b\in \kappa(\overline{\p}^{(t)})^{\times}$ such that $\ord_{0^{(t)}}(a)=1$; we also denote by $\ord_{\infty^{(t)}}$ the similar extension of $\ord_{\infty^{(t)}}$ to $(\Omega^1_{\kappa(\overline{\p}^{(t)})})^{\otimes m_t} - \{0\}$.
The rational numbers
\begin{align}
\sw_{\AS}(\cF, t)=\sw_{G_{\overline{\q}^{(t)}}}^{\rm AS}(M_{\overline{\q}^{(t)}}),
\label{SW&CC 1}\\
\varphi_s(\cF, t)=-\ord_{0^{(t)}}(\CC_{\psi}(M_{\overline{\q}^{(t)}})) - m_t. \label{SW&CC 2}
\end{align}
are well-defined, independently of the chosen $\overline{\q}^{(t)}\in \fX_{K'}^{[t]}(\overline{\p}^{(t)})$, as can be seen a posteriori from \ref{Abbes-Saito-Kato}.

\begin{prop}
	\label{Abbes-Saito-Kato}
We use the notation of \ref{Coeffs-finis}.  The following statements hold.
\begin{itemize}
\item[(1)] For a rational number $r\leq t < r'$, we have
\begin{equation}
	\label{Abbes-Saito-Kato 1}
\sw_{\AS}(\cF, t)= \widetilde{a}_f(\chi_{\cF}, t).
\end{equation}
\item[(2)] For rational numbers $r\leq t < t' \leq r'$, we have
\begin{equation}
	\label{Abbes-Saito-Kato 2}
\varphi_s(\cF, t) - \varphi_s(\cF, t')= \widetilde{\sw}_f^{\beta}(\chi_{\cF}, [t, t']).
\end{equation}
\end{itemize}
\end{prop}

\begin{proof}
For rational numbers $r\leq t < t' \leq r'$, let $(\ox_{\tau}, \q_{\tau})$ (resp. $(\ox_{\tau'}, \q_{\tau'})$) be an element of $S_{f^{[t, t']}}$ (resp. $S_{f^{[t, t']}}$) (\ref{FunctionsArtinSwan}) and denote by $\overline{\q}^{(t)}$ (resp. $\overline{\q}^{(t')}$) the geometric generic point of the normalized integral model of $f^{-1}(A(t', t))$ corresponding to $\q_{\tau}$ (resp. $\q_{\tau'}$) \eqref{Artin-Intrinseque}. The geometric point $\ox_{\tau}$ is above the origin point $o^{(t)}$ defined by the coordinate $\xi$ with respect to the outer radius $t$ while $\ox_{\tau'}$ is above the point at infinity $\infty^{(t')}$ defined by $\xi$ with respect to the inner radius $t'$. As ${\rm d}\frac{1}{\xi}= -\xi^{-2}{\rm d}\xi$, a straightforward computation shows that, for any $\omega\in (\Omega^1_{\kappa(\overline{\p}^{(t')})})^{\otimes m_{t'}} - \{0\}$, we have
\begin{equation}
	\label{Abbes-Saito-Kato 3}
\ord_{\infty^{(t')}}(\omega)= - \ord_{0^{(t')}}(\omega) - 2m_{t'}.
\end{equation}
We also have $G_{\overline{\q}^{(t)}}=G_{\tau}$ and $G_{\overline{\q}^{(t')}}=G_{\tau'}$. Then, it follows from \cite[8.24]{VCS} and \eqref{FunctionsArtinSwan 4} that
\begin{equation}
	\label{Abbes-Saito-Kato 4}
\sw_{\AS}(\cF, t)=\langle \lvert G_{\tau}\lvert a_{G_{\tau}}^{\alpha}, \chi_{\cF}\lvert G_{\tau}\rangle =\langle \widetilde{a}_{f^{[t, t']}}^{\alpha}, \chi_{\cF}\rangle=\widetilde{a}_f(\chi_{\cF}, t),
\end{equation}
\begin{equation}
	\label{Abbes-Saito-Kato 5}
\varphi_s(\cF, t)=-\ord_{0^{(t)}}(\CC_{\psi}(M_{\overline{\q}^{(t)}})) - m_t =\langle \lvert G_{\tau}\lvert \sw_{G_{\tau}}^{\beta}, \chi_{\cF}\lvert G_{\tau}\rangle =\langle \widetilde{\sw}_{f^{[t, t']}}^{\beta}, \chi_{\cF}\rangle.
\end{equation}
Now applying \eqref{Abbes-Saito-Kato 3} to $\omega=\CC_{\psi}(M_{\overline{\q}^{(t')}}$, we get also from \cite[8.24]{VCS}
\begin{equation}
	\label{Abbes-Saito-Kato 6}
\begin{split}
\varphi_s(\cF, t') &=-\ord_{0^{(t')}}(\CC_{\psi}(M_{\overline{\q}^{(t')}})) - m_{t'} = \ord_{\infty^{(t')}}(\CC_{\psi}(M_{\overline{\q}^{(t')}})) + m_{t'}\\
&= - \langle \lvert G_{\tau'}\lvert \sw_{G_{\tau'}}^{\beta}, \chi_{\cF}\lvert G_{\tau}\rangle = -\langle \widetilde{\sw}_{f^{[t, t']}}^{\prime\beta}, \chi_{\cF}\rangle.
\end{split}
\end{equation}
Putting \eqref{Abbes-Saito-Kato 5} and \eqref{Abbes-Saito-Kato 6} together with \eqref{FunctionsArtinSwan 6} yields \eqref{Abbes-Saito-Kato 2}.
\end{proof}

\begin{cor}
	\label{Concatenation}
We use the notation of \eqref{Coeffs-finis 1}.
For rational numbers $r\leq t < t' < t'' \leq r'$, we have
\begin{equation}
	\label{Concatenation 1}
\widetilde{\sw}_f^{\beta}(\chi_{\cF}, [t, t']) + \widetilde{\sw}_f^{\beta}(\chi_{\cF}, [t', t''])= \widetilde{\sw}_f^{\beta}(\chi_{\cF}, [t, t'']).
\end{equation}
\end{cor}

\begin{thm}
	\label{Variation-Faisecau-Lisse}
The function
$\sw_{\rm AS}(\cF, \cdot): ~ [r, r']\cap \bQ \to \bQ$
is continuous, convex and piecewise linear with finitely many slopes which are all integers. Moreover, for rational numbers $r\leq t < t' \leq r'$, the difference of the right and left derivatives of $\sw_{\rm AS}(\cF, \cdot)$ at $t$ and $t'$ respectively is
\begin{equation}
	\label{Variation-Faisecau-Lisse 1}
\frac{d}{dt}\sw_{\rm AS}(\cF, t+) - \frac{d}{dt}\sw_{\rm AS}(\cF, t'-)=\varphi_s(\cF, t) - \varphi_s(\cF, t').
\end{equation}
\end{thm}

\begin{proof}
Since $f$ is étale, for every subgroup $H$ of $G$, $X/H$ is also étale, hence has trivial canonical sheaf.
Then, the theorem follows from \ref{Variation-Coeffs-finis}, \ref{Convexity} and \ref{Abbes-Saito-Kato}.
\end{proof}

\begin{rem}
	\label{Compatibility}
If $\cF$ is the restriction to the annulus $C$ of a lisse sheaf of $\overline{\Lambda}$-modules on $D$, then the functions $\sw_{\AS}(\cF, \cdot)$ and $\varphi_s(\cF, \cdot)$ coincide on $[r, r']$ with the similarly denoted functions in \cite[1.9, 9.3]{VCS}.
Hence, if $t'$ is a non-critical radius for $f$, we deduce from the main result \cite[1.9]{VCS} that
\begin{equation}
	\label{Compatibility 1}
\frac{d}{dt}\sw_{\rm AS}(\cF, t+) - \frac{d}{dt}\sw_{\rm AS}(\cF, t'-)=\varphi_s(\cF, t) - \varphi_s(\cF, t'),
\end{equation}
which recovers \ref{Variation-Faisecau-Lisse 1}.
\end{rem}


\begin{thebibliography}{99} 
\bibitem[Abb10]{EGR} {\sc A. Abbes}, \'Eléments de Géométrie Rigide, I Construction et étude géométrique des espaces rigides, Birkhäuser, \textit{Progress in Mathematics} {\bf 286} (2010).

\bibitem[AS02]{A.S.1} {\sc A. Abbes, T. Saito}, Ramification of locals fields with imperfect residue fields, Amer. J. Math. {\bf 124} (2002), 879-920.

\bibitem[AS11]{A.S.3} {\sc A. Abbes, T. Saito}, Ramification and cleanliness, Tohoku Math. Journal. Centennial issue, {\bf 63} (2011), no. 4, 775-853.

\bibitem[Bah20]{VCS} {\sc A. Bah}, Variation of the Swan conductor of an an $\bF_{\ell}$-sheaf on a rigid disc, Preprint \url{https://arxiv.org/abs/2010.14843} (2020).

\bibitem[Bos14]{Bosch} {\sc S. Bosch}, Lectures on Formal and Rigid Geometry, \textit{Lectures Notes in Mathematics}, Springer {\bf 2105} (2014).

\bibitem[BGR84]{BGR} {\sc S. Bosch, U. Güntzer, R. Remmert}, Non-Archimedean analysis, Springer-Verlag, {\bf 261} (1984).

\bibitem[Bou06]{Bourbaki1} {\sc N. Bourbaki}, Algèbre Commutative, Springer-Verlag, (2006).

\bibitem[EN55]{EN55} {\sc S. Eilenberg, T. Nakayama}, On the dimension of modules and Algebras, II: (Frobenius algebras and quasi-Frobenius rings), Nagoya Math. Journal {\bf 9} (1955), 1-16.

\bibitem[Eis94]{Eisenbud} {\sc D. Eisenbud}, Commutative Algebra with a view toward Algebraic Geometry, Springer-Verlag, (1994).

\bibitem[deJ95]{deJong} {\sc A. J. de Jong}, {\'E}tale fundamental groups of non-Archimedean analytic spaces, Compositio math. {\bf 97} (1995), no. 1-2, 89-118.

\bibitem[Fu11]{Fu}{\sc L. Fu}, \'Etale cohomology theory, World Scientific, Nankai Tracts in Mathematics, {\bf 169} (2011).

\bibitem[EGA III]{EGA.III} {\sc A. Grothendieck, J.A. Dieudonné}, \'Eléments de Géométrie Algébrique, III \'Etude cohomologique des faisceaux cohérents, Pub. Math. IH\'ES {\bf 11} (1961), {\bf 17} (1963).

\bibitem[EGA IV]{EGA.IV} {\sc A. Grothendieck, J.A. Dieudonné}, \'Eléments de Géométrie Algébrique, IV \'Etude locale des schémas et
des morphismes de schémas, Pub. Math. IH\'ES {\bf 20} (1964), {\bf 24} (1965), {\bf 28} (1966), {\bf 32} (1967).

\bibitem[Hu15]{Hu} {\sc H. Hu}, Ramification and nearby cycles for $\ell$-adic sheaves on relative curves, Tohoku Math. Journal {\bf 67}, (2015), no. 2, 153-194.

\bibitem[Ill94]{Illusie1} {\sc L. Illusie}, Autour du théorème de monodromie locale, In \textit{Périodes $p$-adiques}, Astérisque {\bf 223} (1994).

\bibitem[Jan88]{Jannsen} {\sc U. Jannsen},  Continuous étale cohomology,  Math.  Ann.  {\bf 280},  (1988),  207-245.

\bibitem[Kat87]{K1} {\sc K. Kato}, Vanishing cycles, ramification of valuations and class field theory, Duke Math. J. {\bf 55}, (1987), no. 3, 629-659.

\bibitem[Lip69]{Lipman} {\sc J. Lipman}, Rational singularities with applications to algebraic surfaces and unique factorization,  Pub. Math. IH\'ES {\bf 36} (1969), 195-279.

\bibitem[Lüt93]{Lutke} {\sc W. Lütkebohmert}, Riemann’s existence problem for a p-adic field, Invent. Math. {\bf 111} (1993), 309-330.

\bibitem[Ram05]{Ramero} {\sc L. Ramero}, Local monodromy in non-archimedean analytic geometry, Pub. Math. IH\'ES {\bf 102} (2005), 167-280.

\bibitem[Ray94]{Raynaud-Abh} {\sc M. Raynaud}, Revêtements de la droite affine en caractéristique $p> 0$ et conjecture d'Abhyankar, Invent. Math. {\bf 116} (1994), no. 1, 425-462.

\bibitem[Sai87]{Saito87} {\sc T. Saito}, Vanishing cycles and geometry of curves over a discrete valuation ring, Amer. J. Math. {\bf 109} (1987), no. 6, 1043-1085.

\bibitem[Sai]{SaitoBook} {\sc T. Saito}, Ramification theory and Vanishing cycles, Provisional title of a book in preparation.

\bibitem[Ser68]{Serre1} {\sc J.-P. Serre}, Corps Locaux, Hermann, Paris (1968).

\bibitem[Ser98]{Serre2} {\sc J.-P. Serre}, Representations linéaires des groupes finis, Hermann, Paris (1998).

\bibitem[SGA 4]{SGA4} {\sc M. Artin, A. Grothendieck, J.-L. Verdier}, \textit{Séminaire de Géométrie Algébrique du Bois-Marie}, Théorie des topos et cohomologie étale des schémas (SGA 4), Springer-Verlag, Lect. Notes in Math. {\bf 269} (1972), {\bf 270} (1972), {\bf 305} (1973).

\bibitem[SGA 5]{SGA5} {A. Grothendieck, L. Illusie}, \textit{Séminaire de Géométrie Algébrique du Bois-Marie}, Cohomologie $\ell$-adique et Fonctions $L$ (SGA 5), Springer-Verlag, Lect. Notes in Math. {\bf 589} (1977).

\bibitem[SGA 6]{SGA6} {\sc P. Berthelot, A. Grothendieck, L. Illusie}, \textit{Séminaire de Géométrie Algébrique du Bois-Marie}, Théorie des Intersections et Théorème de Rieman-Roch (SGA 6), Springer-Verlag, Lect. Notes in Math. {\bf 225} (1971).

\bibitem[SGA 7]{SGA7} {\sc P. Deligne, N. M. Katz}, \textit{Séminaire de Géométrie Algébrique du Bois-Marie}, Groupes de Monodromie en Géométrie Algébrique (SGA 7), Springer-Verlag, Lect. Notes in Math. {\bf 288} (1972), {\bf 340} (1973).

\bibitem[SGA 4$\frac{1}{2}$]{SGA4demi} {\sc P. Deligne}, \textit{Séminaire de Géométrie Algébrique du Bois-Marie}, Cohomologie étale (SGA 4$\frac{1}{2}$), Springer-Verlag, Lect. Notes in Math. {\bf 569} (1977).

\end{thebibliography}
\end{document}